\documentclass[leqno,11pt]{amsart}

\usepackage{geometry}

\usepackage[all,cmtip]{xy}
\usepackage{tikz-cd}

\usepackage[all]{xy} 

\usetikzlibrary{patterns}

\usepackage{mathtools}
\usepackage{amsmath, amssymb, amsfonts, latexsym, mdwlist, amsthm}
\usepackage{subfig}
\usepackage{graphicx}
\usepackage{wrapfig}

\usepackage[bookmarks, colorlinks, breaklinks, pdftitle={},
pdfauthor={}]{hyperref}
\hypersetup{linkcolor=blue,citecolor=blue,filecolor=black,urlcolor=blue}

\usepackage{tikz}
\usetikzlibrary{calc,trees,positioning,arrows,chains,shapes.geometric,%
    decorations.pathreplacing,decorations.pathmorphing,shapes,%
    matrix,shapes.symbols}

\tikzset{
>=stealth',
  punktchain/.style={
    rectangle,
    rounded corners,
    draw=black, thick,
    minimum height=3em,
    text centered,
    on chain},
  line/.style={draw, thick, <-},
  element/.style={
    tape,
    top color=white,
    bottom color=blue!50!black!60!,
    minimum width=8em,
    draw=blue!40!black!90, very thick,
    text width=10em,
    minimum height=3.5em,
    text centered,
    on chain},
  every join/.style={->, thick,shorten >=1pt},
  decoration={brace},
  tuborg/.style={decorate},
  tubnode/.style={midway, right=2pt},
}

\usepackage{paralist}
\setdefaultenum{(a)}{(i)}{}{}
\usepackage{enumitem} 
\usepackage{graphicx}

\usetikzlibrary{patterns}

\renewcommand\_{^{}_}
\renewcommand\;{\hspace{.6pt}}
\newcommand\To{\longrightarrow}
\newcommand\into{\hookrightarrow}
\newcommand{\Into}{\ensuremath{\lhook\joinrel\relbar\joinrel\rightarrow}}
\newcommand\Onto{\longrightarrow\hspace{-5.5mm}\longrightarrow}
\newcommand\Mapsto{\ensuremath{\shortmid\joinrel\relbar\joinrel\rightarrow}}
\newcommand{\rt}[1]{\xrightarrow{\ #1\ }}
\newcommand\PP{\mathbb P}
\newcommand\C{\mathbb C}
\newcommand\Q{\mathbb Q}
\newcommand\R{\mathbb R}
\newcommand\N{\mathbb N}
\newcommand\Z{\mathbb Z}
\newcommand\J{\mathsf J\;}

\newcommand\cA{\mathcal A}
\newcommand\cD{\mathcal D}
\newcommand\cH{\mathcal H}
\newcommand\cO{\mathcal O}

\newcommand\cF{\mathcal F}

\newcommand\cM{\mathcal M}

\newcommand\cT{\mathcal T}
\newcommand\ch{\operatorname{ch}}
\newcommand\Aut{\operatorname{Aut}}

\newcommand\Hom{\operatorname{Hom}}
\newcommand\Pic{\operatorname{Pic}}
\newcommand\Ext{\operatorname{Ext}}
\newcommand\rk{\operatorname{rank}}

\newcommand\cok{\operatorname{coker}}
\newcommand\js{\operatorname{JS}}

\newcommand\vi{v}

\newcommand\im{\operatorname{im}}

\newcommand\onto{\to\hspace{-3mm}\to}
\newcommand\so{\operatorname{\Longrightarrow}}
\renewcommand\({\big(}
\renewcommand\){\big)}
\renewcommand\={\ =\ }
\newcommand\arXiv[1]{\href{http://arxiv.org/abs/#1}{arXiv:#1}}
\newcommand\mathAG[1]{\href{http://arxiv.org/abs/math/#1}{math.AG/#1}}

\newcommand\beq[1]{\begin{equation}\label{#1}}
\newcommand\eeq{\end{equation}}
\newcommand\beqa{\begin{eqnarray*}}
\newcommand\eeqa{\end{eqnarray*}}

\makeatletter
\newtheorem*{rep@theorem}{\rep@title}
\newcommand{\newreptheorem}[2]{%
\newenvironment{rep#1}[1]{%
 \def\rep@title{#2 \ref{##1}}%
 \begin{rep@theorem}}%
 {\end{rep@theorem}}}
\makeatother

\newtheorem{Thm}{Theorem}[section]
\newreptheorem{Thm}{Theorem}
\newtheorem{Thm*}{Theorem}
\newtheorem{Prop}[Thm]{Proposition}

\newtheorem{Lem}[Thm]{Lemma}

\newtheorem{Cor}[Thm]{Corollary}
\newreptheorem{Cor}{Corollary}
\newtheorem{Con}[Thm]{Conjecture}

\newtheorem{thm-int}{Theorem}

\theoremstyle{definition}
\newtheorem{Def-s}[Thm]{Definition}
\newtheorem{Def}[Thm]{Definition}
\newtheorem{Rem}[Thm]{Remark}

\newcommand{\ignore}[1]{}

\begin{document}

\title{\ \vspace{-2cm} \\ Rank $r$ DT theory from rank $0$\vspace{-3mm}}
\author{S. Feyzbakhsh and R. P. Thomas\vspace{-8mm}}
\maketitle
\begin{abstract}
Fix a Calabi-Yau 3-fold $X$ satisfying the Bogomolov-Gieseker conjecture of Bayer-Macr\`i-Toda, such as the quintic 3-fold. We express Joyce’s generalised DT invariants counting Gieseker semistable sheaves of any rank $r\ge1$ on $X$ in terms of those counting sheaves of rank 0 and pure dimension 2.

The basic technique is to reduce the ranks of sheaves by replacing them by the cokernels of their Mochizuki/Joyce-Song pairs and then use wall crossing to handle their stability.
\end{abstract}
\bigskip


Let $X$ be a Calabi-Yau 3-fold satisfying the Bogomolov-Gieseker conjecture of \cite{BMT}.\footnote{In fact the significant weakening of the conjecture described in Section \ref{BGsec} is sufficient for our purposes.} Let $\J(v)\in\Q$ be Joyce-Song's generalised DT invariant \cite{JS} counting Gieseker semistable sheaves of numerical K-theory class $v$ on $X$.

We show these DT invariants are determined entirely by \emph{rank zero} invariants counting Gieseker semistable sheaves of pure dimension 2. These latter invariants are predicted by physicists' S-duality conjectures to be governed by vector valued mock modular forms. See \cite[Section 5]{FT2} for a discussion, and a possible route to an explanation by expressing these rank 0 dimension 2 invariants in terms of
\begin{itemize}
\item intersection numbers with Noether-Lefschetz loci and
\item the Dedekind $\eta$-function governing Hilbert schemes of points on surfaces.
\end{itemize}

\begin{Thm*}\label{1} Let $(X,\cO_X(1))$ be a Calabi-Yau 3-fold satisfying the weak form \emph{\ref{wBG}} of the Bogomolov-Gieseker inequality of
\cite{BMT}. Then for fixed $v$ of rank $\ge1$,
\beq{formu}
\J(v)\=F\(\J(\alpha_1),\J(\alpha_2),\dots\)
\eeq
is a universal polynomial in invariants $\J(\alpha_i)$, with all $\alpha_i$ of rank 0 and dimension 2.
\end{Thm*}

The coefficients of $F$ depend only on $H^*(X,\Q)$ as a graded ring with pairing, $\ch(v)$, the Chern classes of $X$, and the class $H:=c_1(\cO_X(1))$ used to define Gieseker stability. There are countably many terms in the formula \eqref{formu} but only finitely many are nonzero.

In a sequel \cite{FT3} we take the technique described below one step further to express all rank DT invariants in terms of invariants counting rank $-1$ complexes. Dualising and shifting by $[1]$, these become rank 1 stable pairs. Combined with the DT/PT wall crossing formula this finally leads to a variant of Theorem \ref{1} expressing all DT invariants in terms of rank 1 invariants --- a 6 dimensional analogue of the conjectured equivalence between Donaldson theory in any rank and Seiberg-Witten theory in rank 1.

\subsection*{Joyce-Song pairs} 
Our route to Theorem \ref{1} goes via \emph{Joyce-Song stable pairs} $(F,s)$ on a fixed smooth complex projective threefold $(X,\cO_X(1))$. These consist of a fixed $n\gg0$,
\begin{itemize}
\item a rank $r\ge1$ semistable sheaf $F$ of fixed class $v$, and
\item $s\in H^0(F(n))$ which factors through no (semi-)destabilising subsheaf of $F$.
\end{itemize}
For us ``semistable" will refer to a specific weak stability condition of \cite{BMT, BMS} (in contrast to Joyce-Song, who use Gieseker stability). This is enough to ensure that $s$ is an \emph{injective} map of sheaves, so it makes sense to consider its cokernel $E$,
\beq{JSs}
0\To\cO(-n)\rt{s}F\To E\To0,
\eeq
of class $v_n:=v-[\cO(-n)]$. Crucially this has smaller rank $r-1$. We show $E$ is \emph{stable} just above what we call the \emph{Joyce-Song wall} $\ell_{\js}$ \eqref{ljsdef} in the space of weak stability conditions. This is the wall along which the slopes of $\cO(-n)[1],\,F$ and $E$ coincide, so the exact triangle $F\to E\to\cO(-n)[1]$ induced by \eqref{JSs} destabilises $E$ on and below $\ell_{\js}$. The same argument applies to $E\otimes T$ for $T$ any line bundle with \emph{torsion first Chern class},
i.e. for $T$ in
$$
\Pic\_0(X)\ :=\ \big\{T\in\Pic(X)\ \colon\ c_1(T)\,=\,0\,\in\,H^2(X,\Q)\big\}.
$$
Moreover this describes \emph{all} sheaves in an open and closed subset of the moduli space of semistable sheaves. (When $r\le2$ it is the whole moduli space.) Thus the moduli space of rank $r$ Joyce-Song pairs is actually a moduli space of (semi)stable sheaves of rank $r-1$.


\subsection*{Wall crossing} When $X$ is Calabi-Yau this allows us to employ the Joyce-Song wall crossing formula to relate invariants counting the rank $r$ sheaves $F$ to those counting the rank $r-1$ sheaves $E$. 
This gives a universal formula like \eqref{formu}, but for invariants defined by a weak stability condition.

Next we move through the space of weak stability conditions to the large volume limit. The sheaves $F$ in class $v$ remain semistable but the sheaves $E$ of class $v_n$ undergo wall crossing. We prove this only ever involves $E$ being destabilised by other \emph{sheaves} --- rather than complexes of sheaves --- which therefore also have rank $\le r-1$. We can also control their Chern characters and thus set up an induction on rank. The base case $r=1$ was proved in \cite{FT2}, so eventually we get a formula \eqref{formu} for invariants counting \emph{tilt semistable} objects. A final wall crossing to Gieseker stability then proves Theorem \ref{1}.\medskip

Our main technique for finding walls of instability is the Bogomolov-Gieseker conjecture of \cite{BMT, BMS}. In fact we only need the much weaker condition \ref{wBG} of Section \ref{BGsec}. This has now been proved for some Calabi-Yau 3-folds \cite{BMS, Ko20, Li, Liu, MP}. 
\smallskip

\subsection*{Acknowledgements}
We thank Arend Bayer, Tom Bridgeland, Daniel Huybrechts, Adrian Langer, Davesh Maulik, Rahul Pandharipande, J\o rgen Rennemo and Yukinobu Toda for help, encouragement and useful conversations. We thank Dominic Joyce for help, discouragement and useful conversations. In \cite{FT2} we explained our intellectual debt to Toda's work \cite{TodaBG}; here we would also like to record our admiration for the extraordinary papers \cite{BMT, BMS, JS}, which made everything possible. We are grateful to two fantastic anonymous referees for really useful reports.

We acknowledge the support of an EPSRC postdoctoral fellowship EP/T018658/1, an EPSRC grant EP/R013349/1 and a Royal Society research professorship.

\subsection*{Plan} We work on a smooth complex projective threefold $X$. In Section \ref{weak} we review the weak stability conditions and Bogomolov-Gieseker inequality of \cite{BMT, BMS}. Section \ref{shvs} proves that for such stability conditions, semistable objects $E\in\cD(X)$ of class $v_n$ are always \emph{sheaves}. Moreover in Section \ref{subshvs} we show that their destabilising objects are also sheaves, with one exception --- on the Joyce-Song wall, $E$ can be destabilised by an exact triangle ${F\to E\to T(-n)[1]}$ expressing it in terms of a Joyce-Song pair \eqref{JSs} up to tensoring by some $T\in\Pic\_0(X)$.

From Section \ref{DT3} we assume $X$ is Calabi-Yau. We review the invariants counting its (weak) semistable sheaves and their wall crossing formulae in Section \ref{DT3}. Section \ref{r} gives the proof of Theorem \ref{1}. It hides the role of Joyce-Song stable pairs so we explain their relevance in Appendix \ref{JSapp}, and stronger results in rank 2 in Appendix \ref{2}. Neither is necessary for the rest of the paper. Appendix \ref{TomB} uses work of Piyaratne-Toda to check that  moduli stacks of weak semistable objects are algebraic of finite type.

\subsection*{Outlook}
Up until Section \ref{DT3} all of our results apply to any 3-fold satisfying the Bogomolov-Gieseker inequality. From Section \ref{DT3} we restrict attention to Calabi-Yau 3-folds only so that we are able to apply the wall crossing formula of \cite{JS}. 
Using Joyce's new wall crossing formula for Fano 3-folds \cite[Theorem 7.69]{Jo4} instead we would expect our results to prove a version of Theorem \ref{1} with insertions for Fano 3-folds.

On Calabi-Yau 3-folds with $h^1(\cO_X)>0$ the locally free action of Jac$\;(X)$ on moduli spaces of rank $r>0$ sheaves forces $\J(v)=0$. So from Section \ref{DT3} we assume $h^1(\cO_X)=0$ without loss of generality. But it would be interesting to extend our results to counts of sheaves of \emph{fixed determinant} when $h^1(\cO_X)>0$; this should not be too hard because in this case the wall crossing formula simplifies --- see \cite[Proposition 4.4]{TT2} or \cite[Theorem 4.9]{OPT} for instance.

Kontsevich-Soibelman's refinements of $\J(v)$ \cite{KS1, KS2} are probably rigorously defined  now that their integrality conjecture has been proved \cite{BenSven} and orientation data has been shown to exist on Calabi-Yau 3-folds
\cite{JU}. Replacing the wall crossing formula of \cite{JS} with that of \cite{KS1} should prove an analogue of Theorem \ref{1} for these invariants too. As Jan Manschot pointed out, the refined invariants should determine the fixed determinant invariants when $h^1(\cO_X)>0$.

Finally we note that while the universal formulae of Theorem \ref{1} are not explicit, the wall crossing formulae in this paper give explicit formulae when restricted to specific classes and specific Calabi-Yau threefolds; see \cite{F22} for examples.

\setcounter{tocdepth}{1}
\tableofcontents
\vspace{-1cm}

\section{Weak stability conditions}\label{weak}
Let $(X, \cO(1))$ be a smooth polarised complex projective threefold with bounded derived category of coherent sheaves $\cD(X)$ and Grothendieck group $K(\mathrm{Coh}(X))$. Dividing by the kernel of the Mukai pairing gives the numerical Grothendieck group
\beq{Kdef}
K(X)\ :=\ \frac{K(\mathrm{Coh}(X))}{\ker\chi(\ \ ,\ \ )}\,.
\eeq
Notice $K(X)$ is torsion-free and isomorphic to its image in $H^*(X,\Q)$ under the Chern character. We will mainly work with $K(X)$ and its quotient
\beq{K_Hdef}
K_H(X)\ :=\ \frac{K(X)}{\ker\ch\_H}\ \cong\ \im\;(\ch\_H)\ \subseteq\ \Q^3,
\eeq
where $H:=c_1(\cO(1))$ and $\ch_H\colon K(X)\to\Q^3$ is the map
\beq{chH}
            \ch\_H(E)\=\(\!\ch_0(E)H^3,\,\ch_1(E).\;H^2,\,\ch_2(E).\;H\).
\eeq
Thus $K_H(X)$ sees $\ch_i$ only through its $H$-degree $\ch_i\!.\;H^{3-i}$ for $i\le2$, and is blind to $\ch_3$.

Gieseker and slope (semi)stability of sheaves will always be defined by $H$.
Scaling the usual definition by $H^3$, we define the $\mu\_H$-slope of a coherent sheaf $E$ to be
$$
\mu\_H(E)\ :=\ \left\{\!\!\!\begin{array}{cc} \frac{\ch_1(E).H^2}{\ch_0(E)H^3} & \text{if }\ch_0(E)\ne0, \\
+\infty & \text{if }\ch_0(E)=0. \end{array}\right.
$$
Associated to this slope every sheaf $E$ has a Harder-Narasimhan filtration. Its graded pieces have slopes whose maximum we denote by $\mu_H^+(E)$ and minimum by $\mu_H^-(E)$.

For any $b \in \mathbb{R}$, let $\cA_{\;b}\subset\cD(X)$ denote the abelian category of complexes
	\begin{equation}\label{Abdef}
	\mathcal{A}_{\;b}\ =\ \big\{E^{-1} \xrightarrow{\,d\,} E^0 \ \colon\ \mu_H^{+}(\ker d) \leq b \,,\  \mu_H^{-}(\cok d) > b \big\}. 
	\end{equation}
In particular each $E\in\cA_{\;b}$ satisfies
\beq{65}
\ch_1^{bH}(E).H^2\=\ch_1(E).H^2-bH^3\ch_0(E)\ \ge\ 0,
\eeq
where $\ch^{bH}(E):=\ch(E)e^{-bH}$.
By \cite[Lemma 6.1]{Br} $\cA_{\;b}$ is the heart of a t-structure on $\cD(X)$. For any $w>\frac12b^2$, we have on $\cA_{\;b}$ the slope function
\begin{equation}\label{noo}
\nu\_{b,w}(E)\ =\ \left\{\!\!\!\begin{array}{cc} \frac{\ch_2(E).H - w\ch_0(E)H^3}{\ch_1^{bH}(E).H^2}
 & \text{if }\ch_1^{bH}(E).H^2\ne0, \\
+\infty & \text{if }\ch_1^{bH}(E).H^2=0. \end{array}\right.
\end{equation}
Note that if $\ch_H(E) = 0$ the slope $\nu\_{b,w}(E)$ is defined by \eqref{noo} to be $+\infty$. 
By \cite{BMT}\footnote{We use notation from \cite{FT2}; in particular the rescaling \cite[Equation 6]{FT2} of \cite{BMT}'s slope function.} there exists a Harder-Narasimhan filtration on $\cA_{\;b}$ for $\nu\_{b,w}$, thus defining a \emph{weak stability condition} on $\cD(X)$.

\begin{Def}
Fix $w>\frac12b^2$. Given an injection $0 \neq F\into E$ in $\cA_{\;b}$ we call $F$ a \emph{destabilising subobject of} $E$ if and only if
\beq{seesaw}
\nu\_{b,w}(F)\ \ge\ \nu\_{b,w}(E/F),
\eeq
and \emph{strictly destabilising} if $>$ holds. 
We say $E\in\cD(X)$ is $\nu\_{b,w}$-(semi)stable if and only if
\begin{itemize}
\item $E[k]\in\cA_{\;b}$ for some $k\in\Z$, and
\item 
$E[k]$ contains no (strictly) destabilising subobjects.
\end{itemize}
\end{Def}
It is important to note we cannot replace \eqref{seesaw} by $\nu\_{b,w}(F)\ge\nu\_{b,w}(E)$; this is implied by \eqref{seesaw} but does \emph{not} imply it. So for instance the sequence $I_p\into\cO_X\onto\cO_p$, for $p$ a point of $X$, does \emph{not} destabilise $\cO_X$ even though $\nu\_{b,w}(I_p)=\nu\_{b,w}(\cO_X)$.

\begin{Rem}\label{heart} 
Given $(b,w) \in \mathbb{R}^2$ with $w> \frac12b^2$, the argument in \cite[Proposition 5.3]{Br.stbaility} describes $\cA_{\;b}$. It is generated by the $\nu\_{b,w}$-stable two-term complexes $E = \{E^{-1} \to E^0\}$ in $\cD(X)$ satisfying the following conditions on the denominator and numerator of $\nu\_{b,w}$ \eqref{noo}:
\begin{enumerate}
    \item $\ch_1^{bH}(E).H^2 \geq 0$, and
    \item $\ch_2(E).H - w\ch_0(E)H^3 \geq 0$ if $\ch_1^{bH}(E).H^2 = 0$. 
\end{enumerate}
That is, $\cA_{\;b}$ is the extension-closure of the set of these complexes. 
\end{Rem}

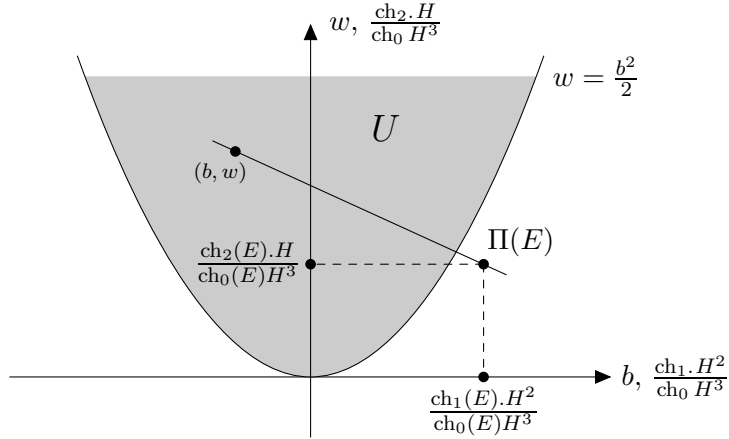
\begin{figure}[h]
	\begin{centering}
		\definecolor{zzttqq}{rgb}{0.27,0.27,0.27}
		\definecolor{qqqqff}{rgb}{0.33,0.33,0.33}
		\definecolor{uququq}{rgb}{0.25,0.25,0.25}
		\definecolor{xdxdff}{rgb}{0.66,0.66,0.66}
		
		\begin{tikzpicture}[line cap=round,line join=round,>=triangle 45,x=1.0cm,y=1.0cm]
		
		\draw[->,color=black] (-4,0) -- (4,0);
		\draw  (4, 0) node [right ] {$b,\,\frac{\ch_1\!.\;H^2}{\ch_0H^3}$};


		\fill [fill=gray!40!white] (0,0) parabola (3,4) parabola [bend at end] (-3,4) parabola [bend at end] (0,0);
		
		\draw  (0,0) parabola (3.1,4.27); 
		\draw  (0,0) parabola (-3.1,4.27); 
		\draw  (3.8 , 3.6) node [above] {$w= \frac{b^2}{2}$};
		
		

		\draw[->,color=black] (0,-.8) -- (0,4.7);
		\draw  (1, 4.3) node [above ] {$w,\,\frac{\ch_2\!.\;H}{\ch_0H^3}$};

		
		\draw [dashed, color=black] (2.3,1.5) -- (2.3,0);
		\draw [dashed, color=black] (2.3, 1.5) -- (0, 1.5);
		\draw [color=black] (2.6, 1.36) -- (-1.3, 3.14);
		
		\draw  (2.8, 1.8) node {$\Pi(E)$};
		\draw  (1, 3) node [above] {\Large{$U$}};
		\draw  (0, 1.5) node [left] {$\frac{\ch_2(E).H}{\ch_0(E)H^3}$};
		\draw  (2.3 , 0) node [below] {$\frac{\ch_1(E).H^2}{\ch_0(E)H^3}$};
		\begin{scriptsize}
		\fill (0, 1.5) circle (2pt);
		\fill (2.3,0) circle (2pt);
		\fill (2.3,1.5) circle (2pt);
		\fill (-1,3) circle (2pt);
		\draw  (-1.2, 2.96) node [below] {$(b,w)$};
		
		\end{scriptsize}
		
		\end{tikzpicture}
		
		\caption{$(b,w)$-plane and the projection $\Pi(E)$ when $\ch_0(E)>0$}
		
		\label{projetcion}
		
	\end{centering}
\end{figure}

By \cite[Theorem 3.5]{BMS} any $\nu\_{b,w}$-semistable object $E \in \mathcal{D}(X)$ satisfies
\beq{BOG}
    \Delta_H(E)\ :=\ \left(\ch_1(E).H^2\right)^2 -2(\ch_2(E).H)\ch_0(E)H^3\ \geq\ 0.
\eeq
Therefore, if we plot the $(b,w)$-plane simultaneously with the image of the projection map
\begin{eqnarray*}
	\Pi\colon\ K(X) \smallsetminus \big\{E \colon \ch_0(E) = 0\big\}\! &\longrightarrow& \R^2, \\
	E &\ensuremath{\shortmid\joinrel\relbar\joinrel\rightarrow}& \!\!\bigg(\frac{\ch_1(E).H^2}{\ch_0(E)H^3}\,,\, \frac{\ch_2(E).H}{\ch_0(E)H^3}\bigg),
\end{eqnarray*}
as in Figure \ref{projetcion}, then $\nu\_{b,w}$-semistable objects $E$ lie outside the
 open set
\begin{equation}\label{Udef}
U\ :=\ \Big\{(b,w) \in \mathbb{R}^2 \colon w > \tfrac12b^2  \Big\}
\end{equation}
while $(b,w)$ lies inside $U$. By \eqref{65} they lie to the right of (or on) the vertical line through $(b,w)$ if $\ch_0(E)>0$, to the left if $\ch_0(E)<0$, and at infinity if $\ch_0(E)=0$. \emph{The slope $\nu\_{b,w}(E)$ of $E$ is the gradient of the line connecting $(b,w)$ to $\Pi(E)$.}

Objects $E\in\cD(X)$ gives the space of weak stability conditions a wall and chamber structure by \cite[Proposition 12.5]{BMS}, as rephrased in \cite[Proposition 4.1]{FT1} for instance.


\begin{Prop}[\textbf{Wall and chamber structure}]\label{prop. locally finite set of walls}
	Fix $v\in K(X)$ such that $\ch_H(v)\ne0$ and $\Delta_H(v)\ge0$. There exists a set of lines $\{\ell_i\}_{i \in I}$ in $\mathbb{R}^2$ such that the segments $\ell_i\cap U$ (called ``\emph{walls}") are locally finite and satisfy 
	\begin{itemize*}
	    \item[\emph{(}a\emph{)}] If $\ch_0(v)\ne0$ then all lines $\ell_i$ pass through $\Pi(v)$.
	    \item[\emph{(}b\emph{)}] If $\ch_0(v)=0$ then all lines $\ell_i$ are parallel of slope $\frac{\ch_2(v).H}{\ch_1(v).H^2}$.
	   		\item[\emph{(}c\emph{)}] The $\nu\_{b,w}$-(semi)stability of any $E\in\cD(X)$ of class $v$ is unchanged as $(b,w)$ varies within any connected component (called a ``\emph{chamber}") of $U \smallsetminus \bigcup_{i \in I}\ell_i$.
		\item[\emph{(}d\emph{)}] For any wall $\ell_i\cap U$ there is a map $f\colon F\to E$ in $\cD(X)$ such that
\begin{itemize}
\item for any $(b,w) \in \ell_i \cap U$, the objects $E,\,F$ lie in the heart $\cA_{\;b}$,
\item $E$ is $\nu\_{b,w}$-semistable of class $v$ with $\nu\_{b,w}(E)=\nu\_{b,w}(F)=\,\mathrm{slope}\,(\ell_i)$ constant on the wall $\ell_i \cap U$, and
\item $f$ is an injection $F\into E $ in $\cA_{\;b}$ which strictly destabilises $E$ for $(b,w)$ in one of the two chambers adjacent to $\ell_i$.
\hfill$\square$

\end{itemize} 
	\end{itemize*} 
\end{Prop}

In this paper, we always assume $X$ satisfies the conjectural Bogomolov-Gieseker inequality of Bayer-Macr\`i-Toda \cite{BMT}.  In the form of \cite[Conjecture 4.1]{BMS}, rephrased in terms of the rescaling \cite[Equation 6]{FT2}, it is the following.

\begin{Con}\label{conjecture}
For any $(b,w)\in U$ and $\nu\_{b,w}$-semistable $E\in\cD(X)$, we have the inequality
\begin{equation}\label{quadratic form}
    B_{b, w}(E)\ :=\ (2w-b^2)\Delta_H(E) + 4\big(\ch_2^{bH}(E).H\big)^2 -6\(\ch_1^{bH}(E).H^2\)\ch_3^{bH}(E)\ \geq\ 0.
\end{equation}
\end{Con}

\begin{Rem}\label{remark}
Multiplying out and cancelling we find that $B_{b, w}$ is actually linear in $(b,w)$:
\beq{boglinear}
\tfrac12B_{b, w}(E)\=\(C_1^2-2C_0C_2\)w+\(3C_0C_3-C_1C_2\)b+(2C_2^2-3C_1C_3),
\eeq
where $C_i:=\ch_i(E).H^{3-i}$.
Thus the Bogomolov-Gieseker inequality \eqref{quadratic form} says that $E$ can be $\nu\_{b,w}$-semistable only on one side of the line $\ell_f(E)$ defined by the equation $B_{b, w}(E)=0$. When $\ch_0(E)\neq0\neq\ch_1(E).H^2$ we can rearrange to see $\ell_f(E)$ is the line through the points $\Pi(E)$ and 
\begin{equation}\label{pi'}
\Pi'(E)\ :=\, \left(\frac{2\ch_2(E).H}{\ch_1(E).H^2}\,, \ \frac{3\ch_3(E)}{\ch_1(E).H^2}\right).    
\end{equation}
\end{Rem}

\section{Semistable objects are sheaves}\label{shvs}
Fix a rank $r\geq1$, arbitrary $p_1,p_2, q\in\N$, and a class $\mathsf v\in K(X)$ whose Chern character
\beq{vdef}
\ch(\mathsf v)\=\(r,\,D,\,-\beta,\,-m\)\ \in\ H^{2*}(X,\Q)
\eeq
satisfies the bounds\footnote{In Section \ref{r} we will show how to remove the $D.H^2=0$ condition, to work with arbitrary $D$.}
\begin{equation}\label{general bound}
D.H^2\,=\,0, \quad -p_1\,\leq\,\beta.H\,\leq\,p_2 \qquad \text{and} \qquad m\,\leq\,q.
\end{equation}
Then we pick $n\gg0$, a function of $r,p_1,p_2,q$, as follows. There are finitely many explicit inequalities of the form $O\(n^i\)<O(n^{i+1})$ in this Section \ref{shvs} and Section \ref{subshvs}, starting at \eqref{equation of final wall} and ending with \eqref{marker}. We fix
\beq{nzero}
n\=n(r,p_1,p_2,q)\ \gg\ 0
\eeq
sufficiently large that all of these inequalities hold for
\begin{itemize}
\item all rank $r$ classes $\mathsf v$ satisfying the bounds \eqref{general bound}, and
\item all rank $1\le r'\le r-1$ classes $\mathsf v'= (r', D', \beta', m')$ satisfying the bounds \eqref{general bound} with $(r, D,\beta,m)$ and $p_1,p_2,q$ replaced by $(r', D',\beta',m')$ and $p’_1,p’_2,q’$. Here $p’_1,p’_2,q’$ are the (maximum of the) functions of $r,p_1,p_2,q$ that arise in \eqref{bounds'} and Theorem \ref{thm:all walls}.
\end{itemize}\smallskip

Set $v_n:=\mathsf v-[\cO_X(-n)]$ with
\begin{equation}\label{class vn}
    \ch(\vi_n)\=\Big( r-1, \ nH+D ,\ -\beta -\tfrac12n^2H^2 , \ -m +\tfrac16n^3H^3\Big)\,\in\ H^{2*}(X,\Q).
\end{equation}
Throughout the paper we study $\nu\_{b,w}$-semistable objects $E$ of class $v_n$ for stability conditions $\nu\_{b,w}$ with \emph{$b$ strictly to the left of $\Pi(E)$}. That is, we take $b<\mu\_H(E)=\frac n{r-1}$ so that the denominator of $\nu\_{b,w}(E)$ is strictly positive. A special role will be played by stability conditions on what we call the \emph{Joyce-Song wall} $\ell_{\js}\subset\R^2$ --- the line joining
\beq{ljsdef}
\Pi(v_n) \quad\text{and}\quad \Pi(\cO(-n)[1]) \quad\text{(and therefore also }\,\Pi(\mathsf v)).
\eeq
The goal of this section is to prove that semistable objects $E$ of class $v_n$ are \emph{sheaves}, with one exception: for $(b,w)\in\ell_{\js}$ objects of the form
\beq{fake}
F\oplus T(-n)[1], \qquad F\text{ a sheaf of class }\mathsf v,\ T\in\Pic\_0(X),
\eeq
are $\nu\_{b,w}$-semistable if $F$ is. (We will see later in Lemma \ref{lem: bounding the first wall} that such $F$ are semistable in the large volume chamber $w\gg0$ and so are torsion-free sheaves.) These objects \eqref{fake} are unstable on both sides of $\ell_{\js}$ so will not participate in the wall crossing formula.

\begin{Thm}\label{Thm. large n-arbitrary rank}
Given $(b,w)\in U$ with $b<\frac n{r-1}$, any $\nu\_{b,w}$-semistable $E \in \cA_{\;b}$ of class $\vi_n$ \eqref{class vn} is either a sheaf or of the form \eqref{fake}.
\end{Thm}


We start by noting the equation of the line $\ell_f:=\ell_f(E)$ passing through the points $\Pi(\vi_n)$ and $\Pi'(\vi_n)$ \eqref{pi'},
\begin{equation}\label{equation of final wall}
w \= \left(-\frac{n}{2} + \frac{n(r-2)\beta.H + 3(r-1)m}{n^2rH^3 +2(r-1)\beta.H}  \right) b 
- \frac{3nm +2n^2\beta.H +2(\beta.H)^2/H^3}{n^2rH^3 +2r\beta.H}\,.
\end{equation}
The Bogomolov-Gieseker Conjecture \ref{conjecture} then gives the following. 
\begin{Lem}\label{final wall}
There are no $\nu\_{b,w}$-semistable objects of class $\vi_n$ \eqref{class vn} if $(b,w) \in U$ lies below the line $\ell_f\cap U.\hfill\square$
\end{Lem}


Note $\ell_f\cap U$ is nonempty for $n\gg0$. Indeed let $a_f<b_f$ be the values of $b$ at the two intersection points $\ell_f\cap\partial U$ as in Figure \ref{llf}.
As $n\to\infty$ the coefficient of $b$ in \eqref{equation of final wall} has leading term $-\frac{n}{2}$, while the constant term of \eqref{equation of final wall} tends to $-\frac{2\beta.H}{rH^3}$. Thus we can make $a_f,\,b_f$ as close as we like to $-n$ and $0$ when $n\gg0$. In fact for any fixed $\epsilon>0$ we may take $n$ sufficiently large that the points
$$
\(\!-n+\epsilon,\tfrac12(-n+\epsilon)^2\)\quad\mathrm{and}\quad
\(\!-\epsilon,\tfrac12\epsilon^2\;\)
$$
of $\partial U$ lie below $\ell_f$, so 
\begin{equation}\label{bounds on bf}
a_f\  <\ -n+ \epsilon \qquad \text{and} \qquad  -\epsilon\ <\ b_f.    
\end{equation}
We fix $\epsilon=1/4r^2H^3$ and choose $n\gg0$ accordingly.

We will use \eqref{bounds on bf} repeatedly in the following argument from \cite{FT1,FT2}. Consider an object $F\in\cA_{\;b}$ which is $\nu\_{b,w}$-semistable for $(b,w)$ on a wall $\ell\cap U$ passing through $\Pi(F)$. If $\ell\cap U$ is very long --- for instance if it is close to $\ell_f\cap U$ --- then the fact that $F\in\cA_{\;b}$ along the whole length of the wall constrains it (and its destabilising objects), in particular forcing its $\cH^{-1}$ to have low rank.

So suppose $\ell$ lies \emph{above or on} another line $\tilde{\ell}$ through $\Pi(F)$ which intersects the boundary $\partial U$ at two points with $b$-values $\tilde{a} < \tilde{b}$.
\begin{figure}[h]
	\begin{centering}
		\definecolor{zzttqq}{rgb}{0.27,0.27,0.27}
		\definecolor{qqqqff}{rgb}{0.33,0.33,0.33}
		\definecolor{uququq}{rgb}{0.25,0.25,0.25}
		\definecolor{xdxdff}{rgb}{0.66,0.66,0.66}
		
		\begin{tikzpicture}[line cap=round,line join=round,>=triangle 45,x=1.0cm,y=1.0cm]
		
		\draw[->,color=black] (-4,0) -- (4,0);
		\draw  (4, 0) node [right ] {$b,\,\frac{\ch_1\!.\;H^2}{\ch_0H^3}$};
		
		\fill [fill=gray!40!white] (0,0) parabola (3,4) parabola [bend at end] (-3,4) parabola [bend at end] (0,0);
		
		\draw[thick]  (0,0) parabola  (3.1,4.27); 
		\draw[thick]  (0,0) parabola (-3.1,4.27); 
		\draw  (3.8 , 3.6) node [above] {$w= \frac{b^2}{2}$};
		
		\draw[->,color=black] (0,-1.2) -- (0,4.6);
		\draw  (0.9, 4.2) node [above ] {$w,\,\frac{\ch_2\!.\;H}{\ch_0H^3}$};
		
		\draw[color=black, dashed] (.58, 2)-- (.58, 0);
		\draw [dashed, color=black] (-1.87,3) -- (-1.87, 0);
		\draw [color=black] (2.5, -1) -- (-3.5, 2.5);
		\draw [color=black] (2.5, -1) -- (-3.5, 3.5);
		
		\draw  (-1.87,0) node [below] {$\tilde{a}$};
		\draw  (.58, 0) node [below] {$\tilde{b}$};
		\draw  (2.8,-.3) node [below] {$\Pi(F)$};
	    \draw  (-3.5, 2.5) node [left] {$\tilde{\ell}$};
	    \draw  (-3.5, 3.5) node [left] {$\ell$};
		
		\begin{scriptsize}
		\fill (.58, .15) circle (2pt);
		\fill (-1.87,1.55) circle (2pt);
		\fill (2.5,-1) circle (2pt);

		\end{scriptsize}
	
		\end{tikzpicture}
		\caption{The lines $\tilde\ell$ and $\ell$}\label{llf}		
	
	\end{centering}
\end{figure}
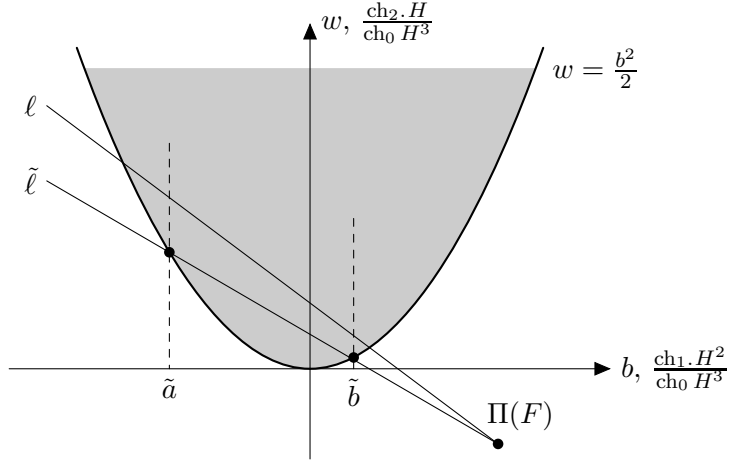
This means that as we move $b\uparrow \tilde{b}$ or $b\downarrow\tilde{a}$ along $\ell$, the point $(b,w)$ lies in $U$, so $F\in\cA_{\;b}$ by Proposition \ref{prop. locally finite set of walls}. By \eqref{Abdef} this implies
$$
\mu\_H\(\cH^{-1}(F)\)\ \le\ \tilde{a} \quad\text{and}\quad \mu\_H\(\cH^0(F)\)\ \ge\ \tilde{b}.
$$
Setting $r':=\rk\;(\cH^{-1}(F))$, so that $\rk\;(\cH^0(F))=\rk\;(F)+r'$, 
yields
\beq{esti1}
\ch_1\!\(\cH^{-1}(F)\).\;H^2\,\le\,r'H^3\tilde{a} \quad\text{and}\quad
\ch_1\!\(\cH^0(F)\).\;H^2\,\ge\,(\rk(F)+r')H^3 \tilde{b}.\hspace{-1mm}
\eeq
Subtracting gives the key inequality
\beq{estimate}
\ch_1(F).\;H^2\ \ge\ r'(\tilde{b} - \tilde{a})H^3+ \rk\;(F)\;H^3\;\tilde{b}.
\eeq
In applications this will bound $r'$. When $\tilde\ell=\ell_f$ then in combination with \eqref{bounds on bf} for small $\epsilon=1/4r^2H^3$ it will give $\ch_1(F).\;H^2\ge r'nH^3$.
Note that when $\ell$ lies \emph{strictly above} $\tilde{\ell}$ then $\mu\_H\(\cH^{-1}(F)\) < \tilde{a}$, so \eqref{estimate} becomes a \emph{strict inequality}. 

The rest of this section is devoted to the proof of Theorem \ref{Thm. large n-arbitrary rank}. We fix a $\nu\_{b,w}$-semistable $E \in \cA_{\;b}$ of class $\vi_n$ \emph{not of the form \eqref{fake}}. Assume for a contradiction that it is also not a sheaf. Then by the definition \eqref{Abdef} of $\cA_{\;b}$ we know $\cH^{-1}(E)$ is a nonzero torsion-free sheaf so $r':=\rk\;(\cH^{-1}(E))\ge1$.

\begin{Lem}\label{lem:cohomologies of E}
$\cH^{-1}(E)$ is a rank one torsion-free sheaf with $\ch_1(\cH^{-1}(E)).H^2 = -nH^3$. 
\end{Lem}
\begin{proof}
By Proposition \ref{prop. locally finite set of walls} $E$ is semistable along the intersection of $U$ with the line $\ell$ passing through $\Pi(E)$ and $(b,w)$. By Lemma \ref{final wall} it lies above (or on) the line $\ell_f$. So by our argument \eqref{estimate} applied to $E, \ell_f, a_f, b_f$ in place of $F, \tilde{\ell}, \tilde{a}, \tilde{b}$ we deduce
\beq{estr}
nH^3\  \geq \ r'(b_f -a_f)H^3 + (r-1)H^3b_f\ >\ r'nH^3-2r'\epsilon H^3-(r-1)\epsilon H^3,
\eeq
where the second inequality comes from \eqref{bounds on bf}.
Therefore $r'=1$ and the inequality \eqref{estr} is tight for $\epsilon=(4rH^3)^{-1}$: if we added a positive integer to the right hand side it would fail. It follows that the inequalities \eqref{esti1} used to derive \eqref{estr} are also tight, giving
\[
\ch_1\!\(\cH^{-1}(E)\).\;H^2\=-nH^3 \quad\text{and}\quad  
\ch_1\!\(\cH^0(E)\).\;H^2\=0. \qedhere
\]
\end{proof}

Now move down $U$, shrinking $w$ while keeping $b<\frac n{r-1}$ to the left of $\Pi(E)$. By Lemma \ref{final wall} and Proposition \ref{prop. locally finite set of walls} we find a wall of instability $\ell_E$ for $E$ above or on $\ell_f$, below which $E$ is strictly unstable. So we may choose a sequence $E_1 \hookrightarrow E \twoheadrightarrow E_2$ in $\cA_{\;b}$ which strictly destabilises below $\ell_E$. That is, $\nu\_{b,w}(E_i)=\nu\_{b,w}(E)$ for $(b,w) \in \ell_E \cap U$ and
$$
\nu\_{b,w'}(E_1)\ >\ \nu\_{b,w'}(E_2) \quad\mathrm{for\ }w'<w.
$$
We next constrain these destabilising objects $E_1,\,E_2$. We will find the class of $E_1$ is similar to $\mathsf v$ \eqref{vdef} but with smaller rank, and the class of $E_2$ is therefore like $v_n$ \eqref{class vn} with smaller rank. This will later allow us to prove Theorem \ref{Thm. large n-arbitrary rank} by an induction on rank.
 
\begin{Prop}\label{prop.first destabilising object}
In the above situation, $E_1$ is a $\nu\_{b,w}$-semistable sheaf with
\begin{align}
&\ch_0(E_1)\ \in\ [1,r-1], \label{ch0} \\
&\ch_1(E_1).H^2\=0,  \label{ch1} \\
&\ch_2(E_1).H\ \in\ \left[-\tfrac{2\beta.H+1}r\ch_0(E_1),\,0\right]\!,
\label{bounds on ch2} \\
    &\ch_3(E_1)\ \leq\ 
    \frac{2}{3}\ch_2(E_1).H\left(\!\ch_0(E_1)^2\ch_2(E_1).H -\frac{1}{2H^3\ch_0(E_1)^2}    \right).\label{bounds on ch3}
\end{align}
\end{Prop}

\begin{proof}
The long exact sequence of sheaf cohomology groups of $E_1 \hookrightarrow E \twoheadrightarrow E_2$ is
\beq{LES}
    0 \to \cH^{-1}(E_1) \To \cH^{-1}(E) \To \cH^{-1}(E_2)
    \To \cH^{0}(E_1) \rt\alpha \cH^{0}(E)\To \cH^{0}(E_2)\to 0.
\eeq
By Lemma \ref{lem:cohomologies of E}, $\cH^{-1}(E)$ is a rank one torsion-free sheaf. Let $r':=\rk\;(\cH^{-1}(E_2))$. Since $\ell_E$ lies above the line $\ell_f$ we find, by taking $\epsilon=(4rH^3)^{-1}$ in \eqref{bounds on bf} and $a_f, b_f$ in place of $\tilde a,\tilde b$ in \eqref{esti1}, 
\begin{equation}\label{c.2}
    \ch_1(\cH^{-1}(E_2)).H\ \le\ -nr'H^3 + \frac{r'}{2r}  \quad\text{and}\quad 
\ch_1(\cH^0(E_1)).H^2\ >\  -\frac{r'+r}{2r}
\end{equation}
because $\rk\;(\cH^0(E_1)) \leq \rk\;(\cH^{-1}(E_2)) + \rk\;(\cH^0(E)) = r'+r$ by \eqref{LES}. Similarly
\begin{equation}\label{c.4}
\ch_1(\cH^0(E)).H^2\ >\ -rH^3\epsilon\ >\  -\frac12 \quad\text{so}\quad
    \ch_1(\cH^0(E_2)).H^2\ \geq\ 0
\end{equation}
since $\rk\;(\cH^0(E_2))\le\ch_0(\cH^0(E))=r$. By Lemma \ref{lem:cohomologies of E},
\beq{hash}
\ch_1(\cH^0(E)).H^2 \= \ch_1(E).H^2+\ch_1(\cH^{-1}(E)).H^2\=0,
\eeq
so \eqref{c.4} and \eqref{LES} give
\begin{equation}\label{c.5}
    \ch_1\!\big(\text{im} \ \alpha \big).H^2 \leq 0\ .
\end{equation}
We now consider two different cases.

\textbf{Case 1.} Suppose $\cH^{-1}(E_1) \neq 0$. Since $\cH^{-1}(E)$ is a rank one torsion-free sheaf, by Lemma \ref{lem:cohomologies of E}, and $\cH^{-1}(E_2)$ is a torsion-free sheaf (because $E_2\in\cA_{\;b}$) we see from \eqref{LES} that
$\cH^{-1}(E_1)\ \cong\ \cH^{-1}(E)$ and \eqref{c.5} becomes 
$$
    \ch_1\!\big(\cH^0(E_1) / \cH^{-1}(E_2)\big).H^2\ \leq\ 0.
$$
Combined with \eqref{c.2} this gives
\begin{align}\label{c.3}
-\frac{r'+r}{2r}+ nr'H^3 - \frac{r'}{2r}  \ <\    \ch_1(\cH^0(E_1)).H^2 -\ch_1(\cH^{-1}(E_2)).H^2\ \leq\ 0.
\end{align}
This forces $r'=0$, so in fact $\cH^{-1}(E_2) = 0$ and $E_2$ is a sheaf. Then \eqref{c.3} gives
$$
\ch_1(\cH^0(E_1)).H^2 \,=\, 0 \quad\text{and so, by \eqref{hash},}\quad  \ch_1(\cH^0(E_2)).H^2 \,=\, 0.
$$
Thus $\ch_1(E_2).\;H^2=0$. Combined with the fact that $E_2$ has finite slope $\nu\_{b,w}(E)$ on $\ell_E$ we deduce it  is a sheaf of strictly positive rank. So $\Pi(E_2)$ lies on the $w$-axis with our stability condition $(b,w)$ to the left (i.e. $b<0$) and $\Pi(E)$ to the right. Then on moving $(b,w)$ below $\ell_E$, the line joining it to $\Pi(E_2)$ has greater slope than the line joining it to $\Pi(E)$. This contradicts the fact that $\nu\_{b,w}(E_2)<\nu\_{b,w}(E)$ below the wall.

\textbf{Case 2.} Thus $E_1$ is a sheaf and \eqref{c.2} and \eqref{hash} imply
\begin{align*}
\ch_1\!\(\cH^0(E_1)-\cH^{-1}(E_2)+\cH^{-1}(E)\).H^2\ 
>\  -\frac{r'+r}{2r}+nr'H^3 - \frac{r'}{2r} -nH^3.
\end{align*}
The left hand side is $\ch_1(\im\alpha).H^2\le0$ \eqref{c.5}. Since the right hand side is $(r'-1)H^3n+O(1)$ for $n \gg 0$ this gives $r' =1$. 
Therefore $\cH^{-1}(E_2)/\cH^{-1}(E)$ is a sheaf of rank zero, so has $\ch_1\!.\,H^2\ge0$, giving
\begin{equation*}
    -nH^3\ \leq\ -nH^3 + \ch_1\!\left(\frac{\cH^{-1}(E_2)}{\cH^{-1}(E)}\right)\!.\;H^2 \= \ch_1(\cH^{-1}(E_2)).H^2\ \le\ -nH^3 + \frac{1}{2r}\,. 
\end{equation*}
(The equality is Lemma \ref{lem:cohomologies of E}; the last inequality comes from  \eqref{c.2}.) Thus $\ch_1(\cH^{-1}(E_2)).H^2 = -nH^3$ and $\cH^{-1}(E_2)/\cH^{-1}(E)$ has support $Z\subset X$ of dimension at most 1. Thus \eqref{LES} gives the exactness of
\beq{SES}
0\To E_1\To\cH^0(E)\To\cH^0(E_2)\To0 \quad\text{on }X\smallsetminus Z.
\eeq
But $\ch_1(\cH^{0}(E_i)) .H^2\ge0$ by \eqref{c.2} (with $r'=1$) and \eqref{c.4}. Since $\ch_1(\cH^{0}(E)) .H^2= 0$ by \eqref{hash} we deduce
$$
\ch_1(E_1) .H^2\= 0\=\ch_1(\cH^{0}(E_2)) .H^2
$$
on $X\smallsetminus Z$ and so also on $X$ because $\dim Z\le1$.
This gives \eqref{ch1}. Since $\nu\_{b,w}(E_1)=\nu\_{b,w}(E)<+\infty$ we see $\rk\;(E_1)>0$, while \eqref{SES} gives $\rk\;(E_1)\le r$. 

Now $\Pi(E_1)=(0, \ch_2(E_1).H/\ch_0(E_1)H^3)$ lies on $\ell_E$ and so above or on $\ell_f$. But by \eqref{equation of final wall} at $b=0$ we have $w\sim-2\beta.H/rH^3$ to leading order in $n\gg0$. Therefore
\beq{mess}
-\frac{2\beta.H +1}{rH^3}\ \le\ \frac{\ch_2(E_1).H}{\ch_0(E_1)H^3}
\eeq
for $n\gg0$. Then \eqref{BOG} gives $\ch_2(E_1).H \leq 0$, so this proves \eqref{bounds on ch2}.

Next $\ch_0(E_1)\in[1,r]$  gives $-\(\!\ch_0(E_1)^2H^3\)^{-1}\le-\(r^2H^3\)^{-1}<b_f$ \eqref{bounds on bf}. So since $\ell_E$ is above or on $\ell_f$ there is $w >0$ such that  
$\left(\frac{-1}{\ch_0(E_1)^2H^3}, w  \right) \in \ell_E \cap U$, where $E_1$ is semi\-stable. Therefore Lemma \ref{lem: bounding the first wall} below applies to $E_1$ to give the bound \eqref{bounds on ch3} for $\ch_3(E_1)$.\medskip

We are just left with ruling out $\rk\;(E_1)=r$. Supposing otherwise we write $\ch(E_1) = (r, \ch_1, \ch_2, \ch_3)$, so
\begin{equation*}
    \ch(E_2) \= \Big(\!-1,\ nH +D -\ch_1,\  -\beta-\ch_2 -\tfrac12n^2H^2, \ -m-\ch_3 + \tfrac16n^3H^3\Big).
\end{equation*}
Now $(D-\ch_1).H^2=0$ and $E_2$ is semistable on $\ell_E\cap U$, which is above $\ell_f\cap U$ and so intersects $b=-n+1/H^3$ by \eqref{bounds on bf}. Therefore Lemma \ref{lem: bounding the first wall, rank -1} below applies to $E_2$, giving
$$
    0\ \leq \ \frac{2}{3}\left(\beta.H+ \ch_2\!.\;H \right)\left(\beta.H +\ch_2\!.\;H - \frac{1}{2H^3}\right) + m+\ch_3+nH(\beta+\ch_2).
    $$
Each term can be bounded above by \eqref{bounds on ch2}, \eqref{bounds on ch3} and \eqref{general bound}. Thus the coefficient $(\beta+\ch_2).H$ of $n$ must be $\ge0$.

If $-\beta.H < \ch_2\!.\;H$, then $\Pi(E_1)$ lies above $\left(0, -\frac{\beta.H}{rH^3}\right)=\Pi(\mathsf v)$. Therefore $\ell_E$ lies above the Joyce-Song wall $\ell_{\js}$ \eqref{ljsdef} passing through $\Pi(\mathsf v)$ and $\Pi(\cO_X(-n))=\(-n,\frac12n^2\)\in\partial U$. Thus the leftmost intersection point of $\ell_E$ with $\partial U$ has $b=b'<-n$. Since $E$ lies in $\cA_{\;b}$ along $\ell_E\cap U$ this gives  the contradiction
$$
\mu\_H(\cH^{-1}(E)) \= -n\ \leq\ b'.
$$
Thus $\ch_2\!.\;H = -\beta.H$, so Lemma \ref{lem: bounding the first wall, rank -1} gives $E_2\cong T(-n)[1]$ for some $T\in\Pic\_0(X)$. Hence $E$ lies in the triangle 
\begin{equation}\label{triangle}
E_1 \To E \To T(-n)[1].     
\end{equation}
Now $E_1$ is torsion-free by Lemma \ref{lem: bounding the first wall}. Thus the map $s \colon T(-n) \rightarrow E_1$ induced by \eqref{triangle} is either zero or an injection. In the first case $E$ takes of the form \eqref{fake}, in the second it is a sheaf. Both contradict our assumptions.
\end{proof}

The next two Lemmas give vertical lines through $U$ on which we can rule out walls of instability for objects of certain classes.

\begin{Lem}\label{lem: bounding the first wall}
Fix $F\in\cD(X)$ with $\ch_0(F)\ge1$ and $\ch_1(F).H^2 = 0$. Set
$b_0:=-1/\ch_0(F)^2H^3$. 
If $F\in\cA_{\;b_0}$ and is $\nu\_{b_0, w}$-semistable for some $w > \frac12b_0^2$ then $F$ is a torsion-free sheaf, $\nu\_{b_0,w}$-semistable for \emph{all} $w > \frac12b_0^2$, with
\begin{equation}\label{upper bound for ch3-general}
\ch_3(F)\ \leq\ \frac{2}{3}\ch_2(F).H\left(\!\ch_0(F)\ch_2(F).H -\frac{1}{2H^3\ch_0(F)^2}    \right). 
\end{equation}
\end{Lem}
\begin{proof}
The first claim follows from \cite[Lemma 8.1]{FT1} if $\ch_0(F) =1$ and \cite[Theorem 3.1]{Fey21} if $\ch_0(F)>1$. (Both results are stated for sheaves but their proofs are valid for any object in $A_{b_0}$ of positive rank.)
Thus $F$ is $\nu\_{b_0,w}$-semistable for all $w > \frac12b_0^2$, so the Bogomolov-Gieseker inequality \eqref{quadratic form} implies that 
\begin{align*}
    0\ \leq\ \lim_{w \rightarrow \frac12b_0^2}\ \frac12B_{b_0,w}(F)\= 2\(\!\ch_2^{b_0H}(F).H\)^2 - 3\(\!\ch_1^{b_0H}(F).H^2\)\ch_3^{b_0H}(F), 
\end{align*}
which implies the upper bound \eqref{upper bound for ch3-general}.  Finally \cite[Lemma 2.7(c)(i)]{BMS} gives that $F$ is a torsion-free sheaf.
\end{proof}

\begin{Lem}\label{lem: bounding the first wall, rank -1}
	Let $b':= -n+\frac{1}{H^3}$ and consider a $\nu\_{b',w}$-semistable object $F \in \cA_{\;b'}$ with
	\begin{equation*}
	\ch(F)\ = \Big(\!-1,\ nH + D' ,\,-\beta' -\tfrac12n^2H^2,\ -m'+ \tfrac16n^3H^3\Big),
	\end{equation*}
	where $D'.H^2 = 0$. Then $F$ is $\nu\_{b', w}$-stable for every $w > \frac{b'^2}{2}$, and
	\begin{equation}\label{ch3-m}
	-n\beta'.H - \frac{2}{3}\beta'.H \left(\beta'.H - \frac{1}{2H^3} \right)\ \leq\ m'. 
	\end{equation}
	Moreover, if $\beta'.H = 0$, 
	then $F \cong T(-n)[1]$ for some $T\in\Pic\_0(X)$.
\end{Lem}	
\begin{proof}
	Similar arguments appear in \cite{FT2, FT1}; we give the details for completeness. The choice of $b'$ ensures the denominator $\ch_1^{b'H}\!.\;H^2$ of $\nu\_{b', w}$ is additive, integral, nonnegative \eqref{65} on $\cA_{\;b'}$ and takes the minimum possible value of 1 on $F$, so $F$ can only be destabilised on the vertical line $b = b'$ by objects with $\ch_1^{bH}\!.\;H^2=0$. These have constant $\nu\_{b',w}=+\infty$, so (i) $F$ has no walls of instability crossing this line and (ii) semistability implies stability, proving the first claim.
	
	Thus $F(n)\in \cA_{\;b'+n} =\cA_h$ is $\nu\_{h, w}$-semistable for $w \gg 0$ and $h\coloneqq 1/H^3$.	Therefore, by \cite[Lemma 5.1.3(b)]{BMT} the shifted derived dual $F(n)^{\vee}[1]$ lies in an exact triangle 
	\begin{equation*}
	F'\ \Into\, F(n)^{\vee}[1]\, \To\hspace{-5.5mm}\To\, Q[-1],
	\end{equation*}
	with $Q$ a zero-dimensional sheaf and $F'$ a $\nu\_{-h,w}$-semistable object of $\cA_{\,-h}$ for $w \gg 0$. Since $\rk F'=1$ it is a torsion-free sheaf by \cite[Lemma 2.7]{BMS}. We also have $\ch_1(F').H^2 = \ch_1(F(n)).H^2 = 0$. Hence Lemma \ref{lem: bounding the first wall} gives 
	\begin{equation*}
	\ch_3(F')\ \leq\ \frac{2}{3} \ch_2(F').H\left(\ch_2(F').H - \frac{1}{2H^3}\right)\!. 
	\end{equation*}
	Since $\ch_2(F').H = \beta'.H$ and $-m'-n\beta'.H = \ch_3(F') -\ch_3(Q) \leq \ch_3(F')$ the claimed inequality \eqref{ch3-m} follows. 
	
	Now assume $\beta'.H = 0$. By \cite[Lemma 2.7]{BMS}, $\cH^{-1}(F)$ is a torsion-free sheaf of rank 1 and $\cH^0(F)$ is supported in dim $\leq 1$. Therefore  $\cH^{-1}(F)$ is $\mu\_H$-semistable and satisfies the classical Bogomolov inequality, 
	\begin{equation}\label{c1}
	\ch_1(\cH^{-1}(F))^2.H\ \geq\ 2\ch_2(\cH^{-1}(F)).H\ =\ n^2H^3 + 2\ch_2(\cH^0(F)).H.  
	\end{equation}
	By the Hodge index theorem, 
	\begin{equation}\label{ch11}
	n^2H^3\ =\ \frac{\big(\!\ch_1(\cH^{-1}(F)).H^2\big)^2}{H^3}\ \ge\ \ch_1(\cH^{-1}(F))^2.H,
	\end{equation}
	with equality if and only if $\ch_1(\cH^{-1}(F))$ is a multiple of $H$ in $H^2(X,\Q)$. Combining \eqref{c1} and \eqref{ch11} gives
	\begin{equation}\label{<>0}
	\ch_2(\cH^0(F)).H\ \le\ 0.
	\end{equation}
	But $\dim\mathrm{supp}\,\cH^0(F)\le1$, so this shows  $\dim\mathrm{supp}\,\cH^0(F)=0$ and (\ref{<>0}, \ref{c1}, \ref{ch11}) are equalities. Thus $\ch_1(\cH^{-1}(F))$ is a multiple of $H$ in $H^2(X,\Q)$ and $D' = 0$. Hence there is a zero-dimensional subscheme $Z\subset X$ such that
	\begin{equation}\label{TC}
	\cH^{-1}(F)\ \cong\ T(-n)\otimes I_Z
	\end{equation}
	for some $T\in\Pic\_0(X)$. If $Z$ were nonempty then $\nu\_{b',w}(\cO_Z)=+\infty$, so combining the $\cA_{\;b'}$-short exact sequences
	$$
	\cO_Z \Into \cH^{-1}(F)[1] \Onto T(-n)[1] \qquad\text{and}\qquad \cH^{-1}(F)[1]\Into F \Onto \cH^0(F)
	$$
	gives the destabilising subobject $\cO_Z \hookrightarrow F$, contradicting the $\nu\_{b',w}$-semistability of $F$.
	
	Therefore $Z=\emptyset$ and $\cH^{-1}(F)\cong T(-n)$.
	Since $\cH^0(F)$ is supported in dimension zero, $\Ext^2(\cH^0(F), T(-n)) = 0$, so $F$ splits as $T(-n)[1] \oplus \cH^0(F)$. If $\cH^0(F)\ne0$ it has $\nu\_{b', w} = +\infty$, destabilising $F$. Thus in fact $F=T(-n)[1]$.
\end{proof}	

\begin{proof}[Proof of Theorem \ref{Thm. large n-arbitrary rank}]
Assume for a contradiction the existence of a $\nu\_{b,w}$-semistable $E \in \cA_{\;b}$ of class $\vi_n$ which is neither a sheaf nor of the form \eqref{fake}. Moving down from the point $(b,w)$ we hit a wall $\ell_E$ for $E$ by Lemma \ref{final wall}, with destabilising sequence $E_1\into E\onto E_2$ in $\cA_{\;b}$ described by Proposition \ref{prop.first destabilising object}. In particular $\ch_0(E_1)\in[1,r-1]$ \eqref{ch0} already gives a contradiction if $r=1$. Thus we can use this as the base case of an induction on rank and assume the Theorem has been proved for $\rk(E)\le r-2$.

Set $\ch(E_1)=(\ch_0, \ch_1, \ch_2, \ch_3)$ so $E_2$ is a $\nu\_{b,w}$-semistable object of Chern character
\begin{equation*}
    \vi_n - \ch(E_1)\,=\,\Big(r-1-\ch_0 , \ nH+D -\ch_1 , \ -\beta -\ch_2-\tfrac12n^2H^2,\  -m -\ch_3 +\tfrac16n^3H^3 \Big).
\end{equation*}
We have $r-1-\ch_0\in[0,r-2]$ by \eqref{ch0} and $(D-\ch_1).H^2 = 0$ by \eqref{ch1}. Moreover, by \eqref{bounds on ch2}, $(\beta +\ch_2).H$ has upper and lower bounds $-p_1',\,p_2'$ depending on $p_1$ and $p_2$, and by \eqref{bounds on ch3} $\ch_3 +m$ has an upper bound $q'$ depending on $r,\,p_1,\,p_2$ and $q$:
\beq{bounds'}
 -p'_1\ \leq\ (\beta+\ch_2).H\ \leq\ p'_2 \qquad \text{and} \qquad m+\ch_3\ \leq\ q'.
\eeq
Therefore we may apply the Theorem to $E_2$ of rank $\le r-2$. But it cannot be a sheaf since then $E$ would be an extension of two sheaves $E_1$ and $E_2$ and hence a sheaf itself.

So $E_2=F_2\oplus T(-n)[1]$, giving us a surjection $E\onto T(-n)[1]$ in $\cA_{\;b}$ for $(b,w) \in \ell_E \cap U$. Its kernel in $\cA_{\;b}$ is a $\nu\_{b,w}$-semistable object $F$ of class $\mathsf v$, which by Lemma \ref{lem: bounding the first wall} is a torsion-free sheaf. So the induced map of sheaves $T(-n)\to F$ is either an injection or zero, giving the contradiction that $E$ is either a sheaf or of the form \eqref{fake}.
\end{proof}

\section{Destabilising objects are sheaves}\label{subshvs}
In this section we will analyse walls for sheaves of class $\vi_n$ and show the destabilising objects are sheaves, except on one wall where we get Joyce-Song pairs.

\begin{Def}\label{+ve}
We say a class $v\in K(X)$ with $\ch(v)= (\ch_0, \ch_1, \ch_2, \ch_3) \in H^{2*}(X, \mathbb{Q})$ is positive, written $v>0$, if and only if
\begin{itemize}
\item $\ch_0>0$, or
\item $\ch_0=0$ and $\ch_1\!.\;H^2>0$.
\end{itemize}
\end{Def}

We will define a ``safe area'' $U_v\subset U$ for a positive class $v$ such that whenever $(b,w)\in U_v$, all $\nu\_{b,w}$-semistable objects of class $v$ are sheaves, and any semi-destabilising objects are also sheaves.\medskip

If $\Delta_H(v)<0$ there are no semistable objects by \cite[Theorem 3.5]{BMS}, so we set $U_v=U$.

If $\Delta_H(v)=0$ set $U_v=\{(b,w)\in U\colon b<\mu\_H(v)\}$ to be the points of $U$ to the left of $\Pi(v)$.

If $\Delta_H(v)>0$ there are two cases. Either $\ch_0>0$, so $\Pi(v)$ lies outside $\overline U=U\cup\partial U$, then draw any line through $\Pi(v)$ that misses $U$. Or $\ch_0=0$, then draw any line of slope $\ch_2\!.\;H/\!\;\ch_1\!.\;H^2$ that misses $U$. Now rotate the line clockwise about $\Pi(v)$ in the first case, or fix its slope $\ch_2\!.\;H/\!\;\ch_1\!.\;H^2$ and move it upwards in the second. In both cases it will eventually pass through $U$ (to the left of $\Pi(v)$ in the first case), intersecting $\partial U$ in two points which move further apart until we find a \emph{unique} $\ell_v$ for which the $b$ values $a_v < b_v$ of the two points of $\ell_v\cap\partial U$ satisfy
    \begin{equation}\label{l-v}
H^3(b_v -a_v) \= \ch_1\!.\,H^2 -b_v\ch_0H^3\ >\ 0.
    \end{equation}
    Notice the right hand inequality is immediate from $v>0$ when $\ch_0=0$; when $\ch_0>0$ it is the statement that both intersection points are to the left of $\Pi(v)$. 

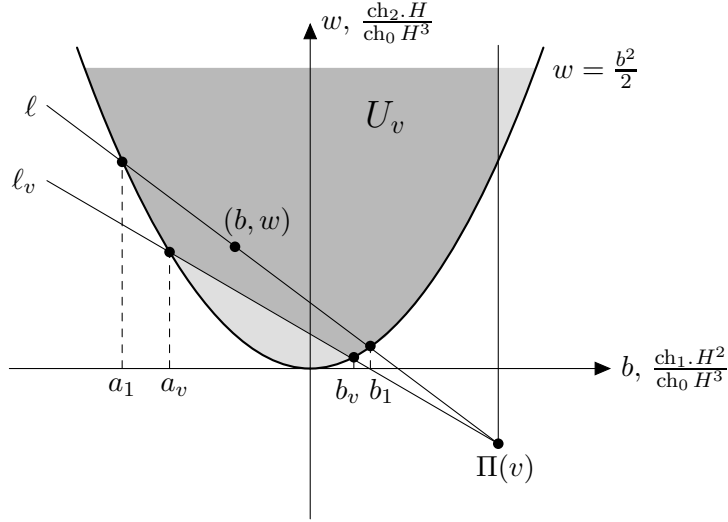
\begin{figure}[h]
	\begin{centering}
		\definecolor{zzttqq}{rgb}{0.27,0.27,0.27}
		\definecolor{qqqqff}{rgb}{0.33,0.33,0.33}
		\definecolor{uququq}{rgb}{0.25,0.25,0.25}
		\definecolor{xdxdff}{rgb}{0.66,0.66,0.66}
		
		\begin{tikzpicture}[line cap=round,line join=round,>=triangle 45,x=1.0cm,y=1.0cm]
		
		\draw[->,color=black] (-4,0) -- (4,0);
		\draw  (4, 0) node [right ] {$b,\,\frac{\ch_1\!.\;H^2}{\ch_0H^3}$};


		\fill [fill=gray!25!white] (0,0) parabola (3,4) parabola [bend at end] (-3,4) parabola [bend at end] (0,0);
		
			    \fill[fill=gray!55!white](-3, 4)--(-2.5,2.75)--(-2, 1.8)--(-1.87,1.55)--(.58, .13)--(1, .42)--(1.4,.85)--(2, 1.75)--(2.5, 2.75)--(2.5, 4); 
		
		\draw[thick]  (0,0) parabola  (3.1,4.27); 
		\draw[thick]  (0,0) parabola (-3.1,4.27); 
		\draw  (3.8 , 3.6) node [above] {$w= \frac{b^2}{2}$};
		
		\draw[->,color=black] (0,-2) -- (0,4.6);
		\draw  (0.9, 4.2) node [above ] {$w,\,\frac{\ch_2\!.\;H}{\ch_0H^3}$};
		
		\draw[color=black, dashed] (.58, 0)-- (.58, .15);
		\draw[color=black, dashed] (.8, 0)-- (.8, .3);
		\draw[color=black] (2.5,-1) -- (2.5,4.3);
		\draw [dashed, color=black] (-1.87,0) -- (-1.87,1.55);
		\draw [dashed, color=black] (-2.5,0) -- (-2.5,2.7);
		\draw [color=black] (2.5, -1) -- (-3.5, 2.5);
		\draw [color=black] (2.5, -1) -- (-3.5, 3.5);

		\draw  (-2.5,0) node [below] {$a_1$};
		\draw  (-1.8,0) node [below] {$a_v$};
		\draw  (1, 3) node [above] {\Large{$U_v$}};
		\draw  (.5, 0) node [below] {$b_v$};
		\draw  (.95, 0) node [below] {$b_1$};
		\draw  (2.6,-1) node [below] {$\Pi(v)$};
	    \draw  (-3.5, 2.5) node [left] {$\ell_v$};
		\draw  (-3.5, 3.5) node [left] {$\ell$};
		\draw  (-.7,1.62) node [above] {$(b,w)$};
		
		\begin{scriptsize}
		\fill (.58, .15) circle (2pt);
		\fill (-1.87,1.55) circle (2pt);
		\fill (2.5,-1) circle (2pt);
		\fill (-2.5,2.75) circle (2pt);
		\fill (.8,.3) circle (2pt);
		\fill (-1,1.62) circle (2pt);
		
		\end{scriptsize}
		
		\end{tikzpicture}
		
		\caption{The safe area $U_v$ when $\ch_0 >0$}
		
		\label{safe area}
		
	\end{centering}
\end{figure}
We define the safe area $U_v\subset U$ of $v$ to be the set of $(b,w) \in U$ strictly above $\ell_v$ and to the left of $\Pi(v)$ as in Figure \ref{safe area}. Explicitly, when $\Delta_H(v)\ge0$,
$$
U_v\ :=\ \big\{(b,w)\in U\colon\,b<\mu_H(v),\ 2w>(a_v+b_v)b-a_vb_v\big\},
$$
where
$$
a_v\=\frac{-\sqrt{1+\ch_0}}{1+\ch_0+\sqrt{1+\ch_0}}\frac{\sqrt{\Delta_H}}{H^3}+\overline\beta,\qquad
b_v\=\frac1{1+\ch_0+\sqrt{1+\ch_0}}\frac{\sqrt{\Delta_H}}{H^3}+\overline\beta
$$
and $\overline\beta$ is defined as in \cite[Equation (16)]{BMS},
$$
\overline\beta\ :=\,\left\{\!\!\begin{array}{ll}
\frac{\ch_2\!.H}{\ch_1\!.H^2} & \text{if }\ch_0=0, \\
\frac{\ch_1\!.H^2-\sqrt{\Delta_H}}{\ch_0H^3} & \text{if }\ch_0>0.
\end{array}\right.
$$
The equality \eqref{l-v} will ensure that walls $\ell\cap U$ in the safe area are sufficiently long that we can apply our standard argument \eqref{estimate} to show that semistable objects $E$ of class $v$ (and their destabilising quotients) must be sheaves.

\begin{Prop}\label{prop. general-sheaf}
If $E\in\cA_{\;b}$ is $\nu\_{b,w}$-semistable object of class $v>0$ and $(b,w) \in U_v$ then
\begin{itemize}
\item $E$ is a sheaf, and
\item if $E_1 \hookrightarrow E \twoheadrightarrow E_2$ is a short exact sequence in $\cA_{\;b}$ with $\nu\_{b,w}(E_i)=\nu\_{b,w}(E)$ then the $E_i$ are sheaves and $(b,w) \in U_{\ch(E_i)}$ for $i=1,2$. 
\end{itemize}
\end{Prop}
\begin{proof}
Consider the line $\ell$ which passes through $(b,w)$ and $\Pi(v)$ if $\ch_0 > 0$, or passes through $(b,w)$ and has slope $\frac{\ch_2\!.\;H}{\ch_1\!.\;H^2}$ if $\ch_0 = 0$. Since $\ell$ lies \emph{strictly above} $\ell_v$, we can run our standard argument \eqref{estimate} with 
$E,\ell_v,a_v,b_v$ in place of $F,\tilde{\ell}, \tilde{a},\tilde{b}$. If $r':=\ch_0(\cH^{-1}(E))\ge1$ it gives
$$
\ch_1\!.\;H^2\ >\ \ch_0b_vH^3 +r'(b_v - a_v)H^3\ \ge\ \ch_0b_vH^3 +(b_v - a_v)H^3,  
$$
contradicting \eqref{l-v}. Since $\cH^{-1}(E)$ is torsion-free for $E\in\cA_{\;b}$ \eqref{Abdef} we conclude  $\cH^{-1}(E)=0$ and $E$ is a sheaf.\medskip

Therefore the long exact sequence of cohomology sheaves of $E_1 \hookrightarrow E \twoheadrightarrow E_2$ is
$$
    0 \To \cH^{-1}(E_2)
    \To E_1 \To E\To \cH^{0}(E_2)\To 0.
$$
If $r''\coloneqq \rk\;(\cH^{-1}(E_2))\ge1$ then \eqref{estimate} applied to $E_2,\ell_v,a_v,b_v$ in place of $F,\tilde{\ell},\tilde{a}, \tilde{b}$ gives 
\begin{align*}
\ch_1(E_2).H^2\ & >\ 
 r''(b_v-a_v)H^3+\(\!\ch_0-\ch_0(E_1)\)b_vH^3
 \\
&\ge \ (b_v-a_v)H^3+\ch_0b_vH^3-\ch_0(E_1)b_vH^3.
\end{align*}
But $E_1\in\cA_{\;b_v}$ so $\ch_1(E_1).H^2\ge\ch_0(E_1)b_vH^3$, giving
\begin{align*}
\ch_1\!.\,H^2\ &=\ \ch_1(E_1).H^2 +\ch_1(E_2).H^2\ \ge\ \ch_0(E_1)b_vH^3+\ch_1(E_2).H^2 \\
&>\ (b_v - a_v)H^3+\ch_0b_vH^3.
\end{align*}
This contradicts \eqref{l-v}, so $r''=0$. Since $\cH^{-1}(E_2)$ is torsion-free it is zero and $E_2$ is a sheaf. Thus both $E_i$ are $\nu\_{b,w}$-semistable sheaves with $\nu\_{b,w}(E_i)=\nu\_{b,w}(E)<+\infty$ so have positive classes $[E_i]>0$ in the sense of Definition \ref{+ve}. \medskip

Finally set $v_i:=\ch(E_i)$. Since $E_i$ is a sheaf in $\cA_{\;b}$ it has $\ch_0(E_i)\ge0$ and $\ch_1(E_i).H^2-b\ch_0H^3\ge0$ by \eqref{65}. Thus $(b,w)$ is to the left of $\Pi(E_i)$, and to prove $(b,w)\in U_{v_i}$ it remains to show the line $\ell$ joining $(b,w),\,\Pi(E)$ and $\Pi(E_1)$ is above $\ell_{v_i}$.

Suppose not, so $\ell_{v_i}$ is strictly above $\ell_v$. Thus
\beq{NA}
b_{v_i}\ >\ b_v \quad\text{and}\quad b_{v_i} -a_{v_i}\ >\ b_v- a_v.
\eeq
Since $\ch_0(E_i) \geq 0$, this gives 
\beq{truro}
    H^3(b_{v_i} -a_{v_i}) \= \ch_1(E_i).H^2 -b_{v_i}\ch_0(E_i)H^3\ \leq\ \ch_1(E_i).H^2 -b_{v}\ch_0(E_i)H^3
\eeq
But since $\ell$ lies above $\ell_v$, we see $E_j\in\cA_{\;b_v}$, where $\{i,j\}=\{1,2\}$. Thus $\ch_1^{b_vH}(E_j).H^2\ge0$, and is actually $>0$ since $\nu\_{b\_v,w}(E_j)=\nu\_{b\_v,w}(E)<+\infty$. That is,
\begin{equation*}
\ch_1(E_i).H^2 - b_v H^3\ch_0(E_i)\ <\ \ch_1(E).H^2 - b_v H^3\ch_0(E).
\end{equation*}
Combined with \eqref{truro} we get
$$
H^3(b_{v_i} -a_{v_i})\ <\ \ch_1(E).H^2 -b_{v}\ch_0(E)H^3 \= H^3(b_v-a_v),
$$
contradicting \eqref{NA}.
\end{proof}

For the next Lemma recall the map $\ch_H=(\ch_0H^3,\ch_1\!.\;H^2,\ch_2\!.\;H)\colon K(X)\to\Q^3$ of \eqref{chH}.

\begin{Lem}\cite[Corollary 3.10]{BMS}\label{lem:quadratic}
Fix $(b,w) \in U$. If $E_1 \hookrightarrow E \twoheadrightarrow E_2$ is an $\cA_{\;b}$-exact sequence of $\nu\_{b,w}$-semistable objects, all of the same slope $\nu\_{b,w}<+\infty$, then either
\begin{align*}
\mathrm{(i)}\ &0\,\le\,\Delta_H(E_i)\,<\,\Delta_H(E) \quad\mathrm{for}\ i=1, 2,\ \mathrm{or}\\
\mathrm{(ii)}\ &\Delta_H(E_i)\,=\,0\,=\,\Delta_H(E) \quad\mathrm{and\ the}\,\ch_H(E_i)\ \mathrm{are\ proportional\ to\,}\ch_H(E)\ \mathrm{for}\ i=1, 2.
\end{align*}
\end{Lem}

\begin{proof} Since $\Delta_H$ and $\nu\_{b,w}$ factor through $\ch_H$ we consider them as functions on $\Q^3$. Thus $\Delta_H$ defines a quadratic form of signature $(2,1)$ on $\R^3$.

Consider the 2-dimensional subspace $V\subset\R^3$ on which either $\nu\_{b,w}$ equals $\nu\_{b,w}(E)$ or both its numerator and denominator vanish. The vector $v:=(1,b,w)\in V$ satisfies
\begin{itemize}
\item $\Delta_H(v)<0$, so $\Delta_H$ has signature (1,1) on $V,$ and
\item the function $\ch_1\!.\;H^2-b\ch_0\!H^3$ (the denominator of $\nu\_{b,w}$) vanishes on $v$.
\end{itemize}
Thus the line $\R v$ splits $V$ --- and the positive cone $\{\Delta_H|_V\ge0\}$ --- into two components on which $\ch_1\!.\,H^2-b\ch_0H^3$ takes different signs. Let $\mathcal C^+$ denote the component of the cone $\{\Delta_H|_V\ge0\}$ on which $\ch_1\!.\,H^2-b\ch_0H^3\ge0$.

By  \cite[Theorem 3.5]{BMS} both $\ch_H(E_i)\in\mathcal{C}^{+}$. Thus the Lemma follows from an elementary geometric fact about the bilinear form $\langle \,\cdot\,, \,\cdot\, \rangle$ associated to $\Delta_H|_V$: if $v_1, v_2 \in \mathcal{C}^{+} -\{0\}$ then $\langle v_1, v_2 \rangle \geq 0$, with equality if and only if $v_1$ and $v_2$ are proportional and $\Delta_H(v_i) =0$.  
\end{proof}

We now describe all walls and semistable factors for $\nu\_{b,w}$-semistable objects of class $v_n$.

If the object is not a sheaf then by Theorem \ref{Thm. large n-arbitrary rank} it takes the form \eqref{fake}; it is therefore semistable on $\ell_{\js}\cap U$ and unstable on $U\smallsetminus\ell_{\js}$.

So now consider a $\nu\_{b,w}$-semistable \emph{sheaf} $E$ of class $v_n$ and let $\ell_E$ be a wall of instability for $E$. We will find three cases. In the first case, $(b,w)$ will be on the Joyce-Song wall $\ell_{\js}\subset\R^2$ \eqref{ljsdef}.
Here we will find $E$ is destabilised by a unique sequence $F\to E\to T(-n)[1]$ coming from a Joyce-Song pair \eqref{JSs}, up to tensoring by some $T\in\Pic\_0(X)$. 

In the other two cases the destabilising objects are \emph{sheaves}. In case 2(a) below their safe areas contain $(b,w)$, which will allow us to continue crossing walls. In case 2(b) one of the sheaves has no further walls and the other is of type $v_n$ for a strictly lower rank $\le r-2$, which will allow us to do an induction.
\medskip

So fix a $\nu\_{b,w}$-semistable sheaf $E$ of class $v_n$ and $(b,w) \in\ell_E\cap U$ to the left of $\Pi(E)$, i.e. $b<\frac n{r-1}$. Suppose $E_1 \hookrightarrow E \twoheadrightarrow E_2$ is a destabilising short exact sequence in $\cA_{\;b}$ and assume that $E_2$ is  semistable \emph{next to the wall}.\footnote{\label{nextto}I.e. $E_2$ is $\nu\_{b,w+\epsilon}$-semistable or $\nu\_{b,w-\epsilon}$-semistable for all sufficiently small $0\le\epsilon\ll1$. The point is we can avoid $E_2$ being an object like \eqref{fake} by replacing it by the final quotient of its Jordan-H\"older filtration.}

\begin{Thm}\label{thm:all walls}
Under the above conditions, one of the following holds.
\begin{enumerate}
\item[\emph{(1)}] $E_1$ is a sheaf in class $\mathsf v,\ E_2 \cong T(-n)[1]$ for some $T\in\Pic\_0(X)$ and $(b,w)\in$ $\ell_{\js}$. On or above $\ell_{\js}$ 
there are no walls for sheaves of class $\alpha[\mathsf v]\in K_H(X)$ for $\alpha\in\big[\frac1r,1\big]$.
\item[\emph{(2)}] Both $E_i$ are sheaves with positive classes $[E_i]>0$ and either
    \begin{enumerate}
        \item[\emph{(a)}] $(b,w)$ lies in the safe areas $U_{[E_i]}$ for $i=1,2$, or
        \item[\emph{(b)}] one of the $E_i$ has $\ch_1(E_i).H^2 = 0$ and sheaves of class $\alpha[E_i]\in K_H(X)$ have
no walls on or above $\ell_E$ for $\alpha\in(0,1]$. The other sheaf $E_j$ has class of type $v_n$ for $\rk\le r-2$ --- that is, $\ch(E_j)+\ch(\cO_X(-n))$ satisfies bounds \eqref{general bound}.\footnote{That is, there are constants $p'_1,p'_2,q'$ depending on $r$ and the original bounds $p_1,p_2,q$. Recall we chose $n\gg0$ at the beginning \eqref{nzero} sufficiently large that our arguments work for all rank $r$ classes $\mathsf v$ satisfying the bounds \eqref{general bound} \emph{and} for all smaller rank classes $\mathsf v'$ satisfying the bounds $p'_1,p'_2,q'$ in \eqref{general bound}.}
    \end{enumerate}
\end{enumerate}
\end{Thm}

\begin{proof}
Note $E_1$ is a sheaf because $E$ is, but $E_2$ need not be.
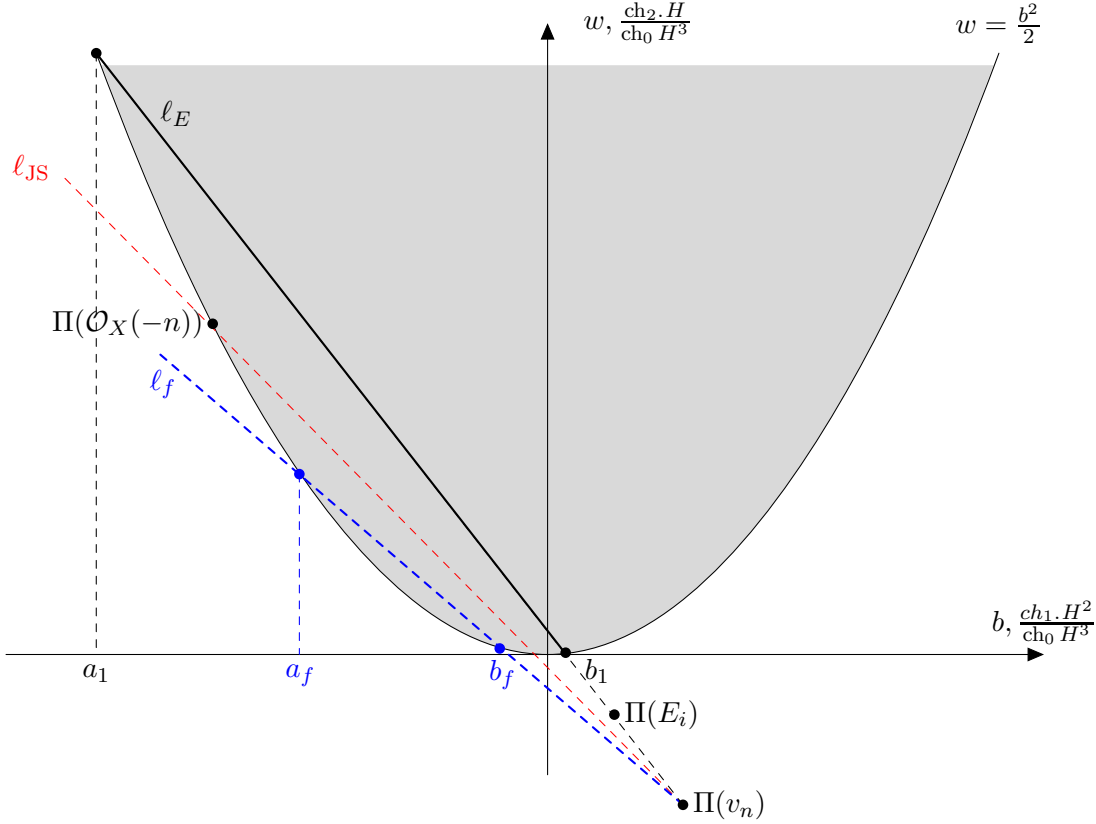
\begin{figure}[h]
	\begin{centering}
		\definecolor{zzttqq}{rgb}{0.27,0.27,0.27}
		\definecolor{qqqqff}{rgb}{0.33,0.33,0.33}
		\definecolor{uququq}{rgb}{0.25,0.25,0.25}
		\definecolor{xdxdff}{rgb}{0.66,0.66,0.66}
		
		\begin{tikzpicture}[line cap=round,line join=round,>=triangle 45,x=1.2cm,y=0.8cm]
		
		\draw[->,color=black] (-6,0) -- (5.5,0);
		\draw  (5.5, 0) node [above] {$b, \frac{ch_1.\;H^2}{\ch_0H^3}$};
		\fill [fill=gray!30!white] (0,0) parabola (4.94, 9.8) parabola [bend at end] (-4.94, 9.8) parabola [bend at end] (0,0);

		\draw  (0,0) parabola (5,10); 
		\draw  (0,0) parabola (-5,10); 
		\draw  (5 , 10) node [above] {$w=\frac{b^2}{2}$};

		\draw[->,color=black] (0,-2) -- (0,10.5);
		\draw  (1, 10) node [above ] {$w, \frac{\ch_2\!.\;H}{\ch_0H^3}$};
			\draw  (1.5, 6) node [above] {\Large{$U$}};
		\draw[dashed,color=red] (1.5, -2.5) -- (-5.4, 8);
		
		\draw[dashed,color=blue , thick] (1.5, -2.5) -- (-4.3, 5);
		\draw[dashed,color=black] (1.5, -2.5) -- (-5, 10);		
		\draw[thick, color=black] (-5, 10) -- (.2, .01);
		
		
		\draw [color=black,dashed] (-5, 10) -- (-5, 0);
		
		
		\draw[dashed, color=blue] (-2.75,3) -- (-2.75,0);
		\draw[color=blue] (-2.74,0) node [below]{$a_f$};
		\draw[color=blue] (-.5,.08) node [below]{$b_f$};
		
		\draw  (-3.7,5.5) node [left] {$\Pi(\mathcal{O}_X(-n))$};
		\draw  (1.5, -2.5) node [right ] {$\Pi(\vi_n)$};
		\draw  (.55, .1) node [below ] {$b_1$};
		\draw  (-5, 0) node [below] {$a_1$};
		\draw  (.74,-1) node [right ] {$\Pi(E_i)$};
		\draw[color=blue] (-4, 4.5) node [left ] {$\ell_f$};
		\draw  (-4.4, 9) node [right] {$\ell_E$};
		\draw[color=red]  (-5.4, 8.1) node [left] {$\ell_{\js}$};
		
		\begin{scriptsize}
				\fill (1.5, -2.5) circle (2pt);
		
		\fill (.2, .03) circle (2pt);
		\fill (-5,10) circle (2pt);
		\fill (-3.71,5.5) circle (2pt);
		
		\fill[color=blue] (-2.75,3) circle (2pt);
		\fill[color=blue] (-.53,.1) circle (2pt);
		
		\fill (.74,-1) circle (2pt);
		

		\end{scriptsize}
		
		\end{tikzpicture}
		
		\caption{Walls for objects of class $\vi_n$}
		
		\label{figure.walls for class v}
		
	\end{centering}
\end{figure}

\textbf{Case (1).} Suppose first that $E_2$ is not a sheaf. Then $E_2\in\cA_{\;b}$ implies $\cH^{-1}(E_2)$ is torsion-free, so it has rank $r'>0$. We have the exact sequence of cohomology sheaves
\beq{SES2}
    0 \To \cH^{-1}(E_2) \To E_1 \To E \To \cH^0(E_2) \To 0.
\eeq
Since $E_2$ lies on $\ell_E\cap U$ above or on $\ell_f$, our usual argument (\ref{bounds on bf}, \ref{esti1}) with $\epsilon=(4rH^3)^{-1}$ gives
\begin{equation*}
    \ch_1(\cH^{-1}(E_2)).H^2\ \le\ -nr'H^3 + \frac{r'}{4r} \quad \text{and} \quad \ch_1(\cH^0(E_2)).H^2\ \ge\ -\frac{\rk\;(\cH^0(E_2))}{4r}\,.
\end{equation*}
But $\rk\;(\cH^0(E_2))\le r$ so $\ch_1(\cH^0(E_2)).H^2 \geq 0$. Similarly $E_1\in\cA_{\;b}$ as $b\uparrow b_f$ gives
\begin{equation*}
  \ch_1(E_1).H^2\ \geq\ -\frac{\ch_0(E_1)}{4r}\ \ge\ -\frac{r'+r}{4r}\,.  
\end{equation*}
Combining these inequalities,
\begin{align*}
    nH^3\=\ch_1(E).H^2 \ =&\ \ch_1(E_1).H^2+ \ch_1(\cH^0(E_2)).H^2 -\ch_1(\cH^{-1}(E_2)).H^2\\
    \ge&\,\  nr'H^3-\frac{2r'+r}{4r}\,.
\end{align*}
Since $n \gg 0$ this shows $r'=1$. Therefore the last term $\frac{2r'+r}{4r}\in\(0,\frac34\big]$
and all the above inequalities between integers must be sharp, yielding
\begin{equation} \label{chern} 
    \ch_1(\cH^{-1}(E_2)).H^2 \= -nH^3 \quad \text{and} \quad \ch_1(\cH^{0}(E_2)).H^2 \= \ch_1(E_1).H^2 \=0.  
\end{equation}
Thus $1\le\ch_0(E_1) \leq r$ by \eqref{SES2}. Now Lemma \ref{final wall} implies $\Pi(E_1)\in\ell_E$ lies above or on $\ell_f$; substituting this into \eqref{equation of final wall} gives
$$
 - \frac{3nm +2n^2\beta.H +2(\beta.H)^2/H^3}{n^2rH^3 +2(r-1)\beta.H}\ \le\ \frac{\ch_2(E_1).H}{\ch_0(E_1)H^3}\ \stackrel{\eqref{BOG}}{\leq}\ 0.
$$
Since $n\gg0$ this gives the bounds
\begin{equation}\label{3nm}
-\tfrac1r\ch_0(E_1)(2\beta.H +1)\ <\ \ch_2(E_1).H\ \le\ 0.
\end{equation}

For $\alpha\in\big[\frac1r,1\big]$ we have $-\frac1{\alpha^2\ch_0(E_1)^2H^3}\in(a_f,b_f)$ \eqref{bounds on bf}, so $b$ takes this value along $\ell_f\cap U$ and therefore along $\ell_E\cap U$ too. Hence Lemma \ref{lem: bounding the first wall} shows that for any sheaf of class $\alpha[\mathsf v]\in K_H(X)$ \eqref{K_Hdef} there is no wall above, on or just below $\ell_{\js}$ and that
\begin{equation}\label{upper bound for ch3}
\ch_3(E_1)\ \leq\ \frac23\ch_2(E_1).H\left(\ch_0(E_1)^2\ch_2(E_1).H -\frac{1}{2H^3\ch_0(E_1)^2}\right).
\end{equation}
Therefore the class $\ch(E_2)$,
\begin{equation*}
\Big(r-1-\ch_0(E_1),\,nH+D -\ch_1(E_1),\,-\beta -\ch_2(E_1) -\tfrac12n^2H^2,\, -m -\ch_3(E_1) + \tfrac16n^3H^3\Big),
\end{equation*}
satisfies the bounds \eqref{general bound}. If it had rank $\ge0$ then by Theorem \ref{Thm. large n-arbitrary rank}
\beq{fake2}
E_2\ \cong\ F\oplus T(-n)[1],
\eeq
since it is not a sheaf. But \eqref{fake2} contradicts the semistability of $E_2$ next to the wall. So $\ch_0(E_2)=-1$ and $\ch_0(E_1)=r$.

By \eqref{bounds on bf} we see $b'=-n+1/H^3$ lies on $\ell_E$, so Lemma \ref{lem: bounding the first wall, rank -1} now applies to $E_2$, giving
\begin{align*}
   -nH(\beta+\ch_2(E_1))\ \leq\ \frac{2}{3}\left(\beta.H+\ch_2(E_1).H\right) \!\left(\!\beta.H + \ch_2(E_1).H - \frac{1}{2H^3}\!\right)  +m+\ch_3(E_1). 
\end{align*}
Since $n\gg0$ and $\ch_3(E_1)$ is bounded above by \eqref{upper bound for ch3} this means $\beta.H+\ch_2(E_1).H\ge0$. 
From
$$
\Pi(v_n)\=\tfrac1{r-1}\left(n,\,-\tfrac{\beta.H}{H^3}-\tfrac{n^2}{2}\right), \quad
\Pi(E_1)\=\left(0,\,\tfrac{\ch_2(E_1).H}{rH^3}\right), \quad
\Pi(\cO_X(-n))\=\left(-n,\,\tfrac{n^2}2\right)
$$
we see this is the condition that $\Pi(E_1)$ lies on or above the Joyce-Song wall $\ell_{\js}$ joining $\Pi(\cO_X(-n))$ and $\Pi(\vi_n)$. If it is strictly above, i.e. $\beta.H+\ch_2(E_1).H>0$, then $E_2\in\cA_{-n}$ since $\ell_{\js}$ intersects $\partial U$ when $b=-n$. Thus $\mu\_H(\cH^{-1}(E_2)) < -n$, contradicting \eqref{chern}. So in fact $\beta.H+\ch_2(E_1).H=0$ and by Lemma \ref{lem: bounding the first wall, rank -1} again $E_2 \cong T(-n)[1]$ for some $T\in\Pic\_0(X)$. \medskip

\textbf{Case (2).} Now assume that $E_2$ is sheaf, as are $E$ and $E_1$. Since $E$ is $\nu\_{b,w}$-semistable, $\ell_E$ lies above or on $\ell_f$. Hence by \eqref{bounds on bf} we know the vertical line $b= b_0 \coloneqq \frac{-1}{rH^3}<b_f$ intersects $\ell_E$ at a point $(b_0, w_0) \in U$. Therefore $E_1,E_2$ lie in $\cA_{\;b_0}$ so $\ch_1^{b_0H}\!.\,H^2\ge0$. Moreover their $\nu\_{b_0,w_0}$-slopes equal $\nu\_{b_0,w_0}(E)<+\infty$ since $b_0$ is to the left of $\Pi(E)$. Thus $\ch_1^{b_0H}(E_i).H^2>0$, giving
\begin{equation*}
   0\ <\  \ch_1(E_i).H^2 + \frac1r\ch_0(E_i)\ <\  nH^3 + \frac{r-1}r\,.
\end{equation*}
Since $0 \leq \ch_0(E_i) \leq r-1$ this shows $0 \leq \ch_1(E_i).H^2 \leq nH^3$. By 
Lemma \ref{lem:quadratic} and the Bogomolov-Gieseker inequality \eqref{quadratic form} we have $0 \leq \Delta_H(E_i) < \Delta_H(E)$ and  $B_{b_0,w_0}(E_i) \geq 0$. Since $\ch_1^{b_0H}(E_i).H^2>0$ this gives upper bounds on the $\ch_3(E_i)$, then $\ch_3(E_1)+\ch_3(E_2)=\ch_3(E)$ turns these into lower bounds as well. Now consider the stated cases (a),\,(b).

\begin{enumerate}
\item [(a)] If $\ch_1(E_i).H^2>0$ for $i=1,2$ then $\ch_1(E_i).H^2 =\(\!\ch_1(E)-\ch_1(E_j)\).H^2<nH^3$. We show $(b,w)\in U_{[E_i]}$ in this case by proving the following general result.

Suppose that
$F \in \cA_{\;b}$ has the same $\nu\_{b,w}$-slope as the $\nu\_{b,w}$-semistable sheaf $E$, where $(b,w)$ is strictly to the left of $\Pi(E),\,\Pi(F)$. Then
\beq{claim}
\qquad (b,w)\,\in\,U_{[F]} \quad\text{if}\quad \ch_1(F).H^2\,<\,nH^3 \quad\text{and}\quad 0\,\le\,\ch_0(F)\,\le\,r.
\eeq
Indeed, suppose for a contradiction that the line $\ell$ passing through $\Pi(E),\,\Pi(F)$ and $(b,w)$ lies below or on the line $\ell_{[F]}$ defined by \eqref{l-v}. Then  $\ell_{[F]}$ must lie above $\ell_f$ \eqref{equation of final wall}, so by \eqref{bounds on bf},
    \begin{equation*}
    a\_{[F]}\ <\ -n + \frac{1}{4rH^3} \qquad \text{and}\qquad    -\frac{1}{4rH^3}\ <\ b\_{[F]}\;.
    \end{equation*}
This implies
$$
\hspace{13mm}nH^3- \frac{1}{2r}\,<\,H^3\(b\_{[F]} -a\_{[F]}\)\,=\,\ch_1(F).H^2 -b\_{[F]}\ch_0(F)H^3\,\le\,\ch_1(F).H^2 + \frac14\,,
$$
giving the contradiction $\ch_1(F).H^2 > nH^3-1$.
    \item[(b)] Suppose $\ch_1(E_i).H^2 = 0$ so $\ch_1(E_j).H^2 = nH^3$. Since $\nu\_{b,w}(E_i)=\nu\_{b,w}(E)<+\infty$ we see $\ch_0(E_i)>0$. Since $E_j$ is also a sheaf, $\ch_0(E_i) \leq r-1$; in particular $r>1$ when there are type 2(b) walls. Thus $-\frac1{\alpha^2\ch_0(E_i)^2H^3}<b_f$ \eqref{bounds on bf} for $\alpha\in[\frac1{r-1},1]$ and $b$ takes this value along $\ell_E\cap U$. Hence Lemma \ref{lem: bounding the first wall} applies to $E_i$ and any semistable sheaf in class $\alpha[E_i]$ to show it has no walls above or on $\ell_E$. In particular $E_i$ is $\nu\_{b,w_+}$-semistable for all $w_+\ge w$. Moreover, just as in \eqref{3nm}, $\Pi(E_i)\in\ell_E$ lying above or on $\ell_f$ implies, for $n\gg0$,
\beq{marker}
 -\frac{2\beta.H +1}{rH^3}\ <\ \frac{\ch_2(E_i).H}{\ch_0(E_i)H^3}\ \stackrel{\eqref{BOG}}{\leq}\ 0.\phantom\qedhere
\eeq
This bounds $\ch_2(E_j).H$. Combined with the bound on $\ch_3(E_j)$ derived above, this  gives the bounds \eqref{general bound} with $p_1,p_2,q$ replaced by appropriate functions $p_1',p_2',q'$ of the constants $p_1,p_2,q$ we used in rank $r$.  $\hfill\square$
\end{enumerate}
\end{proof}


\begin{Rem}
In case 1 we say \emph{the wall is of type 1}, and similarly for 2(a) and 2(b). Note that a wall can be of more than one type simultaneously; in fact even a single $E$ may sit in exact sequences of different types at the same time.
\end{Rem}

%

\begin{Def}\label{gJH}
Fix a $\nu\_{b,w}$-semistable object $E$ and a filtration $0\subseteq E_0\subsetneq E_1\subseteq E$ in $\cA_{\;b}$, all of whose graded pieces are $\nu\_{b,w}$-semistable of the same slope $\nu\_{b,w}(E)<+\infty$. If $0\subseteq E_0$ and $E_1\subseteq E$ are not both equalities then we call $E_1/E_0$ a \emph{semistable factor} of $E$.

If furthermore $E_1/E_0$ is semistable next to the wall joining $(b,w)$ and $\Pi(E)$ (in the sense of Footnote \ref{nextto}) we call it a \emph{nondegenerate semistable factor} of $E$.
\end{Def}

In particular objects like \eqref{fake} could never be nondegenerate semistable factors of some $E$, but the summands of \eqref{fake} could be.

\begin{Cor}\label{ssf}
Given a $\nu\_{b,w}$-semistable sheaf $E$ of class $v_n$, each of its nondegenerate semistable factors is one of the following. The last two occur only on $\ell_{\js}$.
\begin{itemize}
\item A sheaf of rank $\le r-1$ with $(b,w)$ in its safe area.
\item A sheaf of type $v_n$ for some rank $\le r-2$.
\item A sheaf of class $\mathsf v$ which is $\nu\_{b,w'}$-semistable for all $w'\geq w$, with $(b,w)\in U_{\mathsf v}$.
\item $T(-n)[1]$ for some $T\in\Pic\_0(X)$.
\end{itemize}
\end{Cor}

\begin{proof}
Given a filtration $0\subseteq E_0\subsetneq E_1\subseteq E$ as in Definition \ref{gJH},
we claim that (i) $E_1=E$, or (ii) $E_1$ is a sheaf of type $v_n$ for some rank $<r-1$, or (iii) $[E_1]=\mathsf v\in K(X)$ or (iv) $E_1$ is a sheaf of rank $\le r-1$ and $(b,w)\in U_{[E_1]}$.

If $E_1\into E\onto E/E_1$ does not create a wall for $E$ then $[E_1],\,[E/E_1],\,[E]\in K_H(X)$ must be proportional: $[E_1]=\alpha[E]$. By $\nu\_{b,w}$-semistability all have $\ch_1^{bH}\!.\;H^2\ge0$; therefore $\ch_1^{bH}\!.\;H^2(E_1)>0$ since this is true of $E$. Thus $\alpha\in(0,1]$. If $\alpha=1$ then $E_1=E$. Otherwise $\alpha\in\big[\frac1{r-1},\frac{r-2}{r-1}\big]$ so $\rk(E_1)\in[1,r-2],\ \ch_1\!.H^2<nH^3$ and $(b,w)\in U_{[E_1]}$ by \eqref{claim}.

So we suppose now that $E_1\into E\onto E/E_1$ does create a wall for $E$. If $E/E_1$ is not a sheaf then it is either 
\begin{itemize}
\item $T(-n)[1]$ for some $T\in\Pic\_0(X)$ by Theorem \ref{thm:all walls}\;(1), or
\item an object of the form \eqref{fake} of type $v_n$ with rank less than $r-1$, by \eqref{fake2} in the proof of Theorem \ref{thm:all walls}\;(1).
\end{itemize}
\noindent In both cases we are on $\ell_{\js}$ and $\ch_1(E_1).H^2 = 0$. Therefore $E_1$ is $\nu\_{b,w'}$-semistable for all $w' \geq w$ by Lemma \ref{lem: bounding the first wall} and $(b,w)\in U_{\mathsf [E_1]}$ by \eqref{claim}. 

Otherwise $E_1$ and $E/E_1$ are sheaves of rank $\le r-1$ by Theorem \ref{thm:all walls}\;(2). In case 2(a) we know $(b,w)\in U_{[E_1]}$.
In case 2(b) there are two possibilities for $E_1$ and $E/E_1$. One is that $E_1$ has $\ch_1\!.\;H^2=0,\,\ch_0>0$, but then $(b,w)\in U_{[E_1]}$ by \eqref{claim}. The other is that $E_1$ is of type $v_n$ for a lower rank $<r-1$. This proves our claim.\medskip

If $E_0=0$ then one of (ii),\,(iii),\,(iv) holds, proving the Corollary. Now assume $E_0 \neq 0$. In cases (iii),\,(iv) Proposition \ref{prop. general-sheaf} shows $E_1/E_0$ is a sheaf with $(b,w)\in U_{[E_1/E_0]}$. To check it has rank $\le r-1$ we note $E_0\ne0$ cannot be of rank zero in case (iii) because $E_1$ is torsion-free by \cite[Lemma 2.7(c)(i)]{BMS}.


In cases (i),\,(ii) the class of $E_1$ in $K(X)$ is of type $v_n$ so we can run the first half of the proof again for the destabilising quotient $E_1/E_0$ of $E_1$ in place of the destabilising subobject $E_1$ of $E$.

We conclude that (i) $E_1/E_0$ is $0$, or (ii) $E_1/E_0$ is a sheaf of type $v_n$ for some rank $<r-1$, or (ii$'$) an object of type $v_n$ of the form \eqref{fake}, or (iii) $E_1/E_0\cong T(-n)[1]$ or (iv) $E_1/E_0$ is a sheaf of rank $\le r-1$ and $(b,w)\in U_{[E_1/E_0]}$. And by the definition of nondegenerate semistable factor we can rule out case (i) and (ii$'$).
\end{proof}

So we have particularly good control on the type 2 walls of Theorem \ref{thm:all walls}. The type 1 wall crossing can be shown to describe a component of the moduli space of $\nu\_{b,w}$-semistable objects of class $v_n$. It parameterises \emph{strictly $\nu\_{b,w}$-stable sheaves} which are all cokernels of Joyce-Song stable pairs up to tensoring by some $T\in\Pic\_0(X)$. Since we will not strictly need the full details of this result for our wall crossing, we relegate it to Theorem \ref{jsthm} in Appendix \ref{JSapp}.

\subsection{Bogomolov-Gieseker conjecture}\label{BGsec}
In this section we note that in place of the full Bogomolov-Gieseker conjecture we can make do with the much weaker \ref{wBG} below, at least when $X$ is Calabi-Yau. As before, fix a class $v \in K(X)$ of rank $r > 0$ and set $v_n = v- [\cO_X(-n)]$ for any $n \gg 0$. In this paper we have applied Conjecture \ref{conjecture} only
\begin{enumerate}
	\item to objects $E$ of type $v_n$ to find $\ell_f$ in Lemma \ref{final wall},
	\item to the rank $r \geq 1$ tilt semistable sheaves $F$ of Lemma \ref{lem: bounding the first wall}, at $b_0 = \mu\_H(E)-\frac{1}{r^2H^3}$ and $w \downarrow \frac12b_0^2$, and
	\item to the destabilising objects $E_i$ of Theorem \ref{thm:all walls}, at points $(b,w)\in U_{[E_i]}$ of the safe area of $E_i$ on the wall $\ell \cap U$ along which $E_i$ and $v_n$ have the same slope. 
\end{enumerate}

Referring to Figure \ref{fig-ell-v} we define $\ell_{v, n}$ to be the lowest possible line --- that is, the line of the maximal possible (negative) gradient --- through $\Pi(v_n)$ which intersects $\partial U$ at two points with $b$-values $b_1<b_2 < \mu\_H(v)$ satisfying
\begin{equation*}
b_1\ \leq\ -n+\tfrac{1}{4r^2H^3}\qquad  \text{and} \qquad b_2\ \geq\ \mu\_H(v)-\tfrac{1}{4r^2H^3}\,. 
\end{equation*}

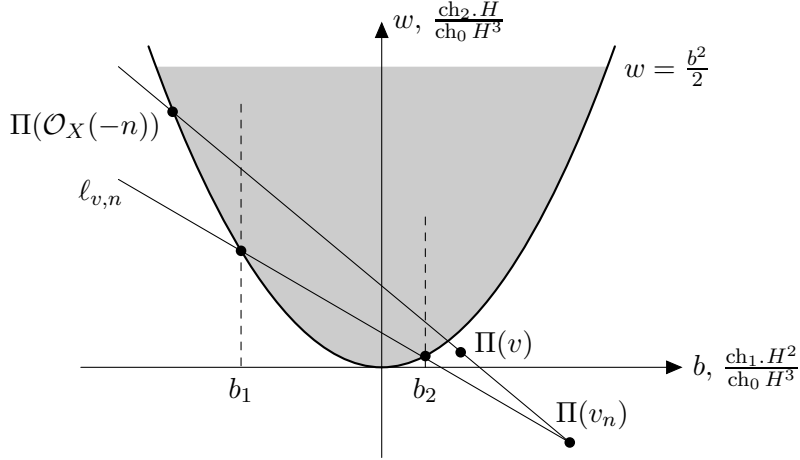
\begin{figure}[h]
	\begin{centering}
		\definecolor{zzttqq}{rgb}{0.27,0.27,0.27}
		\definecolor{qqqqff}{rgb}{0.33,0.33,0.33}
		\definecolor{uququq}{rgb}{0.25,0.25,0.25}
		\definecolor{xdxdff}{rgb}{0.66,0.66,0.66}
		
		\begin{tikzpicture}[line cap=round,line join=round,>=triangle 45,x=1.0cm,y=1.0cm]
		
		\draw[->,color=black] (-4,0) -- (4,0);
		\draw  (4, 0) node [right ] {$b,\,\frac{\ch_1\!.\;H^2}{\ch_0H^3}$};
		
		\fill [fill=gray!40!white] (0,0) parabola (3,4) parabola [bend at end] (-3,4) parabola [bend at end] (0,0);
		
		\draw[thick]  (0,0) parabola  (3.1,4.27); 
		\draw[thick]  (0,0) parabola (-3.1,4.27); 
		\draw  (3.8 , 3.6) node [above] {$w= \frac{b^2}{2}$};
		
		\draw[->,color=black] (0,-1.2) -- (0,4.6);
		\draw  (0.9, 4.2) node [above ] {$w,\,\frac{\ch_2\!.\;H}{\ch_0H^3}$};
		
		\draw[color=black, dashed] (.58, 2)-- (.58, 0);
		\draw [dashed, color=black] (-1.87,3.5) -- (-1.87, 0);
		\draw [color=black] (2.5, -1) -- (-3.5, 2.5);
		\draw [color=black] (2.5, -1) -- (-3.5, 4);
		
		\draw  (-1.87,0) node [below] {$b_1$};
		\draw  (.58, 0) node [below] {$b_2$};
		\draw  (2.8,-.3) node [below] {$\Pi(v_n)$};
		\draw  (-3.3, 2.3) node [left] {$\ell_{v, n}$};
		\draw   (-2.8,3.2) node [left] {$\Pi(\cO_X(-n))$};
		\draw   (1.1,.3) node [right] {$\Pi(v)$};
		
		\begin{scriptsize}
		\fill (.58, .15) circle (2pt);
		\fill (-1.87,1.55) circle (2pt);
		\fill (2.5,-1) circle (2pt);
		\fill (-2.78,3.4) circle (2pt);
		\fill (1.05,.2) circle (2pt);

		\end{scriptsize}
		
		\end{tikzpicture}
		\caption{The line $\ell_{v, n}$}\label{fig-ell-v}		
		
	\end{centering}
\end{figure}

To handle part (a) above we only need the following weakening of Conjecture \ref{conjecture}.\bigskip

\begin{enumerate}[label=\textbf{(wBG)}]
\item\label{wBG} \emph{There is no $\nu\_{b,w}$-semistable object of class $v_n$ for $(b,w) \in \ell_{v, n} \cap U$.}
\end{enumerate}

\bigskip

For part (b) we need \emph{any} upper bound for $\ch_3(F)$, rather than the explicit upper bound of Lemma \ref{lem: bounding the first wall}. The following result --- proved at the end of this section --- provides this.

\begin{Thm}\label{BG2} Suppose $c_1(X)=0$. 
Fix $(r,c,s) \in \N \oplus \Z \oplus \Q$ with $c>0$ if $r=0$. There exists $m \in \mathbb{Z}$ depending on $(r,c,s)$ such that $\ch_3(F) \leq m$ for all $\mu\_H$-semistable sheaves $F$ with $\ch_H(F) = (r,c,s)$.
\end{Thm}

To handle part (c) we extend Theorem \ref{BG2} from $\mu\_H$-semistable sheaves to $\nu\_{b,w}$-semistable objects. We again take $c_1(X)=0$ and fix $(r,c,s)\in \N \oplus \Z \oplus \Q$  with $c>0$ if $r=0$. 

\begin{Cor}
There exists $\widetilde{m}\in \mathbb{Z}$ depending on $(r,c,s)$ such that $\ch_3(F) \leq \widetilde{m}$ for all $\nu\_{b,w}$-semistable $F \in\cA_{\;b}$ with $\ch_{H}(F) = (r,c,s)$ and $(b,w) \in U_{[F]}$. 
\end{Cor}
\begin{proof}
	We prove the claim by induction on $\Delta = c^2-2rs$. If $\Delta = 0$ then by Lemma \ref{lem:quadratic} there are no walls for objects $F$ of class $\ch_H(F) =(r,c,s)$. Thus we are already in the large volume chamber and $F$ is a tilt semistable sheaf. Consider its Harder-Narasimhan filtration with respect to Gieseker stability. Each Gieseker semistable factor has $\ch_H$ proportional to $(r,c,s)$ so can take only finitely many possible values. Applying Theorem \ref{BG2} to each and summing over all possible filtrations provides an upper bound for $\ch_3$.
		
	If $\Delta > 0$, then we know there are finitely many walls in the safe area $U_{[F]}$ for classes $F$ of fixed $\ch_H(F)$. This determines finitely many possible destabilising classes $\ch_H(F_i)$ along the walls, each with strictly smaller discriminant $\Delta(F_i)<\Delta(F)$. By induction on $\Delta$ we can assume the $\ch_3(F_i)$ are already bounded; summing gives the bound for $\ch_3(F)$. 
\end{proof}

Therefore the results of Sections \ref{shvs} and \ref{subshvs} (and the proof of Theorem \ref{1} in the next section) all hold for smooth projective Calabi-Yau threefolds satisfying \ref{wBG}.

We note that this applies to any smooth quintic 3-fold, and to (2,4) complete intersections in $\PP^{\;5}$.
In these cases Conjecture \ref{conjecture} has been proved for a restricted subset of weak stability conditions: those $(b,w)$ satisfying 
\begin{equation}\label{in for b, w}
w\ >\ \frac{1}{2} b^2 + \frac{1}{2}\big(b - \lfloor b \rfloor\big)\big (\lfloor b \rfloor+1 - b \big). 
\end{equation}
See \cite[Theorem 2.8]{Li} and \cite[Theorem 2.8]{Liu} respectively.  
In particular if $b\in\Z$ then Conjecture \ref{conjecture} holds for any $w > \frac12b^2$. But the bounds \eqref{bounds on bf} show that $b$ takes an integer value $b_0$ along $\ell_f\cap U$ \eqref{equation of final wall}, so Conjecture \ref{conjecture} holds at $(b_0,w)$ for the class $v_n$. Thus there is no $\nu\_{b_0,w}$-semistable object of class $v_n$ when $(b_0, w)$ lies below $\ell_f$. Since $\ell_{v,n}$ lies below $\ell_f$, this proves \ref{wBG}. 

A weaker form of Conjecture \ref{conjecture} is also proved in \cite[Theorem 1.2]{Ko20} for certain Calabi-Yau threefold double and triple covers of $\PP^{\;3}$. This similarly suffices to prove \ref{wBG} --- and thus Theorem \ref{1} --- for these manifolds. (While Koseki's version of the Bogomolov-Gieseker inequality includes additional terms, the line $\ell_f(E)$ of \eqref{equation of final wall} --- where the inequality becomes an equality for objects $E$ of class $v_n$ --- still has equation $w = xb + y$ with coefficients $x\sim-\frac{n}{2},\ y\sim$ constant as $n\to\infty$. Thus the argument \eqref{bounds on bf} is unaffected.)

\medskip

\noindent\textbf{Proof of Theorem \ref{BG2} for $r>0$.}
The reflexive hull $F^{**}$ of $F$ sits in the exact sequence
\begin{equation}\label{exact}
    0 \To F \To F^{**} \To Q \To 0
\end{equation}
with $\dim Q\le1$. Thus $\ch_2(Q).H\ge0$ and $\ch_1(F^{**})=\ch_1(F)$, so by the $\mu\_H$-semistability of $F^{**}$ the classical Bogomolov inequality gives
\beq{interval}
    s\ \leq\ \ch_2(F).H+\ch_2(Q).H\=\ch_2(F^{**}).H\,\stackrel{\eqref{BOG}}{\leq}\,\frac{c^2}{2r}\,.
\eeq
Thus $\ch_H(F^{**})$ lies in a finite set depending only on $\ch_H(F)$. Since $\mathrm{Td}_1(X)=0$ we now have uniform bounds --- depending only on $\ch_H(F)$ --- on the leading 3 coefficients of the Hilbert polynomial of $F^{**}$. But by \cite[Theorem 4.4]{langer}, $\mu\_H$-semistable reflexive sheaves with such bounds form a bounded family.
Thus we deduce a uniform bound for $\ch_3(F^{**})$.

By \eqref{interval} the leading coefficient $\ch_2(Q).H$ of the Hilbert polynomial of $Q$ lies in the interval $\big[0,\frac{c^2}{2r}\big]$. But $Q$ is a quotient of the fixed sheaf $\cO_X(-N)^{\oplus M}$, so a result of Grothendieck \cite[Lemma 2.6]{Gr} gives a lower bound on the subleading term of its Hilbert polynomial. Since $\mathrm{Td}_1(X)=0$ this is
$$
\ch_3(Q)\ \geq\ C.    
$$
By our uniform bound on $\ch_3(F^{**})$ this gives the required upper bound on $\ch_3(F)=\ch_3(F^{**})-\ch_3(Q)$.
\medskip

\noindent\textbf{Proof of Theorem \ref{BG2} for $r=0$.}
Take a $\mu\_H$-semistable pure sheaf $F$ of dimension 2 with $\ch_H(F)=(0,c,s)$ and $c>0$. To prove an upper bound on $\ch_3(F)$ we may as well assume that $\ch_3(F)\ge0$.
Since $\mathrm{Td}_1(X)=0$ the Hilbert polynomial of $F$ is
$$
\chi(F(t))\=\tfrac12ct^2+st+\big[\!\ch_1(F).\mathrm{Td}_2(X)+\ch_3(F)\big].
$$
Its leading and subleading coefficients are fixed by $\ch_H(F)$. If we can bound $\ch_1(F).\mathrm{Td}_2(X)$ then combined with $\ch_3(F)\ge0$ this bounds below the constant term in square brackets. Thus by \cite[Theorem 4.4]{langer} $F$ lies in a bounded family, giving an upper bound for $\ch_3(F)$ depending only on $\ch_H(F)$, as required. 

The key point is that $\ch_1(F)=D$ is an \emph{effective} divisor, so its degree is comparable to its size measured in any norm $|\,\cdot\,|$ on $H^2(X,\R)$: there are constants $c_i>0$ such that
$$
c_1|D|\ \le\ D.H^2\ \le\ c_2|D|.
$$
(See the proof of \cite[Corollary 7.3.3]{BMT}, for instance.) Thus $|D|\le c/c_1$ and $D$ lies in a finite subset of $H^2(X,\Z)$, implying the claimed uniform bound on $D.\mathrm{Td}_2(X)$.

\section{Enumerative invariants and wall crossing formulae}\label{DT3}

We have now proved the technical results we need to prove Theorem \ref{1} by wall crossing. Next we review Joyce theory: the enumerative invariants we will use, and the wall crossing formula relating them as we change stability condition.

From now on we assume $(X,\cO_X(1))$ is a smooth complex projective Calabi-Yau 3-fold: $K_X\cong\cO_X$. Without loss of generality we assume also that $H^1(\cO_X)=0$; when this does not hold the locally free action of Jac$\;(X)$ on moduli spaces of sheaves of rank $\ge1$ forces all invariants to vanish. In particular then $\Pic\_0(X)=H^2(X,\Z)_{\mathrm{tors}}$.

\subsection*{Abelian categories}
Let $\cA$ denote one of the abelian categories Coh$(X)$ or $\cA_{\;b}$ \eqref{Abdef}. 
Both have the same Grothendieck group $K(\mathrm{Coh}(X))$, numerical Grothendieck group $K(X)$ \eqref{Kdef}, and quotient $K_H(X)$ \eqref{K_Hdef}.
In contrast the positive cone $C(\cA)\subset K(X)$ defined by
\beq{CA}
C(\cA)\ :=\ \big\{[F]\in K(X)\colon F\in\cA\big\}
\eeq
depends on $\cA$. To deal with strictly semistable objects in $\cA$ we would like to apply Joyce's Ringel-Hall algebra technology \cite{Jo2,Jo3,BrHall,BR}. To do this we must verify that $\cA$ satisfies a number of conditions.

Firstly, $[a]=0\in C(\cA)$ implies $a=0\in\cA$. For $\cA=\mathrm{Coh}\;(X)$ this is implied by the injectivity of the Chern character on numerical K-theory $K(X)$, and for $\cA=\cA_{\;b}$ it is \cite[Lemma 3.2.1]{BMT}. Secondly Coh$\;(X)$ is Noetherian because $X$ is, while $\cA_{\;b}$ is Noetherian by \cite[Proposition 5.22]{BMT}. Both have finite dimensional Homs.

In \cite[Section 9]{Jo1} Joyce shows that Coh$(X)$ satisfies \cite[Assumption 7.1]{Jo1}. For $\cA_{\;b}$ it follows from \cite[Proposition B.4]{BCR}
if \cite[Assumption B.1]{BCR} holds: namely that the stack of objects of $\cA_{\;b}$ is open in the stack of objects $E\in\cD(X)$ with $\Ext^{<0}(E,E)=0$. But this is proved in \cite{TodaK3}; see Proposition \ref{prop:pi-to-heart} for more details. Finally \cite[Assumption 8.1]{Jo1} has two parts; the first is used to handle weak stability conditions like purity whose moduli stacks are not of finite type, so we do not need it --- the permissibility of \eqref{filtrate} below and its weakening Proposition \ref{finsum} will be sufficient for our purposes. The second is that $\pi_1\times\pi_2$ in \eqref{Cart} below has finite type. This is proved for $\cA_{\;b}$ or Coh$(X)$ by using a flattening stratification over which $\pi_1\times\pi_2$ has a locally constant fibre of the form $\Ext^1(E_2,E_1)/\Hom(E_2,E_1)$ as in \cite[Proposition 6.2]{BrHall} for instance.

\subsection*{Hall algebra}
We denote Joyce's Hall algebra $\mathrm{SF}_{\!\mathrm{al}}\;(\cA)$ by $H(\cA)$. It is constructed in \cite{Jo2} starting with the $\Q$-vector space on the following generators: isomorphism classes of representable morphisms from (algebraic stacks of finite type over $\C$ with affine stabilisers) to (the stack of objects of $\cA$). Joyce then quotients by the scissor relations for closed substacks, and makes the result into a ring with his Hall algebra product $*$ on stack functions.

At the level of individual objects, the product $1_E*1_F$ of (the indicator functions of) $E$ and $F$ is the stack of all extensions between them,
\beq{extstack}
\frac{\Ext^1(F,E)}{\Aut(E)\times\Aut(F)\times\Hom(F,E)}\,.
\eeq
Here $e\in\Ext^1(F,E)$ maps to the corresponding extension in $\cA$ of $F$ by $E$, and elements of $\Hom(F,E)$ act as automorphisms of this extension in the obvious way.\footnote{Write the extension as $0\to E\rt\iota G\rt\pi F\to 0$. Then the action of any $\phi\in\Hom(F,E)$ on $G$ via id$_{\;G}+\iota\circ\phi\circ\pi$ preserves both $E\into G$ and $G\to\hspace{-2.5mm}\to F$.} More generally $*$ is defined via the stack $\mathfrak{Ext}$ of all short exact sequences
\beq{exten}
0\To E_1\To E\To E_2\To0
\eeq
in $\cA$, with its morphisms $\pi_1,\pi,\pi_2\colon\mathfrak{Ext}\to\cA$ taking the extension to $E_1,\,E,\,E_2$ respectively. This defines the universal case, which is the Hall algebra product of $\cA$ with itself:
$$
1_{\cA}*1_{\cA}\ =\ \Big(\mathfrak{Ext}\rt\pi\cA\Big).
$$
Other products are defined by fibre product with the universal case: given two stack functions $U,V\to\cA$ we define $U*V\to\cA$ by the Cartesian square
\beq{Cart}
\xymatrix@C=35pt{
U*V \ar[r]\ar[d]& \mathfrak{Ext} \ar[r]^-\pi\ar[d]_{\pi_1\!}^{\!\times\pi_2}& \cA \\
U\times V \ar[r]& \cA\times\cA\,.\!\!\!}
\eeq
Joyce uses weak stability conditions $\tau$ on $\cA$ that are \emph{permissible}
in the sense of \cite[Definition 4.7]{Jo3}. That is
\begin{enumerate}
\item[(i)] the stacks $\cM^{ss}_\tau(\alpha)$ of $\tau$-semistable objects of fixed class $\alpha\in K(X)$ are algebraic of finite type, and
\item[(ii)] $\cA$ is $\tau$-Artinian: any filtration
\beq{filtrate}
\hspace{5mm} \dots\,\subsetneq\,E_2\,\subsetneq\,E_1\,\subsetneq\,E_0\,:=\,E\,\text{ in }\,\cA,\,\text{ with }\,\tau(E_{n+1})\,\ge\,\tau(E_n/E_{n+1})\,\text{ for all }\,n,
\eeq
is automatically finite.
\end{enumerate}
Both conditions fail\;\footnote{For instance,
\emph{all} sheaves supported in dimension $\le1$ are in $\cA_{\;b}$ and are $\nu\_{b,w}$-semistable of slope $+\infty$.} for the weak stability conditions $\nu\_{b,w}$ on $\cA_{\;b}$, but in mild ways that we can manage. Firstly, $\cA_{\;b}$ admits finite $\nu\_{b,w}$-Harder-Narasimhan filtrations by \cite[Lemma 3.2.4]{BMT}, so we do not need to use (ii) to produce them (as Joyce does). For his other applications of permissibility the following will be sufficient. Recall the notion of semistable factors from  Definition \ref{gJH}.

\begin{Prop}\label{finsum} The following weakening of permissibility holds for $\(\cA_{\;b},\nu\_{b,w}\)$,
\begin{enumerate}
\item[\emph{(i)}] the stacks $\cM^{ss}_\tau(\alpha)$ of $\nu\_{b,w}$-semistable objects of fixed class $\alpha\in K(X)$ are algebraic of finite type when $\nu\_{b,w}(\alpha)<+\infty$, and
\item[\emph{(ii)}] if $b\in\Q$ and $\nu\_{b,w}(E_n)<+\infty$ for all $n$ in \eqref{filtrate} then the filtration is finite.
\end{enumerate}
Moreover, for any $(b,w)\in U$ and $\alpha\in K(X)$ with $\nu\_{b,w}(\alpha)<+\infty$, there are only finitely many $K(X)$ classes of semistable factors of $\nu\_{b,w}$-semistable objects in class $\alpha$.
\end{Prop}

\begin{proof}
We relegate the proof of (i) to Theorem \ref{ftstack} in Appendix \ref{TomB}. For (ii), the $\nu\_{b,w}<+\infty$ condition implies $\ch_1^{bH}>0$ for $E$ and all $E_n,\,E_n/E_{n+1}\in\cA_{\;b}$. Since $b\in\Q$ there is an $\epsilon>0$ such that $\ch_1^{bH}(F)>\epsilon$ for any $F\in\cA_{\;b}$ with $\ch_1^{bH}>0$. Therefore if $E_n \neq 0$,
$$
\ch_1^{bH}(E)\=\ch_1^{bH}(E_n)+\ch_1^{bH}(E_{n-1}/E_n)+\dots+\ch_1^{bH}(E_0/E_1)\ >\ (n+1)\;\epsilon
$$
bounds $n$, as required.

For the final claim, assume for a contradiction there are infinitely many ways to write $\alpha\in K(X)$ as a sum of classes $\alpha_1 , \dots, \alpha_n\in K(X)$ with $\nu\_{b,w}(\alpha_i) = \nu\_{b,w}(\alpha)<+\infty$ and $\cM^{ss}_{b,w}(\alpha_i) \neq \emptyset$. Then we can generate an unbounded family of $\nu\_{b,w}$-semistable objects of class $\alpha$ by taking direct sums of objects in $\cM^{ss}_{b,w}(\alpha_i)$, contradicting (i).  
\end{proof}

By (i) the inclusion
\beq{delt}
1_{\cM^{ss}_\tau(\alpha)}\ \colon\,\cM^{ss}_\tau(\alpha)\ \Into\ \cA
\eeq
defines an element of $H(\cA)$. In the case of $\(\cA_{\;b},\nu\_{b,w}\)$ we assume $\nu\_{b,w}(\alpha)<+\infty$ to get the same result. To ameliorate the stabilisers of strictly semistable objects, Joyce replaces this indicator stack function by the logarithm of the indicator function of the stack of all semistable sheaves of the same slope,
\beq{epsi}
\epsilon_\tau(\alpha)\ :=\ \mathop{\sum_{m\ge 1,\ \alpha_1,\dots,\;\alpha_m\;\in\,C(\cA)\,:}}_{\sum_{i=1}^m \alpha_i\,=\,\alpha,\ \tau(\alpha_i)\,=\,\tau(\alpha)\ \forall i}\frac{(-1)^m}m\ 1_{\cM^{ss}_\tau(\alpha_1)}*\cdots*1_{\cM^{ss}_\tau(\alpha_m)}\,.
\eeq
For permissible weak stability conditions $\tau$ Joyce shows this is a finite sum using \cite[Proposition 4.9]{Jo3}. The last statement of Proposition \ref{finsum} gives the same finiteness for $\(\cA_{\;b},\nu\_{b,w}\)$ when we assume $\nu\_{b,w}(\alpha)<+\infty$.

\subsection*{Invariants}
Joyce defines a set of \emph{virtually indecomposable} stack functions with algebra stabilisers $\mathrm{SF}_{\mathrm{al}}^{\mathrm{ind}}(\cA)\subset \mathrm{SF}_{\mathrm{al}}(\cA)$, which we denote $H^{\mathrm{ind}}(\cA)\subset H(\cA)$, and prove it is a sub Lie algebra under the Hall algebra commutator. He then proves the really deep result \cite[Theorem 8.7]{Jo3} that $\epsilon_\tau(\alpha)$ \eqref{epsi} lies in it,
$$
\epsilon_\tau(\alpha)\ \in\ H^{\mathrm{ind}}(\cA)\ \subset\ H(\cA).
$$
Morally this means it has only $\C^*$ stabilisers. Precisely, by \cite[Proposition 3.4]{JS}, \eqref{epsi} is equal in the Hall algebra to a $\Q$-linear combination
\beq{BC*}
\sum_i\alpha_i\(Z_i/\C^*,\rho_i\),\quad\alpha_i\in\Q,
\eeq
where $Z_i$ is a quasi-projective \emph{variety} with trivial $\C^*$-action and $\rho_i$ is a representable morphism $Z_i/\C^*\to\cA$. Thus we can define the generalised DT invariant
\beq{Jdef}
\mathsf J\;_\tau(\alpha)\ :=\ -\sum_i\;\alpha_i\;e\(Z_i,\rho_i^*\chi^B\)\=-\sum_{i,n}\alpha_in\;e\((\rho_i^*\chi^B)^{-1}\{n\}\)\ \in\ \Q \vspace{-1mm}
\eeq
counting $\tau$-semistable objects of class $\alpha$ independently of choices as in  \cite[Equation 5.5]{JS}. Here $\chi^B\colon\cA\to\Z$ is the Kai Behrend function \cite{Be} on the stack of objects of $\cA$ and $e$ denotes topological Euler characteristic weighted by the constructible function $\rho_i^*\chi^B$. (The extra sign in \eqref{Jdef} is to cancel the Behrend function $-1$ of $\C^*$.) If there are no strictly semistable objects of class $\alpha$ and the coarse moduli space $\cM^{ss}_\tau(\alpha)$ can be proved to be projective then it carries a symmetric obstruction theory whose dimension 0 virtual cycle has length $\J_\tau(\alpha)\in\Z$. \medskip

In the case of $\(\cA_{\;b},\nu\_{b,w}\)$ this recipe defines $\mathsf J\;_{b,w}(\alpha)$ when $\nu\_{b,w}(\alpha)<+\infty$. We can also consider stability conditions on $\cA=$\,\;Coh$(X)$. Given a sheaf of class $\alpha\in K(X)$ write its Hilbert polynomial as
$$
P_\alpha(t)\ :=\ \chi\(\alpha(t)\)\=a_dt^d+a_{d-1}t^{d-1}+\dots+a_0,
$$
where $d\le3$ and $a_d\ne0$. Its \emph{reduced} Hilbert polynomial is
$$
p_\alpha(t)\ :=\ \frac{P_\alpha(t)}{a_d}\,.
$$
Following Joyce \cite[Section 4.4]{Jo3} we introduce a total order $\prec$ on monic polynomials like $p_\alpha(t)$ by saying $p\prec q$ if and only if
\begin{itemize}
\item[(i)] deg$\,p>\ $deg$\,q$, or
\item[(ii)] deg$\,p=\,\;$deg$\,q$ and $p(t)<q(t)$ for $t\gg0$.
\end{itemize}
(Notice the strange direction of the inequality in (i), ensuring that sheaves of lower dimension have \emph{greater} slope.)
Then $E\in\mathrm{Coh}\;(X)$ is called Gieseker (semi)stable if for all non-trivial exact sequences $0\to A\to E\to B\to0$ in Coh$\;(X)$ we have
$$
p\_{\;[A]}\ \,(\preceq)\,\ p\_{\;[B]}.
$$
Here $(\preceq)$ means $\prec$ for stability and $\preceq$ (which is $\prec$ or $=$) for semistability. In particular (i) ensures that Gieseker semistable sheaves are pure. 

Discarding the constant term of the Hilbert polynomial before dividing by its top coefficient as before gives
\beq{tidef}
\widetilde p_\alpha(t)\ :=\ p_\alpha(t)-\frac{a_0}{a_d}\=
t^d+\dots+\frac{a_1}{a_d}t.
\eeq
It depends only on the class of $\alpha$ in the group $K_H(X)$ of \eqref{K_Hdef}.
Then \emph{tilt (semi)stability} on Coh$(X)$ is defined by the inequalities
$$
\widetilde p\_{\;[A]}\ \,(\preceq)\,\ \widetilde p\_{\;[B]}
$$
for all non-trivial exact sequences of sheaves $0\to A\to E\to B\to0$. By an argument from \cite[Proposition 14.2]{Br} it turns out there is a large volume chamber $w\gg0$ for $\alpha$ in which $\nu\_{b,w}$-(semi)stability is the same as tilt (semi)stability for $\rk\;(\alpha)\ge1$ or $\rk\;(\alpha)=0$ and $\ch_1(\alpha).H^2>0$.

By \cite[Theorem 4.20]{Jo3} tilt and Gieseker stability are \emph{permissible weak stability conditions} on $\cA=\mathrm{Coh}(X)$ in the sense of \cite[Definition 4.7]{Jo3}. (For Gieseker stability we can drop the adjective \emph{weak}.)
Therefore \eqref{delt} and \eqref{epsi}, applied to tilt and Gieseker stability respectively, define elements
$$
1_{\cM^{ss}_{\mathrm{ti}}(\alpha)},\,1_{\cM^{ss}(\alpha)}\in H(\mathrm{Coh}\;(X))\ \text{ and }\ \epsilon\_{\mathrm{ti}}(\alpha),\,\epsilon(\alpha)\,\in\,H^{\mathrm{ind}}(\mathrm{Coh}\;(X))
$$
and the corresponding invariants
\beq{Jtilt}
\J_{\mathrm{ti}}(\alpha)\ \in\ \Q
\quad\text{and}\quad \J(\alpha)\ \in\ \Q 
\eeq
as in \eqref{Jdef}. If all Gieseker semistable sheaves of class $\alpha\in K^0(X)$ are Gieseker stable then $\J(\alpha)\in\Z$ is the original DT invariant defined in \cite{Th}.

\subsection*{Wall crossing formulae} Joyce shows how the invariants $\J_\tau(\alpha)$ defined above vary as we change the stability condition $\tau$.
So fix $\cA$ and three stability conditions $\tau_+,\tau\_0,\tau_-$ thereon. We may think of these as lying above, on and below a wall respectively, but it is important to note that we can swap $\tau_\pm$ or take either to equal $\tau\_0$ --- the wall crossing formula will still apply.

Under the conditions described in Section \ref{conditions} below, the wall crossing formula relates the invariants $\J_\pm$ counting $\tau_\pm$-semistable objects. It takes the general form
\beq{WCF}
\J_+(\alpha)\,=\,\J_-(\alpha)\ +\ \mathop{\sum_{m\ge2,\ \alpha_1,\dots,\;\alpha_m\,\in\,C(\cA),}}_{\sum_{i=1}^m\alpha_i\,=\,\alpha,\ \tau\_0(\alpha_i)\,=\,\tau\_0(\alpha)\,\forall i}C_{+,-}(\alpha_1,\dots,\alpha_m)\prod_{i=1}^m\J_-(\alpha_i).
\eeq
Here the $C_{+,-}(\alpha_1,\dots,\alpha_n)\in\Q$ are universal coefficients  depending only on the Mukai pairings $\chi(\alpha_i,\alpha_j)$ and the relative sizes of the set of slopes $\{\tau_\pm(\alpha_i)\}$. The term
$$
C_{+,-}(\alpha_1,\dots,\alpha_m)\prod_{i=1}^m\J_-(\alpha_i)
$$
is zero unless there is a $\tau\_0$-semistable object $E$ of class $\alpha$ with $m$ $\tau\_0$-semistable factors of classes $\alpha_1,\dots,\alpha_m$. The formula reflects the different Harder-Narasimhan filtrations of $E$ on the two sides of the wall (and then further filtrations of the semistable Harder-Narasimhan factors by semi-destabilising subobjects).

The formula comes from the identity \cite[Theorem 3.14]{JS},
expressing $\epsilon_{\tau_+}(\alpha)-\epsilon_{\tau_-}(\alpha)$ as a weighted sum of iterated commutators of terms $\epsilon_{\tau_-}(\alpha_i)\in H(\cA)$, where $\alpha_i$ are the classes of $\tau\_0$-semistable factors of $\alpha$. We then take $\chi^B$-weighted Euler characteristics of both sides as in \eqref{Jdef}. Denote this integration map $\Psi\colon H^{\mathrm{ind}}(\cA)\to\Q$, so that $\Psi(\epsilon_\tau(\beta))=-\J_\tau(\beta)$. Then by \cite[Theorem 5.14]{JS} it satisfies\footnote{\cite[Theorem 5.14]{JS} proves \eqref{psi} for $\cA=$ Coh$\;(X)$, but it also holds for $\cA=\cA_{\;b}$ since Toda \cite[Theorem 1.1]{TodaHall} showed that \cite[Theorem 5.11]{JS} is valid for the moduli stack of objects $E\in\cD(X)$ with $\Ext^{<0}(E,E)=0$.}
\beq{psi}
\Psi\big[\epsilon(\alpha),\epsilon(\beta)\big]\=
(-1)^{\chi(\alpha,\beta)}\chi(\alpha,\beta)\,\Psi\(\epsilon(\alpha)\)\Psi\(\epsilon(\beta)\)
\eeq
whenever $\epsilon(\alpha),\,\epsilon(\beta)\in H^{\mathrm{ind}}(\cA)$ are virtual indecomposable stack functions supported on the stack of objects of classes $\alpha,\,\beta\in K(X)$ respectively. Thus $\Psi$ turns the iterated commutators into products of invariants $\J_-(\alpha_i)$, hence giving the formula \eqref{WCF}.


The coefficients $C_{+,-}$ are given by complicated formulae. We will only need to know that when $m =2$, $\tau\_0(\alpha_1)=\tau\_0(\alpha_2),\ \tau_-(\alpha_1)>\tau_-(\alpha_2)$ and $\tau_+(\alpha_1)<\tau_+(\alpha_2)$ then
\beq{><}
C_{+,-}(\alpha_1,\alpha_2)+C_{+,-}(\alpha_2,\alpha_1)\=(-1)^{\chi(\alpha_1,\alpha_2)-1}\chi(\alpha_1,\alpha_2).
\eeq
In fact, in the notation of  \cite[Equation 3.8]{JS}, $-\;U(\alpha_1,\alpha_2;\tau_+,\tau_-)=1=U(\alpha_2,\alpha_1;\tau_+,\tau_-)$. Therefore, applying the integration map $\Psi$ to \cite[Equation 3.10]{JS} and using \eqref{psi} gives \eqref{><}.

\subsection{Conditions}\label{conditions} To prove \eqref{WCF} --- and in particular to show it contains only finitely many nonzero terms --- Joyce \cite{Jo3} assumes 
that (a) $\tau_\pm,\tau\_0$ are permissible and (b) $\tau\_0$ \emph{dominates} $\tau_\pm$ in the sense of \cite[Definition 4.10]{Jo3}. That is, $\tau_+,\tau\_0$ should satisfy
\beq{cond2}
\tau_+(\beta)\ \le\ \tau_+(\gamma)\ \ \so\ \ 
\tau\_0(\beta)\ \le\ \tau\_0(\gamma)
\eeq
for all $\beta,\,\gamma\in C(\cA)$, and similarly for $\tau_-$\;.

When $\tau_+=\tau_0$ is tilt stability and $\tau_-$ is Gieseker stability on $\cA=$ Coh$\;(X)$, both conditions hold; (a) by \cite[Theorem 4.20]{Jo3} and (b) is immediate from the definitions.

Now take $\cA=\cA_{\;b}$ and $\tau\_0=\nu\_{b,w_0}$ on a wall of instability $\ell\cap U$ for $\alpha\in K(X)$, with $\tau_\pm=\nu\_{b,w_\pm}$ either side of $\ell$. Then although (a) fails in general, we can replace it by Proposition \ref{finsum} when $\nu\_{b,w_0}(\alpha)<+\infty$ and $b$ is rational. This will be sufficient for our purposes.

Similarly (b) fails because there is a dense set of walls between any two distinct weak stability conditions $\nu\_{b,w_\pm}$ when we allow the classes $\beta,\,\gamma$ in \eqref{cond2} to vary, but the following weakening holds. Suppose $\nu\_{b,w_0}(\alpha)<+\infty$. By Proposition \ref{finsum} it has finitely many classes of $\nu\_{b,w_0}$-semistable factors; denote them $\alpha_1,\dots,\alpha_k$.

\begin{Prop}\label{close}
We may choose $(b,w_\pm)$ sufficiently close to $(b,w_0)$ that the only wall for $\alpha$ or $\alpha_i$ between (or containing) $(b,w_\pm)$ is $\ell\;\cap\;U$. In particular, \eqref{cond2} holds for $\beta,\gamma\in\{\alpha,\alpha_i\}$.
\end{Prop}

\begin{proof}
Since there are only finitely many classes $\alpha_1,\dots,\alpha_k$ of $\nu\_{b,w_0}$-semistable factors this follows immediately from the fact that the walls of instability for any fixed class are locally finite; see Proposition \ref{prop. locally finite set of walls}.
\end{proof}

Let $A:=\{\alpha_1,\dots,\alpha_k\}\subset C(\cA_{\;b})$. Taking $\nu\_{b,w_\pm}$-Harder-Narasimhan filtrations of $\nu\_{b,w_0}$-semi\-stable objects, Proposition \ref{close} gives the following two identities,
\beq{1*1}
1_{\cM^{ss}_{b,w_0}}(\alpha)\=\mathop{\sum_{m\ge1,\ a_1,\dots,\;a_m\,\in\,A,\ \sum_{i=1}^ma_i\,=\,\alpha,}}_{\nu\_{b,w_\pm}(a_1)\,>\,\dots\,>\,\nu\_{b,w_\pm}(a_m)}1_{\cM^{ss}_{b,w_\pm}}(a_1)*\cdots*1_{\cM^{ss}_{b,w_\pm}}(a_m).
\eeq
Thus, when $\nu\_{b,w_0}(\alpha)<+\infty$, we have the formula \cite[Equation 3.11]{JS} with only \emph{finitely many terms}. From this the wall crossing formula follows. That is, \eqref{1*1} can be inverted as in \cite[Equation 3.12]{JS} and then expressed in terms of $\epsilon_{b,w_+}(\alpha)$ and $\epsilon_{b,w_-}(a_i)$ to give \cite[Equation 3.10]{JS}. This can then be rewritten in terms of a finite number of iterated commutators as in \cite[Theorem 3.14]{JS}. Applying the integration map $\Psi$ then gives the wall crossing formula \eqref{WCF} just as before.



In particular, so long as we keep $b$ rational and to the left of $\frac n{r-1}$ at all times to ensure $\nu\_{b,w}<+\infty$ in class $v_n$ (and therefore in the classes of all its semistable factors on any wall) the wall crossing formula \eqref{WCF} holds in the four different settings we will apply it to in the next section:
\begin{enumerate}
\item[(1)] crossing the Joyce-Song wall $\ell_{\js}$,
\item[(2)] crossing other type 2 walls in $U$,
\item[(3)] passing from the large volume chamber $w\gg0$ to tilt stability,
\item[(4)] passing from tilt stability to Gieseker stability.
\end{enumerate} 
Only (1) involves non-trivial complexes of sheaves; (2,3,4) use sheaves alone. In (1,2) we work with $\(\cA_{\;b},\nu\_{b,w}\)$, then (3) takes us into Coh$\;(X)$ where (4) takes place.

\section{Crossing the walls}\label{r}

In this section we will use wall crossing formulae to prove Theorem \ref{1} for a fixed class $\mathsf v\in K(X)$ with
$$
\ch(\mathsf v)\=(r,D,-\beta,-m)\,\text{ such that }\,r\,>\,0\,\text{ and }\,D.H^2\,=\,0.
$$
(In (\ref{norm1}, \ref{norm2}) at the end of this section we explain how to generalise to arbitrary $D.H^2$.) Picking $p_1,p_2,m$ so that the bounds \eqref{general bound} holds, we then choose $n\gg0$ accordingly.

By \cite[Proposition 5.3]{Br.stbaility} each class $\alpha\in K(X)$ which is positive in the sense of Definition \ref{+ve} has a large volume chamber $w\gg0$ in which $\alpha$ has no walls. For $(b,w)$ in this chamber the invariants 
$\J_{b,w}(\alpha)$ are independent of $(b,w)$; we denote them by $\J_{b,\infty}(\alpha)$.

\subsection*{Safe wall crossing to large volume}
Our first wall crossing formula takes place in the abelian category $\cA=\cA_{\;b}$. Fix a positive $\alpha\in C(\cA)\subset K(X)$ in the sense of Definition \ref{+ve}. We show that when $(b,w)$ lies in its safe area $U_\alpha$ the invariant $\J_{b,w}(\alpha)$ counting $\nu\_{b,w}$-semistable sheaves of class $\alpha$ can be expressed in terms of invariants $\J_{b,\infty}$ counting sheaves of rank$\,\le\rk(\alpha)$ in their large volume chambers.

\begin{Prop}\label{safety}
For $(b,w)\in U_\alpha$ there exists a universal formula
\beq{ufor}
\J_{b,w}(\alpha)\=\J_{b,\infty}(\alpha)+F_{b,w}\(\J_{b,\infty}(\beta_i)\)
\eeq
with $\beta_i$ a positive class of $\,\rk\le\rk(\alpha)$ for all $i$.
\end{Prop}

Here and throughout we stick to the following conventions.
\begin{itemize}
\item The universal formulae need not be invertible in any sense; for instance the function $F_{b,w}=0$ is a perfectly good universal formula.
\item The functions $F_{b,w}$ can change from line to line, as can the classes $\beta_i>0$; we only care that each such function is given by a polynomial with finitely many nonzero terms and coefficients depending only on $b,\,w$ and topological invariants of $X$.
\end{itemize}

\begin{proof}[Proof of Proposition \ref{safety}]
In the large volume chamber $w\gg0$ for $\alpha$ we have the required universal formula by setting $F\equiv0$.

If $\Delta_H(\alpha)<0$ there is nothing to prove because $\J_{b,w}(\alpha)=0$ by \eqref{BOG}.

If $\Delta_H(\alpha)=0$ there are no walls of instability for $\alpha$ by Lemma \ref{lem:quadratic} and the Bogomolov inequality \eqref{BOG}, so its large volume chamber is all of $U$ and the result follows.

So we may assume $\Delta_H(\alpha)>0$ and induct on it by assuming we have proved the Proposition for all classes with $\Delta_H<\Delta_H(\alpha)$.

Now we move our stability condition downwards through $U$ from the large volume chamber to the point $(b,w)$ of the Proposition. By the local finiteness of walls for $\alpha$ of Proposition \ref{prop. locally finite set of walls}, we cross only finitely many walls $\ell_i$. Fixing one and choosing $(b,w_0)$ on it, by Proposition \ref{close} we may pick $(b,w_\pm)\in U_\alpha$ just above and below it so that there are no further walls for $\alpha$ and its $\nu\_{b,w_0}$-semistable factors $\alpha_i$ between (or containing) $(b,w_\pm)$. By a further induction on the $\ell_i$ above $(b,w_+)$ we may assume we have a universal formula \eqref{ufor} for $(b,w_+)$ and it now suffices to show this gives one for $w_-$. (Notice that if $(b,w)$ in the Proposition lies on a wall for $\alpha$ then we take $w_-=w_0$ in the last step of the induction, but the rest of the proof is unchanged.)

Setting $\tau_\pm:=\nu\_{b,w_\mp}$ and $\tau\_0:=\nu\_{b,w_0}$ in \eqref{WCF} gives a formula
\beq{safeWCF}
\J_{b,w_-}(\alpha)\=\J_{b,w_+}(\alpha)+F\(\J_{b,w_+}(\alpha_i)\),
\eeq
with $F$ a universal function in classes $\alpha_i$ of $\tau\_0$-semistable factors of $\alpha$. By Proposition \ref{prop. general-sheaf} and Lemma \ref{lem:quadratic} these are classes of sheaves of rank$\,\le\rk(\alpha)$ with $\Delta_H(\alpha_i)<\Delta_H(\alpha)$ and $(b,w_+)\in U_{\alpha_i}$. By our first induction assumption this means Proposition \ref{safety} applies to the $\alpha_i$ so the $\J_{b,w_+}(\alpha_i)$ can all be written as functions of a finite number of terms $\J_{b,\infty}(\beta_j)$ in positive classes $\beta_j$ of $\rk\le\rk(\alpha)$. And by our second induction assumption, so can $\J_{b,w_+}(\alpha)$. Thus \eqref{safeWCF} is the required universal formula for $\J_{b,w_-}(\alpha)$.
\end{proof}

\subsection*{Induction on rank} 
We work with the class $v_n:=\mathsf v-[\cO(-n)]$.
Before dealing with its type 1 wall crossing over $\ell_{\js}$ we first set up another induction --- this time on $r=\rk(\mathsf v)$ --- to deal with the type 2 wall crossings of Theorem \ref{thm:all walls}.\footnote{By Appendix \ref{2} we could simplify things in rank $r=2$ since the type 2 walls for $v_n$ all lie strictly above $\ell_{\js}$.} We will prove we have a universal formula
\beq{forminduc}
\J_{b,w}(v_n)\=F_{b,w}\(\J_{b,\infty}(\alpha_i)\) \quad\text{in positive classes }\alpha_i\text{ of }\rk\,\le\,\rk\;(v_n)
\eeq
for each $(b,w)\in U$ with $b<\frac n{r-1}$ outside of a finite number of straight lines in $U$. (These will be the finitely many walls for $v_n$ and its semistable factors. There is also a (different) universal formula on these walls, but we do not need it.)

First recall from \eqref{bounds on bf} that there are points of $U$ below $\ell_f$ \eqref{equation of final wall}, so by Lemma \ref{final wall} we get the universal formula $\J_{b,w}(v_n)=0$ below the lowest wall for $v_n$. Similarly in the large volume chamber $w\gg0$ for $v_n$ we have the universal formula $\J_{b,w}(v_n)\=\J_{b,\infty}(v_n)$.

When $\rk\;(\mathsf v)=1$, so $\rk(v_n)=0$, this gives the universal formulae \eqref{forminduc} \emph{in all of} $U\smallsetminus\ell_{\js}$ by the results of \cite{FT2}. There we showed objects of class $v_n$ have only one wall; they are all strictly unstable below $\ell_{\js}$ and are both tilt and Gieseker stable rank 0 pure dimension two sheaves above it. This proves the base case of our induction.

So now we assume \eqref{forminduc} holds for $\rk(\mathsf v)\le r-1$ and prove it when $\rk(\mathsf v)=r$. We have already seen it is true in the large volume chamber, and below $\ell_f$. From either of these chambers we move through $U$ towards $\ell_{\js}$, never crossing it and always keeping $b<\frac n{r-1}$. By the local finiteness of walls for $v_n$ (Proposition \ref{prop. locally finite set of walls}), we cross 
only a finite number of walls of types 2(a) and 2(b) in Theorem \ref{thm:all walls}.
So we need only show that crossing a single wall of type 2(a) or 2(b) preserves the existence of a universal formula \eqref{forminduc}.

\subsection*{Type 2 walls}
Suppose $(b,w_1)$ and $(b,w_2)$ are either side of a wall for $v_n$ other than $\ell_{\js}$. By Theorem \ref{thm:all walls} it has type 2. The wall crossing formula \eqref{WCF} takes the form
\beq{2a}
\J_{b,w_2}(v_n)\=\J_{b,w_1}(v_n)+F\(\J_{b,w_1}(\alpha_i)\),
\eeq
where the $\alpha_i$ are all possible semistable factors of semistable sheaves of class $v_n$ on the wall. By Corollary \ref{ssf} these takes two forms.

The first is classes of type $v_n$ for a smaller rank $\le r-2$. By our induction assumption we have the formula \eqref{forminduc} for their invariants $\J_{b,w_1}$, expressing them in terms of invariants $\J_{b,\infty}(\beta_j)$ of positive classes of rank $\le r-2$.

Secondly we get positive classes $\alpha_i$ of rank $\le r-1$ with $(b,w)\in U_{\alpha_i}$. Proposition \ref{safety} expresses them in terms of more invariants $\J_{b,\infty}(\beta_j)$ of positive classes of rank $\le r-1$.

Substituting these into \eqref{2a} gives the wall crossing formula
\beq{2b}
\J_{b,w_2}(v_n)\=\J_{b,w_1}(v_n)+F\(\J_{b,\infty}(\beta_j)\)\ \ \text{where }\beta_j>0\text{ has }\rk\,\le\,r-1\ \,\forall\;j.
\eeq
This completes our induction and proves \eqref{forminduc} for $\rk(v_n)=r-1$.

\subsection*{Joyce-Song wall}
We can now cross the Joyce-Song wall. Let $(b,w_\pm)$ be points just above and below $(b,w_0)\in\ell_{\js}$. We write the wall crossing formula \eqref{WCF} as
\beq{1JS}
\J_{b,w_+}(v_n)\=\J_{b,w_-}(v_n)+(-1)^{\chi(\mathsf v(n))-1}\chi(\mathsf v(n))\cdot\#H^2(X,\Z)_{\mathrm{tors}}\cdot\J_{b,\infty}(\mathsf v)+\dots.
\eeq
Here we have included one of the $m=2$ terms of \eqref{WCF} with $\alpha_1,\,\alpha_2$ equal to $\big[\cO(-n)[1]\big],\,\mathsf v$. These are the classes of the semistable factors $T(-n)[1]$ and a $\nu\_{b,w_-}$-semistable sheaf $F$ that appear on a type 1 wall in Theorem \ref{thm:all walls}. Since $F$ has no walls either on or above $\ell_{\js}$ we have $\J_{b,w_-}(\mathsf v)=\J_{b,\infty}(\mathsf v)$; since $T(-n)[1]$ is rigid the counting invariant of its class $\big[\cO(-n)[1]\big]$ is
\beq{triv}
\J_{b,w_-}\big[\cO(-n)[1]\big]\=\#\Pic\_0(X)\=\#H^2(X,\Z)_{\mathrm{tors}}\;.
\eeq
Finally the coefficient $(-1)^{\chi(\mathsf v(n))-1}\chi(\mathsf v(n))$ is \eqref{><}.

Corollary \ref{ssf} gives a list of $K(X)$ classes $\alpha_i$ of possible nondegenerate semistable factors on $\ell_{\js}$. Since the classes of arbitrary semistable factors are sums of those of nondegenerate semistable factors, the list shows that a decomposition $\sum_{i=1}^m\alpha_i=v_n$ of $v_n$ is either $\mathsf v+\big[\cO(-n)[1]\big]$ or \emph{it does not involve $\mathsf v$.} Thus the extra term we included in \eqref{1JS} is the \emph{only} term in the wall crossing formula \eqref{WCF} for $v_n$ involving $\mathsf v$. All other terms involve only \eqref{triv} and terms $\J_{b,w_-}(\alpha_i)$ where $\alpha_i$ is either (i) a sheaf of rank $\le r-1$ with $(b,w_-)\in U_{\alpha_i}$, or (ii) a sheaf of type $v_n$ for some rank $\le r-2$.

Using Proposition \ref{safety} in case (i) and \eqref{forminduc} in case (ii) shows the remaining terms $\dots$ in \eqref{1JS} can be written as a universal function in invariants $\J_{b,\infty}(\beta_i)$ of positive classes $\beta_i>0$ of rank $\le r-1$.

Therefore rearranging \eqref{1JS} and replacing $\J_{b,w_\pm}(v_n)$ by their expressions \eqref{forminduc} yields a universal formula expressing $\J_{b,\infty}(\mathsf v)$ in terms of counting invariants of rank $<r$,
\beq{lastt}
\J_{b,\infty}(\mathsf v)\=F\(\J_{b,\infty}(\beta_j)\)\,\text{ with }\beta_j>0 \text{ of rank}\,\le\,r-1\ \,\forall\;j.
\eeq

\subsection*{Large volume to tilt}
Here there is no actual wall crossing: for a fixed positive class $\alpha\in K(X)$, tilt stability of sheaves in class $\alpha>0$ is equivalent to $\nu\_{b,w}$-stability for $w$ in its large volume chamber by \cite[Proposition 14.2]{Br}. That is for $w \gg 0$ (and $b < \mu\_H(\alpha)$ if $\rk\;(\alpha)>0$) we find $\J_{b,\infty}(\alpha)=\J_{\mathrm{ti}}(\alpha)$. Since \eqref{lastt} involves only finitely many classes we can take $w\gg0$ uniformly large for each of them, whereupon it becomes
\beq{poly3}
\J_{\mathrm{ti}}(\mathsf v)\=F\(\J_{\mathrm{ti}}(\beta_j)\)\,\text{ with }\beta_j>0 \text{ of rank}\,\le\,r-1\ \,\forall\;j.
\eeq

\subsection*{Tilt to Gieseker} 
Since the invariants in \eqref{poly3} count only \emph{sheaves}, we can now work in $\cA=$\,Coh$(X)$ for our final wall crossing from tilt stability to Gieseker stability.
Here tilt stability \emph{dominates} Gieseker stability in the sense of \cite[Definition 4.10]{Jo3}:
\beq{cond}
p_{\alpha}(t)\ \preceq\ p_{\beta}(t)\ \ \so\ \ \widetilde p_{\alpha}(t)\ \preceq\ \widetilde p_{\beta}(t) \qquad\forall\,\alpha,\,\beta\,\in\,K(X).
\eeq
They are both permissible (weak) stability conditions, so letting $\tau_+=\tau_0$ be the former and $\tau_-$ the latter, \eqref{WCF} gives a wall crossing formula
\beq{Gti}
\J_{\mathrm{ti}}(\alpha)\=\J(\alpha)+\mathop{\sum_{m\ge2,\ \alpha_1,\dots,\;\alpha_m\,\in\,C(\cA),}}_{\sum_{i=1}^m\alpha_i\,=\,\alpha,\ \widetilde p_{\alpha_i}(t)\,=\,\widetilde p_{\alpha}(t)\ \forall i}C(\alpha_1,\dots,\alpha_m)\prod_{i=1}^m\J(\alpha_i).
\eeq
The definition \eqref{tidef} of $\widetilde p(t)$ ensures that if $\rk(\alpha)>0$ then  $\rk\;(\alpha_i)>0$ in the above formula, so in fact $\rk\;(\alpha_i)\in[1,\,\rk(\alpha)-1]$. In particular, applying \eqref{Gti} to $\alpha=\mathsf v$ gives
$$
\J_{\mathrm{ti}}(\mathsf v)\=\J(\mathsf v)+F\(\J(\alpha_i)\)\,\text{ with }\,1\,\le\,\rk(\alpha_i)\,\le\,r-1\ \,\forall\;j.
$$
Applying it to the classes $\beta_j$ in \eqref{poly3} gives a similar formula but with $\alpha_i$ positive of rank $\le r-1$ in that case. Substituting both into \eqref{poly3} gives a new polynomial expression
\beq{atlast}
\J(\mathsf v)\=F\(\J(\alpha_i)\)\,\text{ with }
\alpha_i>0 \text{ of rank}\,\le\,r-1\ \,\forall\;i.
\eeq
We have finally expressed the rank $r$ invariant $\J(\mathsf v)$ in terms of invariants counting Gieseker semistable sheaves of $\rk\;(\alpha_i)\le r-1$. To finish up, we would like to prove the same for classes $\mathsf v$ with arbitrary $\ch_1(\mathsf v).H^2$.

\subsection*{Normalisation}
So far we have worked with classes $\mathsf v$ with $\ch_1(\mathsf v).\;H^2=0$. For an arbitrary $\mathsf v\in K(X)$ we reduce to this case by replacing any class $\ch\in H^{2*}(X,\Q)$ by $\ch^{tH}:=e^{-tH}\!\cdot\ch$. We calculate
\beq{norm1}
\nu\_{b,w}(\ch)\=\nu\_{b^t,w^t}\(\!\ch^{tH}\)+t,
\eeq
where $b^t:=b-t$ and $w^t:=w-bt+\frac12t^2$. Then fixing $t:=\ch_1(\mathsf v).\;H^2/rH^3$ gives
\beq{norm2}
\ch^{tH}_1(\mathsf v).\;H^2\=0.
\eeq
So replacing $\ch$ by $\ch^{tH}$ and $(b,w)$ by $(b^t,w^t)$ throughout the paper, (\ref{norm1}, \ref{norm2}) ensure we get all the same results for $\mathsf v$.\medskip

In particular we get a formula \eqref{atlast} for 
arbitrary classes $\mathsf v$ of rank $r$. By induction on rank we may assume each $\J(\alpha_i)$ on the right hand side has been written as a function of rank 0 dimension 2 classes, thus doing the same for the left hand side $\J(\mathsf v)$. This proves Theorem \ref{1}.

\appendix
\section{Joyce-Song pairs}\label{JSapp}
Mochizuki \cite[Chapter 3]{Mo} and Joyce-Song \cite[Section 5.4]{JS} introduced certain stable pairs. Here we modify Joyce-Song's definition, replacing their Gieseker (semi)stability by $\nu\_{b,w}$-(semi)stability. We could also use tilt (semi)stability --- for $(b,w)$ just above $\ell_{\js}$ the two definitions give the same stable pairs by Theorem \ref{thm:all walls}\;(1).

\begin{Def}
A $\nu\_{b,w}$-Joyce-Song stable pair $(F,s)$ of class $(\mathsf v,\,n\gg0)$ consists of a coherent sheaf $F$ of class $\mathsf v\in K(X)$ and a section $s\in H^0(F(n))$ such that
\begin{itemize}
\item $F$ is $\nu\_{b,w}$-semistable and
\item $s\colon\cO(-n)\to F$ does not factor through any $\nu\_{b,w}$-destabilising subobject of $F$.
\end{itemize}
\end{Def}

In fact any $\nu\_{b,w}$-destabilising subobject $F_1\into F$ in $\cA_{\;b}$ is actually a \emph{subsheaf} of $F$ when $(b,w)$ is above $\ell_f$ \eqref{equation of final wall} --- in particular for $(b,w)$ on or above $\ell_{\js}$. Let $F_2:=F/F_1\in\cA_{\;b}$ and set
$r':=\rk\;(\cH^{-1}(F_2)),\ r_i:=\rk\;(F_i)$. By \eqref{estimate} applied to $F_2$ and \eqref{esti1} for $F_1$ we get
\beqa
0 &=& \ch_1(F).H^2\=\ch_1(F_2).H^2+\ch_1(F_1).H^2 \\ &\ge& r'(n-2\epsilon)H^3-r_2\epsilon H^3-r_1\epsilon H^3\=r'(n-2\epsilon)H^3-r\epsilon H^3,
\eeqa
with $\epsilon=1/4r^2H^3$. Thus $r'=0$, but $\cH^{-1}(F_2)$ is torsion-free by the definition of $\cA_{\;b}$. So $F_2$ is a sheaf and by \eqref{LES} this shows $F_1$ is a subsheaf of $F$.

Morphisms of pairs $(F_i,s_i)$ are given by maps $F_1\to F_2$ intertwining the $s_i$. The stability condition ensures that the only automorphism of a stable Joyce-Song pair is the identity. Thus the stack $\js_{\;b,w}(\mathsf v)$ of Joyce-Song stable pairs (representing the functor mapping base schemes $B$ to families of stable Joyce-Song pairs flat over $B$) is an algebraic space.

Let $\cM^{ss}_{b,w}(v_n)$ denote the moduli stack of $\nu\_{b,w}$-semistable objects in class $v_n$. Since $\nu\_{b,w}(v_n)<+\infty$ it is an algebraic stack of finite type by Theorem \ref{ftstack}.

Fix $(b,w_0)$ on $\ell_{\js}$ \eqref{ljsdef} and $(b,w)$ just above $\ell_{\js}$ such that none of the finite set of walls for $v_n$ or $\mathsf v$ passes between them.

\begin{Thm}\label{jsthm} Let $\C^*\curvearrowright\js_{\;b,w}(\mathsf v)$ be the trivial action. There is a morphism
\begin{eqnarray} \label{map}
\js_{\;b,w}(\mathsf v)/\C^*\times\Pic\_0(X) &\To& \cM^{ss}_{b,w}(v_n), \\
\((F,s),T\) &\Mapsto& E\,:=\,\cok(s)\otimes T, \nonumber
\end{eqnarray}
which is an isomorphism onto an open and closed substack of $\cM^{ss}_{b,w}(v_n)$.

The image is those $E\in\cM^{ss}_{b,w}(v_n)$ which admit a nonzero map to $T(-n)[1]$ for some $T\in\Pic\_0(X)$; we then recover $(F\otimes T,s)$ from its cone. Any such $E$ is strictly $\nu\_{b,w}$-stable.
\end{Thm}

\begin{proof}
Fix a Joyce-Song stable pair $(F,s)$. By case (1) of Theorem \ref{thm:all walls}, $F$ remains $\nu\_{b,w}$-semistable if we send $w\to+\infty$. Since its slope $\to-\infty$ while that of any rank 0 subsheaf remains constant, we see that $F$ must be torsion-free. In particular, $s$ is injective and $E:=\cok(s)$ is a sheaf.

Suppose $E_1\into E$ is a $\nu\_{b,w}$-stable subobject in $\cA_{\;b}$ with $\nu\_{b,w}(E_1)\ge\nu\_{b,w}(E)$. Composing with $E\to\cO(-n)[1]$ gives $E_1\to\cO(-n)[1]$ with $\nu\_{b,w_0}(E_1)\ge\nu\_{b,w_0}(\cO(n)[1])$. Suppose first this map is nonzero.  By the stability of $E_1$ it is therefore injective in $\cA_{\;b}$; denote its $\cA_{\;b}$-cokernel by $Q$. Then by the strict $\nu\_{b,w_0}$-stability of $\cO(-n)[1]$ \cite[Corollary 3.11(a)]{BMS} and the see-saw inequality, we find that both the denominator $\ch_1^{bH}\!.\;H^2$ and numerator $\ch_2\!.\;H-w\ch_0$ of $\nu\_{b,w}(Q)$ must vanish for $(b,w)$ along the wall. It follows that $\ch_H(Q)$ \eqref{chH} is zero, so $Q$ is a sheaf supported in dimension 0 \cite[proof of Lemma 3.2.1]{BMT}. This is impossible because any $\cA_{\;b}$-quotient of $\cO_X(-n)[1]$ has $\cH^{-1}(Q)\ne0$.

Thus $E_1\to\cO(-n)[1]$ is zero and $E_1\into E$ lifts to $E_1\into F$. But as $(b,w)$ is above $\ell_{\js}$,
$$
\nu\_{b,w}(\mathsf v)\,<\,\nu\_{b,w}(v_n)\,<\,\nu\_{b,w}(\cO(-n)[1]) \quad\so\quad \nu\_{b,w}(F)\,<\,\nu\_{b,w}(E)\,\le\,\nu\_{b,w}(E_1).
$$
This means $E_1$ strictly destabilises $F$, which is a contradiction.
Therefore $E$ is strictly $\nu\_{b,w}$-stable; in particular Aut$\;(E)=\C^*$. The same argument applies to $E\otimes T$ for $T\in\Pic\_0(X)$.\medskip

Conversely, start with a $\nu\_{b,w}$-semistable object $E$ admitting a nonzero map to $T(-n)[1]$. By the $\nu\_{b,w}$-stability of $T(-n)[1]$ and the same argument as above it is a surjection in $\cA_{\;b}$. This gives a sequence $F\into E\twoheadrightarrow T(-n)[1]$ in $\cA_{\;b}$ which $\nu\_{b,w_0}$-destabilises $E$, so by case (1) of Theorem \ref{thm:all walls} $F$ is a $\nu\_{b,w}$-semistable sheaf. The same is therefore true of $F\otimes T^{-1}$.

Suppose the induced $s\colon T(-n)\to F$ factors through a subsheaf $F_1\into F\twoheadrightarrow F_2$ with $\nu\_{b,w}(F_1)=\nu\_{b,w}(F)=\nu\_{b,w}(F_2)$. Then $F\twoheadrightarrow F_2$ factors through a map $E\twoheadrightarrow F_2$. But this contradicts $\nu\_{b,w}(E)>\nu\_{b,w}(F)=\nu\_{b,w}(F_2)$. Thus $(F\otimes T^{-1},s)$ is a Joyce-Song stable pair whose cokernel, tensored with $T$, is $E$. \medskip

Finally, given $T'\in\Pic\_0(X)$, apply $\Hom(\ \cdot\ ,T'(-n))$ to the sequence $T(-n)\to F\to E$. We find that $\Hom(E,T'(-n)[1])=H^0(T^*\otimes T')$ is $\C$ if $T'=T$ and zero otherwise. Thus $s$ and $T$ are uniquely determined by $E$, and \eqref{map} is a bijection to its image
$$
\big\{E\colon\Hom(E,T(-n)[1])\ne0\text{ for some }T\in\Pic\_0(X)\big\}.
$$
This is clearly closed. Working with the universal Joyce-Song pair and taking its cokernel upgrades the arrow \eqref{map} to a morphism of stacks. It remains to show it is a local isomorphism when we compose with the projection to the coarse moduli space $M^{ss}_{b,w}(v_n)$. When $r>1$ this follows from the observation that the deformation-obstruction theory of $\js_{b,w}(\mathsf v)$ in \cite[Proposition 12.14]{JS} is the same as the deformation-obstruction theory of $M^{ss}_{b,w}(v_n)$ in
\cite[Theorem 4.1]{HT} because the complex $\mathbb I=\{\cO(-n)\to F\}$ used in the former is quasi-isomorphic to $E$ used in the latter. When $r=1$ so $\rk(v_n)=0$ the truncations of \cite[Theorem 12.20]{JS} and \cite[Equation 4.10]{HT} are also isomorphic.
\end{proof}

\section{Rank 2 simplifications}\label{2}
In this section we consider the special case $r=2$, where we can strengthen Theorem \ref{thm:all walls} by showing there are no walls below $\ell_{\js}$ and that $\ell_{\js}$ is not a type 2 wall. First we recall the result in the $r=1$ case, which was proved in \cite{FT2} when the term $D$ in \eqref{class vn} is zero.

\begin{Thm}\label{ft2}
When $r=1$ the only wall for the class $v_n$ is $\ell_{\js}$. If $F$ is $\nu\_{b,w}$-semistable above $\ell_{\js}$ then there exists a unique $T\in\Pic\_0(X)$ and a unique nonzero map $F\to T(-n)[1]$ which $\nu\_{b,w}$-destabilises $F$ below $\ell_{\js}$.
\end{Thm}

\begin{proof}
Theorem \ref{thm:all walls} describes the walls of instability for $v_n$. We claim there are no walls of type 2. On such a wall the semistable factors would have rank 0 since $\rk(v_n)=r-1=0$. Thus their $\nu\_{b,w}$-slopes and that of $v_n$ are all constant in $(b,w)$, so there is no such wall.

We conclude that $\ell_{\js}$ is the only wall. By Lemma \ref{final wall} there are no $\nu\_{b,w}$-semistable objects of type $v_n$ in the nonempty region of $U$ below $\ell_f$. Moving $(b,w)$ upwards, the absence of walls means there are none below $\ell_{\js}$. On $\ell_{\js}$ Theorem \ref{thm:all walls}\;(1) shows any semistable objects $E$ admits a nonzero map to $T(-n)[1]$ for some $T\in\Pic\_0(X)$. Therefore the isomorphism of Theorem \ref{jsthm} to a connected component of moduli space is in fact an isomorphism to the whole moduli space: all semistable objects $E$ are cokernels of Joyce-Song pairs $(F,s)$ twisted by some $T\in\Pic\_0(X)$, with $F,s,T$ uniquely determined by $E$.
\end{proof}

The main result of this section is an analogous result when $r=2$. 

\begin{Thm}\label{thm: rank-2}
Fix $r=2$ and work in class $v_n$.
\begin{itemize}
\item  For $(b,w)$ below $\ell_{\js}$ there are no semistable objects: $\cM^{ss}_{b,w}(v_n)=\emptyset$.
\item For $(b,w)$ just above $\ell_{\js}$ the semistable objects are Joyce-Song pairs tensored with elements of $\Pic\_0(X)$. Thus they are all $\nu\_{b,w}$-stable and 
\eqref{map} is an isomorphism,
$$
\cM^{ss}_{b,w}(v_n)\ \cong\ \js_{\;b,w}(\mathsf v)/\C^*\times\Pic\_0(X).
$$
\end{itemize}
\end{Thm} 

\begin{proof}
Take a point $(b,w)\in U$ just above $\ell_{\js}$ and $E\in\cM^{ss}_{b,w}(v_n)$. By Lemma \ref{final wall} we can decrease $w$ until $(b,w)$ lies on a wall $\ell\cap U$ below which $E$ is strictly destabilised. Thus either $\ell=\ell_{\js}$ or $\ell$ is below $\ell_{\js}$; our first aim is to show the former.

We apply Theorem \ref{thm:all walls} to $(b,w)\in\ell\cap U$ and $E$. If we are in case (1) then by Theorem \ref{jsthm} we are done. So we suppose we are in case (2).

Since $r=2$, the rank of $E$ is 1 and the destabilising sheaves $E_1,\,E_2$ have ranks zero and one; we relabel and order them $E_0,\,E_1$ so that $\rk\;(E_i)=i$. We write
$$
\(\!\ch_0(E_1),\ \ch_1(E_1).H^2,\ \ch_2(E_1).H,\ \ch_3(E_1)\)\=\(1,\,cH^3,\,sH^3,\,dH^3\)
$$
for $c\in\frac1{H^3}\Z$ and $s,d\in\Q$.
By Theorem \ref{thm:all walls}, $c\in[0,n]$. By 2(b) of the same Theorem, $\ch_1(E_0).H^2>0$, so $c<n$. Moreover 
if $c=0$ then $E_1$ is tilt semistable and 
the class of the $\nu\_{b,w}$-semistable sheaf $E_0$ has the form of $v_n$ for rank 1 less.

Therefore we may apply Theorem \ref{ft2} to $E_0$ to conclude $(b,w)$ is above or on $E_0$'s Joyce-Song wall. Thus $\ell$ passes through or above $\Pi(\cO(-n)[1])$, as well as through $\Pi(E)$. Since by construction it cannot be above $\ell_{\js}$, we conclude that $\ell=\ell_{\js}$. Hence it passes through $\Pi(\cO(-n)[1])$ so it is also $E_0$'s Joyce-Song wall.

The destabilising sequence is either $E_0\into E \twoheadrightarrow E_1$ or $E_1\into E \twoheadrightarrow E_0$. Since $\nu\_{b,w}(E_0)$ is constant in $w$ while $\nu\_{b,w}(E)$ increases as $w$ decreases, the first sequence does not destabilise $E$ below the wall, so the destabilising sequence is $E_1\into E \twoheadrightarrow E_0$. By Theorem \ref{ft2} there is a nonzero map $E_0\rightarrow T(-n)[1]$ for some $T\in\Pic\_0(X)$. Composing with $E\twoheadrightarrow E_0$ gives a nonzero map $E\rightarrow T(-n)[1]$. Again by Theorem \ref{jsthm} we are now done. \medskip

So finally we assume $c\in(0,n)$. The line $\ell$ passing through
$$
\Pi(\vi_n)\=\left(n,\,-\tfrac{\beta.H}{H^3}-\tfrac{n^2}2\right)
\qquad\text{and}\qquad \Pi(\ch(E_1))\=(c,s)
$$
is 
\beq{ell}
 w(b)\=-\left(\frac{\beta.H}{H^3} +\frac{n^2}2 +s\right)\left(\frac{b-c}{n-c}\right)+s\,. 
\eeq
Since $\ell$ lies above $\ell_f$ by Lemma \ref{final wall}, and $-1\in(a_f,b_f)$ by \eqref{bounds on bf}, we know that $\ell$ is inside $U$ at $b=-1$. At this point $B_{-1,\,w(-1)}(E_1)\ge0$ by the Bogomolov-Gieseker conjecture \eqref{conjecture}. Using the form \eqref{boglinear} this becomes
$$
(c^2-2s)\left[\left(\tfrac{\beta.H}{H^3}+\tfrac{n^2}2+s\right)\tfrac{1+c}{n-c}+s\right]-(3d-cs)+2s^2-3cd\ \ge\ 0,
$$
giving an upper bound for $d$,
\begin{equation}\label{upper bound for d}
d\ \leq\ \frac{sc}{3} + \frac{c^2-2s}{n-c}\left(\frac{\beta.H}{H^3}+\frac{n^2}2+s\right)\!.
\end{equation}
Since the destabilising sheaf $E_0$ is of rank zero, Lemma \ref{lem:rank sero sheaevs} below gives
\begin{equation*}
    \frac{\ch_3(E_0)}{H^3}\ \leq\ \frac{(\ch_2(E_0).H)^2}{2H^3\ch_1(E_0).H^2} + \frac1{24} \left(\frac{\ch_1(E_0).H^2}{H^3}\right)^{\!3}.
\end{equation*}
Substituting $\ch(E_0)=\ch(E)-\ch(E_1)$ gives a lower bound for $d$,
\begin{equation}\label{lower bound for d}
-\frac{m}{H^3} + \frac{n^3}{6}-d\ \leq\ \frac1{2(n-c)}\left(\frac{n^2}{2}+s + \frac{\beta.H}{H^3}\right)^{\!2} + \frac{(n-c)^3}{24}\,.
\end{equation}
Combining \eqref{upper bound for d} and \eqref{lower bound for d} gives the inequality $s^2-\alpha_1s-\alpha_2\le0$, where 
\begin{align*}
    \alpha_1 &\ :=\ 2cn +n^2 +2\tfrac{\beta.H}{H^3}\,,\\
    \alpha_2 &\ :=\ \tfrac14c^4 -c^3n +c^2 \left( \tfrac{5}{2}n^2 +2 \tfrac{\beta.H}{H^3} \right) -c \tfrac{6m}{H^3} + 3n^2 \tfrac{\beta.H}{H^3} + n \tfrac{6m}{H^3} + 3 \left(\tfrac{\beta.H}{H^3}\right)^{\!2}.
\end{align*}
Thus $s$ lies between the roots of $s^2-\alpha_1s-\alpha_2=0$. The larger root is $>\frac{\alpha_1}2$ which is $>0$ since $n\gg0$. Since $c>0$, we see that for $n\gg0$,
$$
s_0\ :=\ -\left(\frac{n}{2} +\frac{\beta.H}{2nH^3}  \right) c -\frac{\beta.H}{2H^3}\ <\ 0.
$$
We will show that $s_0^2-\alpha_1s_0-\alpha_2>0$, so $s_0$ is smaller than the smaller root. Thus $s>s_0$, or equivalently
$$
-\frac1{n-c}\left(\frac{\beta.H}{H^3} +\frac{n^2}2 +s\right)\ <\ -\left(\frac n2+\frac{\beta.H}{2nH^3}\right)\!.
$$
The left hand side is the gradient of $\ell$ \eqref{ell}; the right hand side is the gradient of $\ell_{\js}$. Thus, to the left of their common point $\Pi(v_n)$,  $\ell$ lies strictly above $\ell_{\js}$, a contradiction.\medskip

So what remains is to compute $s_0^2-\alpha_1s_0-\alpha_2$. We get a quartic polynomial $f(c)$ in $c$,
\begin{multline*}
  -\tfrac14c^4 +nc^3 +\left[-\tfrac{5}{4}n^2 -\tfrac{\beta.H}{2H^3} + \tfrac{(\beta.H)^2}{4n^2(H^3)^2}\right]\!c^2+ \left[\tfrac{n^3}{2} +3n \tfrac{\beta.H}{H^3}  +\tfrac{3(\beta.H)^2}{2n(H^3)^2}  + \tfrac{6m}{H^3} \right]\!c\\
-\left[\tfrac{5\beta.H}{2H^3}n^2 + \tfrac{6m}{H^3}n + \tfrac{7(\beta.H)^2}{4(H^3)^2}\right]\!,
\end{multline*}
and we must show $f(c)>0$ for $c\in\big[\frac1{H^3},n-\frac1{H^3}\big]$ when $n\gg0$. To leading order in $n$ we can read off from the above formula,
\begin{itemize}
\item[(i)] $f\(\frac1{H^3}\)\sim\frac{n^3}{2H^3}>0\ $ and $\,f'\(\frac1{H^3}\)\sim\frac{n^3}2>0,$\vspace{1mm}
\item[(ii)] $f\(n-\frac{1}{H^3}\)\sim\frac{n^2}{4(H^3)^2}>0\ $ and $\,f'\(n-\frac{1}{H^3}\)\sim-\frac{n^2}{2H^3}<0$.
\end{itemize}
Moreover as $n\to\infty$ the coefficients of the powers of $c$ in $f'(c)$ tend to
$$
f'(c)\ \sim\ -c^3+3nc^2-\tfrac52n^2c+\tfrac12n^3\=(n-c)\(\tfrac12n^2-2nc+c^2\)
$$
with roots at $n$ and $n\(1\pm\surd\tfrac12\)$. Since the coefficients are symmetric polynomials in the roots this shows the roots of $f'$ tend to $n$ and $n\(1\pm\surd\tfrac12)$. In particular, for $n\gg0$, this and (i),\,(ii) show that precisely one of the roots of $f'$ is in $\left[\frac{1}{H^3}, n-\frac{1}{H^3}\right]$ and $f(c) >0$ for any $c \in \left[\frac{1}{H^3} , n-\frac{1}{H^3}\right]$, as required.
\end{proof}

\begin{Lem}\label{lem:rank sero sheaevs}
If $F$ is a $\nu\_{b,w}$-semistable rank zero sheaf for some $(b,w)\in U$ then 
\begin{equation*}
    \ch_3(F)\ \leq\ \frac{(\ch_2(F).H)^2}{2\ch_1(F).H^2} + \frac{H^3}{24} \left(\frac{\ch_1(F).H^2}{H^3}\right)^{\!3}\!.
\end{equation*}
\end{Lem}
\begin{proof}
For objects of rank 0 the slope $\nu\_{b,w}$ reduces to the ordinary 2-dimensional slope $\nu\_H$,
$$
\nu\_H(F)\ :=\ \left\{\!\!\!\begin{array}{cc} \frac{\ch_2(F).H}{\ch_1(F).H^2} & \text{if }\ch_1(F).H^2\ne0, \\
+\infty & \text{if }\ch_1(F).H^2=0. \end{array}\right.
$$
It follows that a $\nu\_{b,w}$-semistable rank 0 sheaf is automatically $\nu\_H$-semistable. Conversely, by Proposition \ref{prop. general-sheaf}, a rank 0 sheaf $F$ can only be $\nu\_{b,w}$-destabilised by other rank 0 sheaves so long as we take $(b,w)$ inside its safe area $U_v$, where $v=\ch(F)$. It follows that $F$ is $\nu\_{b,w}$-semistable for $(b,w)$ above the line $\ell_v$ \eqref{l-v}, which we claim has equation
\beq{lyne}
    w \= \frac{\ch_2(F).H}{\ch_1(F).H^2}\  b + \frac{1}{8}\left(\frac{\ch_1(F).H^2}{H^3}\right)^{\!2} - \frac{1}{2}\left(\frac{\ch_2(F).H}{\ch_1(F).H^2}\right)^{\!2}\!.
\eeq
Both lines have slope $\frac{\ch_2(F).H}{\ch_1(F).H^2}$. Substituting $w=\frac12b^2$ into the left hand side of \eqref{lyne} gives a quadratic equation $b^2+\alpha_1b+\alpha_2=0$ whose roots $a_v<b_v$ satisfy
$$
(b_v-a_v)^2\=(a_v+b_v)^2-4a_vb_v\=\alpha_1^2-4\alpha_2.
$$
Thus we can calculate $b_v-a_v=\ch_1(F).H^2/H^3$, which by \eqref{l-v} means \eqref{lyne} is indeed $\ell_v$.

Therefore $F$ is $\nu\_{b,w}$-semistable on $\ell_v\cap U$, and in particular at the point
$$
b\=\nu\_H(F)\,,\qquad w=\frac12\nu\_H(F)^2+\frac{(\ch_1(F).H^2)^2}{8(H^3)^2}\,.
$$
Applying the Bogomolov-Gieseker inequality \eqref{quadratic form} at this point gives 
\begin{equation*}
    \ch_3(F)\ \leq\ \frac{(\ch_2(F).H)^2}{2\ch_1(F).H^2} + \frac{H^3}{24} \left(\frac{\ch_1(F)}{H^3}\right)^{\!3}\!.\qedhere
\end{equation*}
\end{proof}

\section{Moduli stacks}\label{TomB}
Piyaratne and Toda proved that Bridgeland semistable objects of a fixed class $v\in K_H(X)$ and fixed $\ch_3$ form an algebraic stack of finite type \cite[Theorem 4.2]{PiTo}.\footnote{They also prove the stack is \emph{quasi-proper}: it satisfies the valuative criterion for properness without the separatedness condition.}
Here we use their result to deduce the same for stacks of $\nu\_{b,w}$-semistable objects of slope $\nu\_{b,w}<+\infty$ by finding a Bridgeland stability condition which is equivalent to $\nu\_{b,w}$-stability on objects of class $v\in K_H(X)$. We thank Yukinobu Toda for his help with this section.

\begin{Prop}\label{prop.reducing to a good one}
	Fix $(b,w)\in U$ and $v\in K_H(X)$ with $\Pi(v)\not\in U$ and $\ch_1^{bH}(v).\;H^2>0$. Then without affecting the (semi)stability of objects of class $v$ we may assume $(b,w)$ satisfies
	\beq{parabw}
	w\ch_0H^3 \= b^2\ch_0H^3 - b\ch_1\!.\;H^2 + \ch_2\!.\;H, \quad\text{where }\ch_i\,:=\,\ch_i(v).
	\eeq
\end{Prop}
\begin{proof}
	If $\ch_0 = 0$ we are required to show we may take $b$ to equal $m:=\frac{\ch_2\!.\;H}{\ch_1\!.\;H^2}$. Let $\ell$ be the line of slope $m$ through $(b,w)$. By Proposition \ref{prop. locally finite set of walls} the (semi)stability of objects $E$ of class $v$ is unchanged if we move $(b,w)$ along $\ell\cap U$. So we move it to the intersection of $\ell$ with the vertical line $b= m$, which lies inside $U$: both $\partial U$ and $\ell$ have equal gradients $m$ at $b=m$, so $\ell$ still lies above $\partial U$ at this point.
		
	Suppose now that $\ch_0 \neq 0$. Recall the $H$-discriminant $\Delta_H = \left(\ch_1\!.\;H^2\right)^2-2(\ch_2\!.\;H)\ch_0H^3$ of \eqref{BOG} and consider the points 
	\begin{equation*}
	p_\pm\ :=\ \left(\frac{\ch_1\!.\;H^2}{\ch_0H^3} \pm  \frac{\sqrt{\Delta_H}}{\ch_0H^3}\, , \  \frac{1}{2}\! \left(\frac{\ch_1\!.\;H^2}{\ch_0H^3} \pm \frac{\sqrt{\Delta_H}}{\ch_0H^3} \right)^{\!2\,}  \right)
	\end{equation*}
whose lines $\ell_\pm$ through $\Pi(v)$ are tangent to $\partial U$ as in Figure \ref{parab}.
 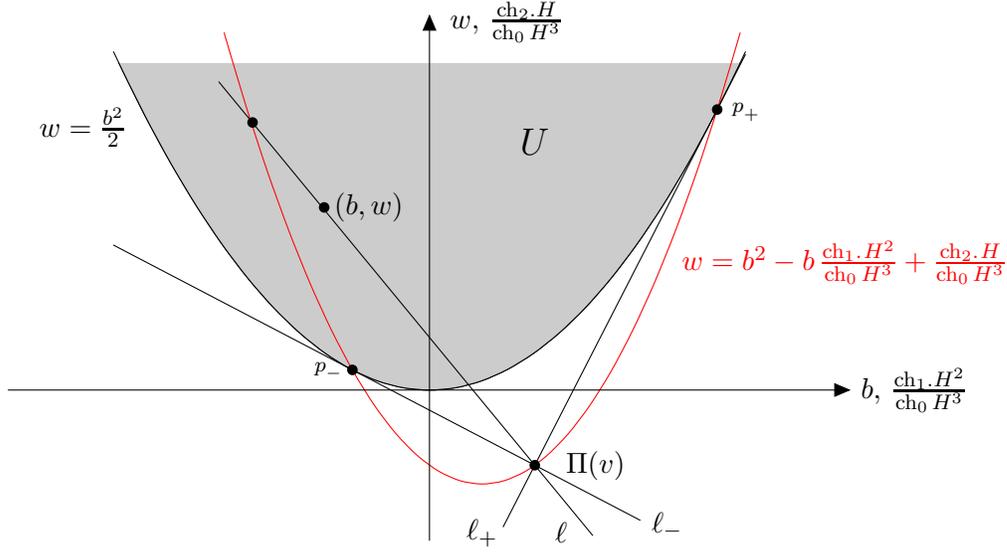
\begin{figure}[h]
	\begin{centering}
		\definecolor{zzttqq}{rgb}{0.27,0.27,0.27}
		\definecolor{qqqqff}{rgb}{0.33,0.33,0.33}
		\definecolor{uququq}{rgb}{0.25,0.25,0.25}
		\definecolor{xdxdff}{rgb}{0.66,0.66,0.66}
		
		\begin{tikzpicture}[line cap=round,line join=round,>=triangle 45,x=1.4cm,y=1.0cm]
		
		\draw  plot[smooth,domain=-3:3] (\x, {\x*\x/2});
		\draw  (-3.8, 3.5) node [right] {$w= \frac{b^2}{2}$};
		\fill[gray!40!white, draw=black] plot[smooth,domain=-2.95:2.95] (\x, {\x*\x/2});	
		\draw  (1, 3) node [above] {\Large $U$};
				
		\draw[red]  plot[smooth,domain=-1.95:2.95] (\x, {\x*\x-\x-1});
		\draw  (2.3,1.7) node [right] [color=red]{$w= b^2 - b\,\frac{\ch_1\!.\;H^2}{\ch_0H^3} + \frac{\ch_2\!.\;H}{\ch_0H^3}$};

		\draw[->,color=black] (-4,0) -- (4,0);
		\draw  (4, 0) node [right ] {$b,\,\frac{\ch_1\!.H^2}{\ch_0H^3}$};
		\draw[->,color=black] (0,-2) -- (0,5);
		\draw  (0.1, 4.9) node [right] {$w,\,\frac{\ch_2\!.H}{\ch_0H^3}$};
			
		\draw  plot[smooth,domain=-3:2] (\x, {-1-0.732*(\x-1)});
		\draw  (2.5,-1.8) node[left] {$\ell_-$};
		\draw  plot[smooth,domain=0.7:3] (\x, {-1+2.732*(\x-1)});
		\draw  (.5, -1.9) node {$\ell_+$};
		\draw  plot[smooth,domain=-2:1.55] (\x, {.7-1.7*\x});
		\draw  (1.25, -1.9) node {$\ell$};
		
		\fill (1, -1) circle (2pt);
		\draw  (1.2, -1) node [right] {$\Pi(v)$};
		\fill (-1.68, 3.56) circle (2pt);
        \fill (-1, 2.43) circle (2pt);
		\draw  (-1, 2.43) node [right] {$(b,w)$};

		\begin{scriptsize}
		\fill (2.732, 3.732) circle (2pt);
		\draw  (2.8, 3.7) node [right] {$p\_+$};
		\fill (-.732, .268) circle (2pt);
		\draw  (-.75, .25) node [left] {$p\_-$};				
		\end{scriptsize}
		
		\end{tikzpicture}
		
		\caption{Finding the new $(b,w)$ when $\ch_0>0$}
		\label{parab}
		
	\end{centering}
\end{figure}	

The parabola through $\Pi(v)$ and $p_\pm$ is precisely \eqref{parabw},
	$$
	w \= b^2 - b\,\frac{\ch_1\!.\;H^2}{\ch_0H^3} + \frac{\ch_2\!.\;H}{\ch_0H^3}\,.
	$$
The line $\ell$ passing through $\Pi(v)$ and $(b,w)\in U$ lies between $\ell_\pm$. It therefore intersects the parabola \eqref{parabw} at a point which is
\begin{itemize}
\item northwest of $p_-$ if $\ch_0>0$ so that $E\in\cA_{\;b}$ implies $(b,w)$ is to the left of $\Pi(v)$, or
\item  northeast of $p_+$ if $\ch_0<0$ so $(b,w)$ is to the right of $\Pi(v)$.
\end{itemize}
(Here $E$ is an object of class $v$.) Thus this point is inside $U$, and we may move $(b,w)$ to it without affecting the (semi)stability of $E$ by Proposition \ref{prop. locally finite set of walls}.
\end{proof}

Given $(b, \alpha) \in \mathbb{R} \times \mathbb{R}^{>0}$, let $Z_{b, \alpha} \colon K(X) \rightarrow \mathbb{C}$ be the group homomorphism
\begin{equation*}
Z_{b, \alpha}(E)\ :=\ \left(-\ch_3^{bH}(E) + \tfrac3{10}\alpha^2\ch_1^{bH}(E).H^2 \right) + i \left( \ch_2^{bH}(E).H - \tfrac12\alpha^2\ch_0(E)H^3   \right). 
\end{equation*}
Although $Z_{b, \alpha}$ does not factor through $K_H(X)$, the following slope function $\widetilde\nu\_{b,\alpha}$ on objects $E$ of $\cA_{\;b}$ does,
$$
\widetilde\nu\_{b,\alpha}(E)\ =\ \left\{\!\!\begin{array}{cc} \frac{\text{Im}\,Z_{b, \alpha}(E)}{\ch_1^{bH}(E).H^2}
& \text{if }\ch_1^{bH}(E).H^2\ne0, \\
+\infty & \text{if }\ch_1^{bH}(E).H^2=0. \end{array}\right.
$$
Define the heart $\cA_{\;b,\alpha}:=\big\langle\mathcal{F}_{b, \alpha}[1],\,\mathcal{T}_{b, \alpha}\big\rangle\subset \cD(X)$, where 
\begin{equation*}
 \mathcal{T}_{b, \alpha}\ :=\ \big\{E \in \cA_{\;b}\ \colon\ \widetilde\nu_{b, \alpha}^{\,-} > 0  \big\}, \quad\qquad  \mathcal{F}_{b, \alpha}\ :=\ \big\{E \in \cA_{\;b}\ \colon\ \widetilde\nu_{b, \alpha}^{\;+} \leq 0  \big\}. 
\end{equation*} 

\begin{Thm}\cite[Theorem 8.6 and Proposition 8.10]{BMS}
	For any $(b, \alpha) \in \mathbb{R} \times \mathbb{R}^{>0}$, the pair $\sigma\_{b, \alpha} \coloneqq \(\cA_{\;b,\alpha},\,Z_{b, \alpha}\)$ defines a Bridgeland stability condition on $\cD(X)$.
\end{Thm}

Taking $(b,w)$ given to us by Proposition \ref{prop.reducing to a good one} and 
setting $\alpha^2=2\(w-\frac12b^2\)>0$ we note
\begin{itemize}
\item $\widetilde\nu\_{b,\alpha} = \nu\_{b,w} +b$, so $\widetilde{\nu}\_{b,\alpha}$-(semi)stability is the same as $\nu\_{b,w}$-(semi)stability, and
\item $\text{Im}\,Z_{b, \alpha}(E)=0=\widetilde\nu\_{b,\alpha}(E)$ for $E$ of class $v\in K_H(X)$.
\end{itemize}

\begin{Lem}\label{lem. reduce to stability condition}
	Fix a class $v\in K_H(X)$ with $\ch_1^{bH}(v).\;H^2>0$ and $\mathrm{Im}\,Z_{b, \alpha}(v) = 0$. An object $E \in \cA_{\;b}$ of class $v$ is $\widetilde{\nu}\_{b,\alpha}$-semistable if and only of it is $\sigma\_{b, \alpha}$-semistable.    
\end{Lem}
\begin{proof}
	First assume $E \in \cA_{\;b}$ is $\widetilde{\nu}\_{b,\alpha}$-semistable. Then $\widetilde\nu\_{b,w}(E) = 0$ implies $E[1] \in \cF_{b,\alpha}[1]\subset\cA_{\;b,\alpha}$ with central charge $Z_{b,\alpha}\in(-\infty,0)$ on the boundary of the closure of the upper half plane $\overline{\mathbb H}\supset Z_{b,\alpha}\(\cA_{\;b,\alpha}\)$. Thus all sub- and quotient objects of $E$ in $\cA_{\;b,\alpha}$ must also have $Z_{b,\alpha}\in(-\infty,0)\subset\overline{\mathbb H}$, so $E$ is $\sigma\_{b,\alpha}$-semistable.
	
	Conversely, if $E$ is $\sigma\_{b,\alpha}$-semistable then $E \in \cA(b) \ \cap\ \cA_{b, \alpha}[k]$ for some $k \in\Z$. This inter\-section is $\cT_{b, \alpha}$ when $k=0$ and $\cF_{b, \alpha}$ when $k=-1$; otherwise it is empty. But Im$\,Z_{b, \alpha}(v) = 0$ and $\ch_1^{bH}(v).\;H^2>0$, so we deduce $E \in \cF_{b, \alpha}$. By the definition of $\cF_{b, \alpha}$,
	\begin{equation*}
	\widetilde{\nu}_{b, \alpha}^{\;+}\(E\)\ \leq\ 0 \= \widetilde{\nu}\_{b, \alpha}\(E\),
	\end{equation*} 
so $E$ is $\widetilde{\nu}\_{b,\alpha}$-semistable.  
\end{proof}

\begin{Prop}\cite{TodaK3}\label{prop:pi-to-heart}
	Given a scheme $S$ of finite type and $\mathcal{E} \in \cD(X \times S)$, the subset of points $s \in S$ for which $\mathcal{E}_s \in \cA_{\;b}$ is open.   
\end{Prop}
\begin{proof}
	By an argument similar to \cite[Lemma 3.6]{TodaK3} we may assume $S$ is a smooth quasi-projective variety. Suppose there is a point $s \in S$ such that $\mathcal{E}_s \in \cA_{\;b}$. By \cite[Theorem 3.3.2]{AP} there is an open neighbourhood $s \in U \subset S$ such that $\mathcal{E}|_U$ lies in
	\begin{equation*}
	\cA_{\;U}\ :=\ \big\{F \in \cD(X \times U) \ \colon\ Rp_*(F\otimes L^k) \in \cA_{\;b} \,\ \text{for all $k \gg 0$}  \big\}.
	\end{equation*}
Here $L$ is the pull back of an ample line bundle from $U$ and $p$ is the projection $X\times U\to X$. Then \cite[Lemma 4.7 and Remark 3.11]{TodaK3} implies that the set of points $s' \in U$ for which $\mathcal{E}_{s'} \in \cA_{\;b}$ is indeed open in $S$.
\end{proof}

\begin{Thm}\label{ftstack}
Fix $v \in K_H(X),\ m\in\Z$ and $(b, w) \in U$ such that $\ch_1^{bH}(v).\;H^2>0$. Then the union over all $\alpha\in K(X)$ with $[\alpha]=v\in K_H(X)$ and $\ch_3(\alpha)=m$
\beq{unyun}
\bigcup\nolimits_\alpha\cM^{ss}_{b,w}(\alpha)
\eeq
is an algebraic stack of finite type.
\end{Thm}

\begin{proof}
By \cite[Theorem 4.2]{PiTo} the $\sigma_{b, \alpha}$-semistable objects $E \in \cA_{\;b, \alpha}$ with fixed
$$
\(\!\ch_0(E),\ \ch_1(E).H^2,\ \ch_2(E).H,\ \ch_3(E)\)\=\(\!\ch_H(E),\ \ch_3(E)\)\ \in\ \Q^4
$$
form an algebraic stack of finite type. Intersecting with the objects in $\cA_{\;b}$ gives the union \eqref{unyun} by Lemma \ref{lem. reduce to stability condition}. By Proposition \ref{prop:pi-to-heart} this is an open substack.
\end{proof}

\bibliographystyle{halphanum}
\bibliography{references}

\bigskip \noindent
{\tt{s.feyzbakhsh@imperial.ac.uk\\ richard.thomas@imperial.ac.uk}}


\end{document}